\documentclass[bj,numbers]{imsart-modif}

\usepackage{natbib}
\usepackage[colorlinks,citecolor=blue,urlcolor=blue]{hyperref}
\usepackage{amsthm,amsmath,amsfonts,amssymb}
\usepackage{graphicx}
\usepackage{url}
\usepackage[utf8]{inputenc} 
\usepackage{dsfont}
\usepackage{caption} 
\usepackage{subcaption}
\usepackage[ruled,vlined]{algorithm2e}
\usepackage{comment}
\usepackage{xspace}
\usepackage{bbm}
\usepackage{float}
\usepackage{mathtools}

\startlocaldefs
\theoremstyle{plain}
\newtheorem{theorem}{Theorem}[section]

\newtheorem{proposition}{Proposition}[section]
\newtheorem{lemma}{Lemma}[section]
\theoremstyle{remark}

\newtheorem{assumption}{Assumption}[section]
\newtheorem{example}{Example}[section]
\newtheorem{remark}{Remark}[section]

\newcommand{\ud}{\mathrm{d}}

\newcommand{\Mz}{\mathcal{M}_0}
\newcommand{\norm}[1]{\lVert{#1}\rVert}
\newcommand{\inpr}[1]{\langle{#1}\rangle}
\newcommand{\borel}{\mathcal{B}}

\DeclareMathOperator{\argmin}{argmin}

\def\point{\,\cdot\,}

\newcommand{\RR}{\mathbb{R}}

\newcommand{\EE}{\operatorname{\mathbb{E}}} 

\newcommand{\Prob}[1]{\operatorname{\mathbb{P}}[{#1}]}

\def\argmin{\mathop{\rm arg\, min}}

\def\A{{\cal A}}

\def\F{{\cal F}}

\def\sphere{\mathbb{S}} 
\def\rset{\mathbb{R}}

\def\sphere{\mathbb{S}} 
\def\rset{\mathbb{R}}

\renewcommand{\ge}{\geqslant}
\renewcommand{\geq}{\ge}
\renewcommand{\le}{\leqslant}
\renewcommand{\leq}{\le}
\newcommand{\reals}{\rset}
\newcommand{\abs}[1]{\lvert{#1}\rvert}

\newcommand{\un}{\mathds{1}}
\newcommand{\1}{\un}

\newcommand{\Oh}{\operatorname{O}}
\newcommand{\eps}{\varepsilon}
\newcommand{\ball}{\mathbb{B}}

\newcommand{\diff}{\mathrm{d}}
\newcommand{\hPhi}{\widehat{\Phi}}

\newcommand{\hmu}{\widehat{\mu}}
\newcommand{\hF}{\widehat{F}}
\newcommand{\hV}{\widehat{V}}

\newcommand{\cA}{\mathcal{A}}
\newcommand{\cF}{\mathcal{F}}
\newcommand{\cG}{\mathcal{G}}
\newcommand{\cX}{\mathcal{X}}
\newcommand{\event}{\mathcal{E}}
\def\cone{\mathcal{C}}

\def\un{\mathbbm{1}}
\def\cone{\mathcal{C}}
\newcommand{\PP}{\Prob}
\def\loi{P}

\DeclareMathOperator{\leb}{leb}

\newcommand{\given}{\mid}

\def\Vf{v}

\def\hVf{\hat{{v}}}
\def\id{\mathrm{id}}

\newcommand{\cst}{60}

\endlocaldefs

\begin{document}
\definecolor{gray}{rgb}{.5,.5,.5}
\definecolor{chgcol}{rgb}{0,0,0}
\newcommand{\chg}[1]{\textcolor{chgcol}{#1}}
\newcommand{\js}[1]{\textcolor{magenta}{\small\sffamily [#1]}}

\begin{frontmatter}
\title{Concentration bounds for the empirical angular measure
	with statistical learning applications}
\runtitle{Concentration bounds for the empirical angular measure}

\begin{aug}
\author[A]{\fnms{St\'ephan}~\snm{Cl\'emen\c{c}on}\ead[label=e1]{stephan.clemencon@telecom-paris.fr}}
\author[A]{\fnms{Hamid}~\snm{Jalalzai}\ead[label=e2]{hamid.jalalzai@telecom-paris.fr}}
\author[B]{\fnms{St\'ephane}~\snm{Lhaut}\ead[label=e3]{stephane.lhaut@uclouvain.be}}
\author[C]{\fnms{Anne}~\snm{Sabourin}\ead[label=e4]{anne.sabourin@u-paris.fr}}
\author[B]{\fnms{Johan}~\snm{Segers}\ead[label=e5]{johan.segers@uclouvain.be}}

\address[A]{LTCI, T\'el\'ecom Paris, Palaiseau , France. \\ \printead{e1,e2}}
\address[B]{LIDAM/ISBA, UCLouvain, Louvain-la-Neuve, Belgium. \\		\printead{e3,e5}}
\address[C]{ Université Paris Cité, CNRS, MAP5, F-75006 Paris, France.\\		\printead{e4}}

\end{aug}

\begin{abstract}
	The angular measure on the unit sphere characterizes the first-order dependence structure of the components of a random vector in extreme regions and is defined in terms of standardized margins. Its statistical recovery is an important step in learning problems involving observations far away from the center. In the common situation that the components of the vector have different distributions, the rank transformation offers a convenient and robust way of standardizing data in order to build an empirical version of the angular measure based on the most extreme observations. However, the study of the sampling distribution of the resulting empirical angular measure is challenging. It is the purpose of the paper to establish finite-sample bounds for the maximal deviations between the empirical and true angular measures, uniformly over classes of Borel sets of controlled combinatorial complexity. The bounds are valid with high probability and, up to logarithmic factors, scale as the square root of the effective sample size. The bounds are applied to provide performance guarantees for two statistical learning procedures tailored to extreme regions of the input space and built upon the empirical angular measure: binary classification in extreme regions through empirical risk minimization and unsupervised anomaly detection through minimum-volume sets of the sphere.
\end{abstract}



\end{frontmatter}


\section{Learning from multivariate extremes}
\label{sec:intro}

Estimation and prediction problems regarding the extremal behaviour of a \chg{$d$-dimensional} random vector $X$ are of key importance for risk assessment in finance, insurance, engineering and environmental sciences and, recently, for the analysis of weak signals in machine learning. A standard assumption to model the joint upper tail of $X$ is that its distribution lies in the maximal domain of attraction of a multivariate extreme value distribution. This assumption comprises two parts: 
\begin{enumerate}[label=(\roman*)]
	\item
	the marginal distributions of $X$ belong to the maximal domains of attraction of some univariate extreme value distributions; 
	\item
	after marginal transformation of $X$ through the probability integral transform or a variation thereof, the joint distribution of the transformed vector belongs to the maximal domain of attraction of a multivariate extreme value distribution with pre-specified margins. 
\end{enumerate}

Under the side assumption that the marginal distributions of $X$ are continuous, point~(ii) above only involves the copula of $X$. Point~(ii) can be imposed on its own, that is, independently of the assumptions on the margins in point~(i), and this is what we will do in this paper. 
\smallskip

\noindent
\textbf{The angular measure for multivariate extremes.} \quad
The multivariate extreme value distribution in point~(ii) is determined by a finite Borel measure, $\Phi$, with support contained in the intersection of $[0, \infty)^d$ with the unit sphere on $\reals^d$ with respect to some norm. This so-called angular measure, originally called spectral measure in \cite{de1977limit}, describes the first-order dependence structure of joint extremes of $X$ and is rooted in the theory of multivariate regular variation \citep{resnick:1987}. Since then, it has been recognized that the modelling of extremal dependence may require finer assumptions than the traditional maximal domain of attraction condition, leading for instance to conditional extreme value models and the theory of hidden regular variation; see for instance \cite{das+m+r:2013, wadsworth+t+d+e:2014} and the references therein. Here, we focus on inference on the angular measure $\Phi_p$ with respect to some $L_p$-norm, for $p \in [1,\infty]$. 

Inference on the angular measure is an important problem in extreme value analysis. It plays a part in the construction of confidence intervals for rare event probabilities \citep{dehaan+r:1998, dehaan+s:1999}. It lies at the basis for a test of the hypothesis that a bivariate distribution is in the maximum domain of attraction of an extreme value distribution \citep{einmahl+dh+l:2006}. It serves to model the action of a covariate on the extremal dependence of a baseline distribution through a density ratio model \citep{deCarvalho+d:2014, castrocamilo+dC:2017}. The angular density is also at the basis of an estimator of bivariate tail quantile regions \citep{einmahl+dh+k:2013}. Bounds for probabilities of joint excesses over high thresholds that are robust against misspecification of the angular measure are derived in \cite{engelke+i:2017}. The angular measure underlies techniques to find groups of variables exhibiting extremal dependence, with dimension reduction and sparse representations as objectives \citep{chautru2015,meyer+w:2019,lehtomaa+r:2020}; see \cite{engelke+i:2020} for a review. In \cite{janssen+w:2020}, the spherical $k$-means algorithm applied to a sample from the angular measure yields prototypes of extremal dependence. 

Furthermore, the angular measure is helpful for solving supervised and unsupervised learning tasks for sample points far away from the center of the distribution. In the spirit of principal component analysis, the eigendecomposition of the Gram matrix of the angular measure yields low-dimensional summaries of extremal dependence \citep{cooley+t:2018, drees+s:2019}. Anomalous data can be detected from unusual combinations of variables being large simultaneously \citep{goixJMVA} or from their lack of membership of minimum-volume sets of the unit sphere containing a large fraction of the total mass of the angular measure \citep{thomas17a}. Binary classification in extreme regions can be performed on the basis of the differences between the intra-class angular measures of the explanatory variables \citep{jalalzai2018binary}.
\smallskip

\noindent
\textbf{The empirical angular measure.} \quad
For a given dimension $d$, the collection of all angular measures is subject only to some moment constraints stemming from the marginal standardization of $X$ but does not form a parametric family. The usual considerations in favor of and against the use of parametric versus non-parametric methods therefore apply. Many parametric models have been proposed \citep{coles+t:1991, cooley+d+n:2010} and new ones continue to be invented, with a growing emphasis on the use of flexible models almost of a semi-parametric nature \citep{boldi+d:2007, sabourin+n:2014, sabourin:2015}. 

Our focus is on non-parametric estimation and inference via the empirical angular measure, $\hPhi_p$, introduced in~\cite{einmahl2001nonparametric} and generalized to every $L_p$-norms in~\cite{einmahl2009maximum}, although already alluded to some years earlier \cite{dehaan+r:1998,dehaan+s:1999}. Given a random sample from an unknown distribution, the marginal standardization mentioned in point~(ii) above is done by means of the empirical cumulative distribution functions. As a result, the estimator depends on the data only through the ranks.
On the one hand, the use of ranks makes the method invariant with respect to marginal scales and reduces sensitivity to outliers. On the other hand, the additional dependence stemming from the ranks greatly complicates the study of the sampling distribution of the estimator.

In fact, the asymptotic distribution of the empirical angular measure is known only in the bivariate case \citep{einmahl2001nonparametric, einmahl2009maximum}. In higher dimensions, only its consistency has been established \citep[Proof of Proposition~3.3]{janssen+w:2020}. This situation stands in contrast with the rank-based estimator of the stable tail dependence function, the asymptotic normality of which is known in any dimension \citep{einmahl2012mestim, bucher2014when}. The reason why the treatment of the empirical angular measure is so much more difficult is that it involves the empirical copula evaluated at sets of which the boundaries are not parallel to the coordinate axes. As a consequence, the usual argument to deal with the marginal empirical distribution functions via the functional delta method breaks down. Even outside the extreme value context, the asymptotic normality of the empirical copula in even a single non-rectangular set is to this day an open problem. 

Variations on the empirical angular measure enforce the aforementioned moment constraints through empirical or Euclidean likelihoods \citep{einmahl2009maximum, deCarvalho+o+s+w:2013} as well as a folding technique \citep{guillou+n+y:2015}. Other variations exploit more specific assumptions on the marginal distributions to estimate them differently than by the empirical distribution functions, for instance by fitting generalized Pareto distributions to the univariate tails \citep{einmahl+dh+s:1997}. Sometimes, asymptotic results are established as if the marginal distributions are known. In our paper, we focus on the original, rank-based empirical angular measure and we make no assumptions on the unknown marginal cumulative distribution functions except for their continuity.
\smallskip

\bgroup
\color{chgcol}
\noindent
\textbf{Concentration inequalities.} \quad
It is the main goal of this paper to carry out a non-asymptotic analysis of the empirical angular measure $\hPhi_p$, which, to the best knowledge, is the first of its kind. The concentration inequality established in our main result, Theorem~\ref{thm:conc}, states that with a certain large probability $1-\delta$ the estimation error is not larger than a certain small quantity depending, among other things, on $\delta$ and on the number, $k$, of sample points used in the definition of the estimator. The bound concerns the worst-case estimation error 
\begin{equation}
	\label{eq:max_dev}
	\sup_{A\in \mathcal{A}} \left| \hPhi_p(A)-\Phi_p(A) \right|
\end{equation}
over classes $\mathcal{A}$ of Borel subsets $A$ of the unit sphere satisfying certain properties. In particular, the complexity of $\mathcal{A}$ comes into play via the Vapnik--Chervonenkis dimension of a collection of sets constructed from $\mathcal{A}$. 

The result relies on two tools: first, the use of framing sets to capture certain random sets, as in \cite{einmahl2001nonparametric} and \cite{einmahl2009maximum}, although the framing sets are defined in a different manner here (see Section~\ref{sec:framing}); and second, a general concentration inequality for empirical processes indexed by rare events given in Theorem~\ref{thm:chaining}. The latter result is inspired by a similar inequality in \cite{GoixCOLT} but the difference is that now the constant in the bound is explicit. 

Although the angular measure $\Phi$ can be studied with respect to any norm on $\reals^d$, this study is limited to the ones most frequently used in practice, the $L_p$-norms for $1 \leq p \leq \infty$, i.e., for $x \in \reals^d$, 
\begin{equation}
	\label{eq: lpNorms}
	\norm{x}_p = 
	\begin{dcases}
		\left(|x_1|^p+\cdots+|x_d|^p\right)^{1/p} &\text{ if } 1 \le p < \infty, \\
		\max\left(|x_1|,\ldots,|x_d|\right) &\text{ if } p = \infty.
	\end{dcases}	 
\end{equation}

\smallskip
\egroup
\noindent
\textbf{Applications to statistical learning.} \quad
The concentration bounds for the empirical angular measure are leveraged to study two statistical learning problems: minimum-volume set estimation and binary classification in extreme regions. These two problems and the methods proposed to solve them were introduced in the conference papers~\cite{thomas17a} and~\cite{jalalzai2018binary}, respectively. In both cases, the method was based upon the empirical angular measure. However, the technical difficulties inherent to the rank transformation were ignored and the theoretical analysis was performed as if the marginal distributions are known. Here, we apply concentration inequalities for the empirical angular measure to obtain finite-sample performance guarantees for the rank-based methods.

We briefly describe the two learning problems and the role of the angular measure. For unsupervised anomaly detection, the learning task can be formulated as minimum-volume set estimation on the sphere, that is, the statistical recovery of a Borel set of the sphere with minimum volume but containing a given, large fraction of the total mass of the angular measure~\citep{thomas17a}. Any large observation whose angle lies outside the minimum-volume set is considered as a potential anomaly. Concentration inequalities for the uniform estimation error over a class of candidate sets enable us in Section~\ref{sec:appliMVsets} to get statistical guarantees on the minimum-volume set selected within the class on the basis of an estimator of the angular measure.

The second problem concerns binary classification in extreme regions. Empirical risk minimization, which is the main paradigm of statistical learning theory, tends to ignore the predictive performance of candidate classifiers in low-density regions of the input space. Instead, we focus on the probability of classification error in such extreme regions. By controlling the fluctuations of the empirical angular measure, we establish generalization bounds for classifiers obtained by minimizing the empirical classification error based on the most extreme input observations only. In Theorem~\ref{theo:deviationsClassif}, we state a bound on the supremum of the empirical risk over a class of candidate classifiers.
\smallskip

\noindent
\textbf{Outline.} \quad
The paper is organized as follows.  In Section~\ref{sec:notations}, the relevant notions pertaining to multivariate extreme value analysis are briefly recalled. The main result related to the non-asymptotic analysis of the empirical angular measure is formulated in Section~\ref{sec:traditionalEmpirical}. Section~\ref{sec:app} illustrates the application of the concentration inequalities to two statistical learning problems, minimum-volume set estimation and binary classification in extreme regions. Numerical experiments are presented in Section~\ref{sec:experiments}. A discussion in Section~\ref{sec:discussion} concludes the paper. Proofs and some auxiliary results, in particular a general concentration inequality for rare event probabilities, are deferred to the Appendices and the Supplement.

\section{Upper tail dependence} 
\label{sec:notations}

We set out some basics of multivariate extreme value theory (Sections~\ref{sec:RVmu} and~\ref{sec:angmeassup}) and recall a rank-based standardization procedure (Section~\ref{subsec:ranks}).
Equipped with these notions, we describe the empirical angular measure (Section~\ref{subsec:empang}). For background, we refer to monographs such as~\cite{resnick:1987, Beirlant05, deHaan2007extreme}.

\subsection{Regular variation and exponent measure}
\label{sec:RVmu}

Let $X = (X_1,\ldots, X_d)$ be a random vector with distribution $P$ and continuous marginal cumulative distribution functions $F_j(u) = \PP{X_j \le u}$ for $u \in \reals$. We standardize each component of $X$ to unit-Pareto margins by a combination of the probability integral transform and the quantile transform through $V_j = 1/(1 - F_j(X_j))$ for $j = 1, \ldots, d$.

The working hypothesis in this paper is that the resulting random vector $V = (V_1, \ldots, V_d)$ is multivariate regularly varying; this formalizes the assumption in point~(ii) in the introduction. Specifically, we assume that there exists a non-zero Borel measure $\mu$ on the punctured orthant $E = [0, \infty)^d \setminus \{(0, \ldots, 0)\}$ which is finite on Borel sets bounded away from the origin and such that
\begin{equation}
	\label{eq:standardRV}
	\lim_{t \to \infty} t\PP{t^{-1} V \in B} = \mu(B)
\end{equation}
for all Borel sets $B$ of $E$ bounded away from the origin and such that $\mu(\partial B) = 0$, with $\partial B$ the topological boundary of $B$. Convergence in \eqref{eq:standardRV} for all such $B$ is equivalent to measure convergence $t \PP{t^{-1} V \in \point} \to \mu(\point)$ as $t \to \infty$ in the space $\Mz$ of Borel measures on $E$ that are finite on Borel sets bounded away from the origin \citep{hult2006regular}. Specifically, we have $\lim_{t \to \infty} t \EE[f(t^{-1} V)] = \int_E f \ud \mu$ for every bounded and continuous real function $f$ on $E$ vanishing in a neighbourhood of the origin. Alternatively, multivariate regular variation can be described in the language of vague convergence of Radon measures on the compactified, punctured orthant $[0, \infty]^d \setminus \{(0, \ldots, 0)\}$ \citep{resnick:1987, resnick2007heavy}. 

The measure $\mu$ is referred to as the exponent measure because of its appearance in the exponent of the expression of the multivariate extreme value distribution to which the distribution of $V$ is attracted~\citep[Definition~6.1.7]{deHaan2007extreme}. The exponent measure is homogeneous: writing $\lambda A = \{ \lambda x : x \in A \}$ for $\lambda > 0$ and $A \subseteq \reals^d$, we have
\begin{equation}
	\label{eq:mu-homogeneous}
	\forall \lambda > 0, \qquad \mu(\lambda \point) = \lambda^{-1} \mu(\point).
\end{equation}
Its margins are standardized: by \eqref{eq:standardRV} and since each component $V_j$ is unit-Pareto distributed, we have 
\begin{equation}
	\label{eq:mu-margins}
	\forall y \in (0, \infty),\;  \forall j \in \{1, \ldots, d\},
	\qquad \mu(\{x \in E : x_j \ge y \}) = y^{-1}.
\end{equation}

\subsection{Angular measure}
\label{sec:angmeassup}

Recall $\norm{x}_p$ in~\eqref{eq: lpNorms} and let
\begin{equation}
	\label{eq:sphere}
	\sphere_p = \{x \in [0, \infty)^d : \norm{x}_p = 1 \},
\end{equation}
Consider the map $\theta_p:E \rightarrow \sphere_p$ that assigns to any vector $x\in E$ its `angle' $\theta_p(x) = x / \norm{x}_p$.

The angular measure $\Phi_p$ is defined as the push-forward measure by $\theta_p$ of the restriction of $\mu$ to $\{x \in E : \norm{x}_p \ge 1\}$: for any Borel set $A \subseteq \sphere_p$, we have
\begin{equation}
	\label{eq:angularMeasure}
	\Phi_p(A) = \mu(\mathcal{C}_A) 
	\quad \text{where} \quad
	\mathcal{C}_A = \{x \in E : \norm{x}_p \ge 1, \; \theta_p(x) \in A \}.
\end{equation}
	In view of the marginal standardization~\eqref{eq:mu-margins} and the ensuing identity~\eqref{eq:Phi-moments} below, the total mass $\Phi_p(\sphere_p) = \int_{\sphere_p} \norm{\theta}_p \, \diff \Phi_p(\theta)$ of the angular measure is finite and we have $1 \le \Phi_p(\sphere_p) \le d$.

The exponent measure $\mu$ is determined by the angular measure: by homogeneity~\eqref{eq:mu-homogeneous},
\[ 
\mu(\{x \in E : \norm{x}_p \ge u, \theta_p(x) \in A\}) = u^{-1} \Phi_p(A)
\]
for every $u > 0$ and every Borel set $A \subseteq \sphere_p$.
More generally, for any non-negative Borel measurable function $f$ on $E$, we have
\begin{equation}
	\label{eq:polar}
	\int_{E} f \, \diff \mu
	=
	\int_{\sphere_p} \int_0^\infty 
		f(r\theta) \, 
	\frac{\diff r}{r^2} \, \diff \Phi_p(\theta).
\end{equation}

The combination of the marginal standardization \eqref{eq:mu-margins} with the change-of-variable formula in \eqref{eq:polar} implies the identities
\begin{equation}
	\label{eq:Phi-moments}
	\forall j = 1, \ldots, d, \qquad
	\int_{\sphere_p} \theta_j \, \diff \Phi_p(\theta) = 1.
\end{equation}
In fact, any non-negative Borel measure $\Phi_p$ on $\sphere_p$ satisfying the moment constraints~\eqref{eq:Phi-moments} is the angular measure of some random vector $X$. Indeed, from such a measure $\Phi_p$ on $\sphere_p$, we can define a measure $\mu$ on $E$ through \eqref{eq:polar} and then consider the max-stable distribution with $\mu$ as exponent measure as in~\cite{de1977limit}. 

\subsection{Data standardization and ranks}
\label{subsec:ranks}

The component-wise transformation of $X$ to a vector $V$ with unit-Pareto margins is formalized by the map $\Vf:\rset^d\rightarrow [1,\; \infty]^d$, with
\begin{equation}
	\label{eq:vfunction}
	\forall x \in \reals^d, \qquad 
	\Vf(x) = \left(\frac{1}{1- F_1(x_1)}, \ldots, \frac{1}{1-F_d(x_d)} \right),
\end{equation}
with the convention $1/0=\infty$.
In this notation, we have $V = \Vf(X)$. 

Let $X_1, \ldots, X_n$ be an independent random sample from the distribution $P$ of $X$, with $X_i = (X_{i1},\ldots,X_{id})$. Since $P$ is unknown in practice,
it is replaced by its empirical version
$\loi_n(\point) = n^{-1}\sum_{i=1}^n \un\{X_i \in \point\}$, where $\un\{\event\}$ denotes the indicator variable of the event $\event$. 
In particular, the marginal cumulative distribution functions $F_j$ are substituted with their empirical counterparts
\begin{equation*}
	\widehat{F}_j(t) = \frac{1}{n}\sum_{i=1}^n \un\{X_{ij} \le t\}, \text{ for } t\in\rset \text{ and } 1\le j\le d,
\end{equation*}
in order to mimic the aforementioned standardization.

The empirically standardized sample points are
\begin{equation}
	\label{eq:Vhat}
	\widehat{V}_i = \hat{v}(X_i) = \bigl( \hat{v}_1(X_{i1}), \ldots, \hat{v}_d(X_{id}) \bigr), 
	\qquad i = 1, \ldots, n
\end{equation}
where $\hat{v}(x) = (\hat{v}_1(x_1), \ldots, \hat{v}_d(x_d))$ for all $x \in \reals^d$ and
\begin{equation}
	\label{eq:empcdf}
	\hat{v}_j(x_j) = \frac{1}{1 - \frac{n}{n+1} \hF_j(x_j)},
\end{equation}
for $j = 1, \ldots, d$ and $i = 1, \ldots, n$. The factor $\frac{n}{n+1}$ in \eqref {eq:empcdf} serves to avoid division by zero in case $x_j \ge \max(X_{1j}, \ldots, X_{nj})$.

We have $\hat{v}_j(X_{ij}) = 1 / (1 - R_{ij}/(n+1))$, where $R_{ij} = n \widehat{F}_j(X_{ij})$ is the rank of $X_{ij}$ among $X_{1j}, \ldots, X_{nj}$. Any statistic that is a function $\widehat{V}_1,\ldots,\widehat{V}_n$ depends on the data $X_1,\ldots,X_n$ only through the component-wise ranks. This will be the case for the empirical angular measure.

\subsection{The empirical angular measure}
\label{subsec:empang}

Let $\delta_x$ denote a point mass at $x$ and put
\begin{equation}
	\label{eq:hPn}
	\widehat{P}_n = \frac{1}{n} \sum_{i=1}^n \delta_{\widehat{V}_i},
\end{equation}
the empirical distribution of the pseudo-observations $\widehat{V}_1, \ldots, \widehat{V}_n$ in \eqref{eq:Vhat}.
The random measure $\widehat{P}_n$ can be legitimately considered as an estimator of the distribution of $V$.

The definition \eqref{eq:standardRV} of the exponent measure $\mu$ involves a limit as $t \to \infty$. Setting $t = n/k$ for $k \in \{ 1,\ldots,n \}$ such that both $k$ and $n/k$ are large yields the empirical exponent measure
\begin{equation}
	\label{eq:empiricalMu}
	\hmu(B) 
	= \tfrac{n}{k} \widehat{P}_n(\tfrac{n}{k} B)
	= \frac{1}{k} \sum_{i=1}^n \un \left\{ \widehat{V}_i \in \tfrac{n}{k} B \right\}
\end{equation}
for Borel sets $B \subseteq E$.
Consider a Borel set $A \subseteq \sphere_p$ and recall the angular measure $\Phi_p$ and the cone $\cone_A$ in~\eqref{eq:angularMeasure}.
In view of \eqref{eq:empiricalMu}, 
\emph{the empirical angular measure} \citep{einmahl2001nonparametric} is defined as 
\begin{equation}
	\label{eq:empiricalAngular}
	\hPhi_p(A) = \hmu(\cone_A) = 
	\frac{1}{k}\sum_{i=1}^n
	\un\left\{ \widehat{V}_i \in (n/k)\cone_A \right\}
	= \frac{1}{k} \sum_{i=1}^n \un \left\{ \norm{\widehat{V}_i}_p \ge n/k, \, \theta_p(\widehat{V}_i) \in A \right\}.
\end{equation}
It is the empirical version of the pre-asymptotic angular measure
\begin{equation*}
	t\PP{t^{-1} V \in \cone_A}
\end{equation*}
at level $t = n/k$.
In the bivariate case, the asymptotic distribution of the empirical angular measure has been investigated in case the sequence $k=k(n)$ satisfies $k \to \infty$ and $k/n \to 0$ as $n \to \infty$.
The max-norm was considered in~\cite{einmahl2001nonparametric}, while the $L_p$-norm for $p \in [1, \infty]$ was studied in~\cite{einmahl2009maximum} together with empirical likelihood methods to exploit the moment constraints~\eqref{eq:Phi-moments}.

\section{Concentration inequalities} 
\label{sec:traditionalEmpirical}

Recall the empirical angular measure $\hPhi_p$ in~\eqref{eq:empiricalAngular}. Our main result provides a concentration inequality for the uniform deviations
\begin{equation}
	\label{eq:suprema}
	\sup_{A \in \mathcal{A}} \left| \widehat{\Phi}_p(A) - \Phi_p(A) \right|.
\end{equation}
The supremum is taken over classes $\mathcal{A}$ of Borel sets $A$ of $\sphere_p$ in~\eqref{eq:sphere} satisfying appropriate assumptions. To bound the supremum, the empirical measure is rewritten in terms of random sets which are subsequently framed in between deterministic sets. We first explain the idea behind the framing approach (Section~\ref{sec:framing}) before stating our main theorem on concentration bound for the empirical angular measure (Section~\ref{subsec:nonparametricestimator}). We conclude with examples of collections $\cA$ (Section~\ref{subsec:examples}).

\subsection{Framing random sets}
\label{sec:framing}

Since the estimators depend on the data only through the ranks, they are invariant under increasing component-wise transformations. So although the marginal distributions $F_1,\ldots,F_d$ are unknown, we can and will nevertheless assume that they are unit-Pareto already. In that case, we have $v(x) = x$ for $x \in [1, \infty]^d$ in \eqref{eq:vfunction}, so that $V = v(X) = X$ and $V_i = v(X_i) = X_i$ for $i = 1, \ldots, n$. In particular, the distribution of $V = X$ is $P$.

Recall $\widehat{P}_n$ in \eqref{eq:hPn} and let $P_n = n^{-1} \sum_{i=1}^n \delta_{V_i}$ denote the empirical distribution of $V_1,\ldots,V_n$. Since $\widehat{V}_i = \hat{v}(X_i) = \hat{v}(V_i)$, we have $\widehat{P}_n = P_n \circ \hat{v}^{-1}$, that is, $\widehat{P}_n$ is the push-forward measure of $P_n$ by $\hat{v}$. Up to a scaling factor, the value of the empirical angular measure $\hPhi_p$ in a Borel set $A$ of $\sphere_p$ can thus be expressed as $P_n$ evaluated in the random set 
\begin{equation}
	\label{eq:hGammaA}
	\widehat{\Gamma}_A =
	\hat{v}^{-1}(\tfrac{n}{k} \cone_A) =
	\left\{ 
	x \in E : \
	\norm{\hat{v}(x)}_p \ge \tfrac{n}{k}, \;
	\theta_p(\hat{v}(x)) \in A
	\right\},
\end{equation}
where, as before, $\theta_p(y) = y / \norm{y}_p$ for $y \in E$. Indeed, we have
\[
\hPhi_p(A) 
= \tfrac{n}{k} \widehat{P}_n(\tfrac{n}{k} \cone_A)
= \tfrac{n}{k} P_n(\hat{v}^{-1}(\tfrac{n}{k} \cone_A))
= \tfrac{n}{k} P_n(\widehat{\Gamma}_A).
\]

Let $\cA$ be a class of Borel sets of $\sphere_p$.
Following in the footsteps of \cite{einmahl2001nonparametric} and \cite{einmahl2009maximum}, we construct for any $A\in\mathcal{A}$ two nested deterministic sets $\Gamma^-_A \subseteq \Gamma^+_A$ framing the cone $\cone_A$, i.e., such that  $\Gamma^-_A\subseteq \cone_A\subseteq \Gamma^+_A$, and such that on an event that occurs with high probability, we have 
\begin{equation}\label{eq:nested_emp}
	\forall A \in \cA, \qquad 
	\tfrac{n}{k} \Gamma^-_A \subseteq \widehat\Gamma_{A} \subseteq \tfrac{n}{k} \Gamma^+_A.
\end{equation}  

On that event, the signed error can then be bounded from above by 
\begin{align*}
	\hPhi_p(A) - \Phi_p(A)
	&= \tfrac{n}{k} P_n(\widehat{\Gamma}_A) - \mu(\cone_A) 
	\le \tfrac{n}{k} P_n(\tfrac{n}{k} \Gamma_{A}^+) - \mu(\Gamma_{A}^-) \\
	&\le \tfrac{n}{k} \left| P_n(\tfrac{n}{k} \Gamma_{A}^+) - P(\tfrac{n}{k} \Gamma_{A}^+) \right|
	+ \left| \tfrac{n}{k} P(\tfrac{n}{k} \Gamma_{A}^+) - \mu(\Gamma_{A}^+) \right|
	+ \mu(\Gamma_{A}^+ \setminus \Gamma_{A}^-).
\end{align*}
A lower bound can be derived in a similar way, yielding, on the same event,
\begin{equation}
	\label{eq:dec3}
	\left.
	\begin{aligned}
		\left| \hPhi_p(A) - \Phi_p(A) \right|
		&\le
		\max_{B \in \{\Gamma_{A}^+, \Gamma_{A}^-\}}
		\left| \tfrac{n}{k} P(\tfrac{n}{k} B) - \mu(B) \right| 
		& \text{(bias term)} \\	
		&\quad {} +
		\max_{B \in \{\Gamma_{A}^+, \Gamma_{A}^-\}}
		\tfrac{n}{k} 
		\left|
		P_n(\tfrac{n}{k} B) 
		- P(\tfrac{n}{k} B) 
		\right| \qquad 
		& \text{(stochastic error)}   \\	
		&\quad {} +
		\mu(\Gamma_{A}^+ \setminus \Gamma_{A}^-)
		& \text{(framing gap).} 
	\end{aligned}
	\quad
	\right\}
\end{equation}

Next we will introduce assumptions to enable the construction of the framing sets $\Gamma_A^-$ and $\Gamma_A^+$ together with a high-probability event on which \eqref{eq:nested_emp} holds. The task is then to control the three terms on the right-hand side of~\eqref{eq:dec3} uniformly over $A\in\cA$. The bias term will be left as such; controlling it is an entirely different subject requiring higher-order multivariate regular variation \citep{beirlant+e+g+g:2016, fougeres+dh+m:2015}; \chg{however, see Remark~\ref{rem:bias:handle}}. Under appropriate complexity assumptions on the class $\cA$ and the collection of framing sets, the stochastic error term can be uniformly bounded by means of the concentration inequality for tail empirical processes established in Theorem~\ref{thm:chaining}. Finally, the framing gap can be controlled by ensuring that the set $\Gamma_A^+ \setminus \Gamma_A^-$ is small.

\subsection{Concentration bounds for the empirical angular measure}
\label{subsec:nonparametricestimator}

The main result of this paper, Theorem~\ref{thm:conc}, is stated in this subsection. 
Let $\ball = [-1, +1]^d$ denote the closed unit ball in $\reals^d$ with respect to the sup-norm $\norm{\point}_\infty$.
For sets $A, B \subseteq \reals^d$ write $A + B = \{a + b : a \in A, b \in B\}$.
For $\eps > 0$ and $A \subseteq \reals^d$, we thus have $A + \eps \ball = \{ x \in \reals^d : \exists a \in A, \norm{x - a}_\infty \le \eps \}$.

\begin{assumption}[Subsets of the sphere]
	\label{ass:A}
	The class $\cA$ is a collection of non-empty Borel sets of $\sphere_p$ with the following properties:
	\begin{enumerate}[label=(\roman*)]
		\item There exists a countable collection $\cA_0 \subseteq \cA$ such that for every $A \in \cA$ there is a sequence $A_n \in \cA_0$ such that $\lim_{n \to \infty} \1\{x \in A_n\} = \1\{x \in A\}$ for every $x \in \sphere_p$.
		\item There exists $\tau \in (0, 1)$ such that
		\begin{equation}
			\label{eq:boundedfrom0}
			\forall A \in \cA, \qquad	A \subseteq \{ x \in \sphere_p : \ \min(x) > \tau \} =: \sphere_p^\tau.     
		\end{equation}
		\item There is a constant $c>0$ such that for any $A \in \cA$ and $\eps > 0$, there exist Borel subsets $A_+(\eps)$ and $A_-(\eps)$ of $\sphere_p$ satisfying
		\begin{equation}
			\label{eq:Phi+-}
			\Phi_p \bigl(A_+(\eps) \setminus A_-(\eps)\bigr) \le c \, \eps \,
		\end{equation}
		together with the inclusions
		\begin{equation}
			\label{eq:Ahull:inclusions}
			\bigl(A_-(\eps) + \eps \ball\bigr) \cap \sphere_p \subseteq A
			\qquad \text{and} \qquad
			(A + \eps \ball) \cap \sphere_p \subseteq A_+(\eps).
		\end{equation}
	\end{enumerate}
\end{assumption}

Condition~(i) amounts to the pointwise measurability of the indicators $\left\{ \1\{ \cdot \in A \} : A \in \cA\right\}$ \citep[Example~2.3.4]{vdvaart+w:1996} and ensures that for any $A \in \cA$ we can find $A_n \in \cA_0$ such that for any finite Borel measure $\nu$ on $\sphere_p$ we have $\nu(A) = \lim_{n \to \infty} \nu(A_n)$. The suprema over $\cA$ in Eq.~\eqref{eq:suprema} are therefore equal to those over $\cA_0$ and are thus measurable.
Condition (ii) stipulates that the elements of the class $\mathcal{A}$ are bounded away from the $2^d-1$ faces of the sphere. Though it may be considered as restrictive at first glance, we point out that maximal deviations of the empirical angular measure over classes of Borel subsets of a given face of the sphere correspond to maximal deviations of the empirical angular measure for the corresponding components of~$X$. 
The crucial point in Condition~(iii) is inequality~\eqref{eq:Phi+-}, bounding the measure of the difference between the inner and outer $\eps$-hulls of $A \in \cA$. 
The inequality is satisfied if it holds with $\Phi_p$ replaced by the $(d-1)$-dimensional Lebesgue measure on $\sphere_p$ and $\Phi_p$ has a bounded density on $\sphere_p$ with respect to this measure.

In order to deal with the estimation error stemming from the use of the marginal empirical distribution functions in \eqref{eq:empcdf}, we frame the cones $\cone_A$ in Eq.~\eqref{eq:angularMeasure} between slightly smaller and larger sets built from the inner and outer hulls.
For $A \in \cA$, $\sigma\in\{-,+\}$ and $r, h > 0$, define
\begin{equation*}
	\Gamma_{A}^{\sigma}(r,h) = \left\{
	x \in [0, \infty)^d : \
	\norm{x}_p \ge \tfrac{1}{r}, \;
	\theta_p(x) \in A_{\sigma}(h \norm{x}_p)
	\right\}.
\end{equation*}
For all $A\in \cA$ and $h>0$, we have $\Gamma_{A}^-(r,h)\subseteq \cone_A$ for $0 < r\leq 1$ and $\cone_A\subseteq \Gamma_{A}^+(r,h)$ for $r\geq 1$.
The upper confidence bound at level $1-\delta$ for the maximal deviation \eqref{eq:max_dev} stated in Theorem~\ref{thm:conc} below is derived from the decomposition~\eqref{eq:dec3} with framing sets $\Gamma_{A}^+(r_+,h)$ and $\Gamma_{A}^-(r_-,h)$ for specific choices of $r_+$, $r_-$ and $h$.

To control the stochastic error, we need a handle on the complexity of the collection of framing sets.
The Vapnik--Chervonenkis (VC) dimension of a collection $\cF$ of subsets of some set $\cX$ is the supremum (possibly infinite) of the set of positive integers $n$ with the property that there exists a subset $\{x_1,\ldots,x_n\}$ of $\cX$ with cardinality $n$ such that all $2^n$ subsets of $\{x_1,\ldots,x_n\}$ can be written in the form $\{x_1,\ldots,x_n\} \cap F$ for some $F \in \cF$.
The VC-dimension is a central quantity in statistical learning and empirical process theory as it lies at the basis of many concentration inequalities for empirical distributions, see for instance \cite{vdvaart+w:1996} and \cite{bousquet2004introduction}.

\begin{assumption}
	\label{ass:VF}
	For any $r, h > 0$, the collections $\{ \Gamma_A^-(r, h) : A \in \cA \}$ and $\{ \Gamma_A^+(r, h) : A \in \cA \}$ of subsets of $E$ have finite VC-dimension.
\end{assumption}

In the framing sets $\Gamma_A^\sigma(r, h)$, the tolerance $\eps$ for the angle $\theta_p(x)$ in the inner and outer hulls of a set $A$ depends on the norm $\norm{x}_p$.
Therefore, in Assumption~\ref{ass:VF} it is not sufficient to assume that the collection $\cA$ or even the collection of inner and outer hulls has finite VC-dimension.
In Section~\ref{subsec:examples} we provide a realistic example of an angular class $\mathcal{A}$---namely a class defined by linear inequality constraints---where Assumptions~\ref{ass:A} and~\ref{ass:VF} are satisfied.

\begin{theorem}
	\label{thm:conc}
	Let $X_1, \ldots, X_n$ be an independent random sample from a distribution $P$ on $\reals^d$ with continuous margins \chg{and let $P_V$ denote the  distribution of $V_1 = v(X_1)$ where $v$ is defined in~\eqref{eq:vfunction}}. Let $\mu$ be an exponent measure having angular measure $\Phi_p$ with respect to $\norm{\,\cdot\,}_p$ for some $p \in [1,\infty]$.
	Let $\cA$ be a collection of Borel sets of $\sphere_p$ such that Assumptions~\ref{ass:A} and~\ref{ass:VF} are fulfilled.
	Let $n, k, \rho$ be such that $\tau n > k > (3 \vee 6c)$ and $k/n < \rho < \tau$, where the constant $c>0$ comes from point~(iii) of Assumption~\ref{ass:A}. Let $\delta \in (0,1)$.
	Put
	\[
	\Delta_1 = \Delta_1(k,\delta,\rho) 
	= \frac{1}{\sqrt{k\rho}} \left(\cst + 2 \sqrt{\log((d+1)/\delta)} \right) + \frac{2}{3k} \log((d+1)/\delta). 
	\]
	For $\sigma \in \{-, +\}$ and $A \in \cA$, abbreviate $\Gamma_A^\sigma = \Gamma_A^\sigma(r_\sigma, 3\Delta_1)$ where $r_\pm = 1 \pm \Delta_2$ with $\Delta_2 = 4(\Delta_1 + 1/k)$.
	Then, with probability at least $1-\delta$, the empirical angular measure satisfies
	\begin{equation}\label{eq:conc1}
		\begin{split}
			\sup_{A\in \cA} \left| \hPhi_p(A)-\Phi_p(A) \right|
			&\le \sup_{A\in \cA,\; \sigma\in\{+,-\}} \left\vert \tfrac{n}{k} P_V(\tfrac{n}{k}  \Gamma_{A}^{\sigma}) - \mu(\Gamma_{A}^{\sigma})\right\vert \\
			&\qquad + \sqrt{\frac{d^{1+1/p}(1+\Delta_2)}{k}} \left( \cst \sqrt{V_\F} + 2 \sqrt{\log((d+1)/\delta)} \right) + \frac{2 \log \left( \tfrac{d+1}{\delta} \right)}{3k}  \\
			&\qquad + 2d\Delta_2 + 3c\big(\log(d/3c) - \log(\Delta_1) + 1 \big) \Delta_1,
		\end{split}
	\end{equation}
	provided $k$ is sufficiently large so that $(2/k) \le \Delta_1 < (1/4 - 1/k) \wedge (1/(3c))$ and $\rho/(1-\Delta_1\rho) \le \tau$ and where $V_{\cF}$ is the VC-dimension of the collection
	\begin{equation}
		\label{eq:cF:union}
		\left\{ \Gamma_A^- : A \in \cA \right\} \cup 
		\left\{ \Gamma_A^+ : A \in \cA \right\}.
	\end{equation}
\end{theorem}

\bgroup
\color{chgcol}
The proof is given in Appendix~\ref{supp:proofs:main} in the Supplement. 

Note that $r_-\leq 1\leq r_+$, so that $\Gamma_A^-(r_-, h) \subseteq \cone_A \subseteq \Gamma_A^+(r_+, h)$ for all $h > 0$. In the decomposition~\eqref{eq:dec3}, the framing gap is bounded by $2d\Delta_2 + 3c\big(\log(d/3c) - \log(\Delta_1) + 1 \big) \Delta_1$ while the estimation error is bounded by 
\[
\sqrt{\frac{d^{1+1/p}(1+\Delta_2)}{k}} \left( \cst \sqrt{V_\F} + 2 \sqrt{\log((d+1)/\delta)} \right) + \frac{2}{3k} \log((d+1)/\delta)
\] 
with probability larger than $1-\delta$.
This term is the one appearing in the concentration inequality for empirical processes over collections of sets of extreme values proved in Theorem~\ref{thm:chaining} applied to the collection $\{ (n/k) \Gamma_{A}^{\sigma}(r_\sigma, 3\Delta_1) :\; \sigma\in\{+,-\},\;  A \in \cA \}$.
By Assumption~\ref{ass:VF}, the collection in \eqref{eq:cF:union} has a finite VC-dimension: For two collections $\cF_1$ and $\cF_2$ of subsets of a set $\cX$ with finite VC-dimensions $d_1$ and $d_2$, respectively, the VC dimension of $\cF_1 \cup \cF_2$ is bounded by $d_1 + d_2 + 1$ \citep[Exercise~3.24]{mohri2018foundations}. 

Note that the bound is a decreasing function of the norm index $p$. It is related to the shape of the associated sphere $\sphere_p$, which is easier to work with if more aligned with the axes. The faces of the unit sphere induced by the supremum norm are parallel to the coordinate axes, a property that links up well with the use of component-wise ranks, and has the advantage that its value is not affected by the precise values of the smaller coordinates of~$x$. Extensions to other $p$-norms can be performed through the use of the equivalences of norms
\begin{equation}
	\label{eq: equiv norms}
	\forall x \in \reals^d, \qquad \norm{x}_\infty \leq \norm{x}_p \leq d^{1/p} \norm{x}_\infty.
\end{equation}

For fixed $\rho$ and $\delta$, when ignoring the bias term, the bound on the right-hand side of \eqref{eq:conc1} is of the order
\begin{equation} 
	\label{eq:conc:rate}
	\Delta_1 \log(1/\Delta_1) 
	=  \Oh \left( \frac{\log k}{\sqrt{k}} \right), \text{ as } k \to \infty.
\end{equation}
Even though $\Delta_1$ and $\Delta_2$ are of the optimal order $\Oh(1/\sqrt{k})$, the factor $\log k$ shows up: its presence is caused by the framing gap, the third line in both~\eqref{eq:dec3} and~\eqref{eq:conc1}. Asymptotic theory for the empirical angular measure in the bivariate case \citep{einmahl2001nonparametric, einmahl2009maximum} suggests a learning rate of the order $1/\sqrt{k}$; whether our logarithmic term is an artifact of our analysis or rather a genuine property of the estimator remains an open problem.
\egroup
	
	\begin{remark}[Other concentration inequalities]
		\label{rem:other}
		The constant $56$ appearing in the error stems comes from the use of chaining techniques as in~\cite[Theorems~1.16--17]{lugosi:2002}. 
			Relying instead on concentration inequalities for rare events derived recently in~\cite{lhaut2021uniform}, it is possible to reduce the constant at the price of an additional logarithmic factor $\sqrt{\log k}$, which, for realistic values of $k$, may well be smaller than the constant involved in our bound. However, the bound would become more complicated and less accurate from an asymptotic point of view. 

			The term $\Delta_1$ in Theorem~\ref{thm:conc} comes from the application of Theorem~\ref{thm:chaining} to the tails of the empirical margins $\widehat{F}_j$. As observed by an anonymous Referee, other concentration inequalities exist for such one-dimensional tail empirical processes, see for instance Inequality~1 on page~446 in~\cite{shorack+wellner2009}. For finite samples, some of these inequalities may be sharper than the one we used; nevertheless, the convergence rate as a function of $k$ would not improve since chaining is already optimal in that respect.	
	\end{remark}

	\begin{remark}[Bias term and penultimate angular measure]
	\label{rem:bias}
		As Theorem~\ref{thm:conc} is not concerned with asymptotics, we did not actually have to assume that $\Phi_p$ is the angular measure associated $P_V$. The link between $P_V$ and $\Phi_p$ is quantified instead by the bias term $\sup_{A,\sigma} | \frac{n}{k} \, P_V(\frac{n}{k} \Gamma_A^\sigma) - \mu(\Gamma_A^\sigma) |$. Even if $P_V$ has an angular measure of its own, it may be different from the one appearing in the theorem. This flexibility allows for viewing the empirical angular measure as an estimator of a penultimate angular measure, that is, the one for which the induced bias term is minimal.
	\end{remark}

\begin{remark}[Controlling the bias term]
\label{rem:bias:handle}
	\bgroup
	\color{chgcol}
	For absolutely continuous models, a primitive condition on the probability density function allows to control the bias term in Theorem~\ref{thm:conc}. Let $P_U$ denote the distribution of $U = (1-F_1(X_1),\ldots,1-F_d(X_d)) = \iota(V)$ on $[0, 1]^d$ where $\iota : (0, \infty)^d \to (0, \infty)^d$ is defined by $\iota(x) = (x_1^{-1}, \ldots, x_d^{-1})$. Assume that $P_U$ is absolutely continuous with density $p_U$ and that the measure $\Lambda = \mu \circ \iota^{-1}$ (the push-forward of $\mu$ by $\iota$) is absolutely continuous with Lebesgue density $\lambda$ on $(0, \infty)^d$. Then
	\begin{equation}
	\label{eq:biasbound}
		\sup_{A \in \cA, \sigma \in \{+,-\}} \left| 
			\tfrac{n}{k} \, P_V(\tfrac{n}{k} \Gamma_A^\sigma) 
			- \mu(\Gamma_A^\sigma) 
		\right|
		\le
		\int_{(0, \infty)^d} 
			\1\left\{\min(y) \le d^{1/p} r_+\right\} \,
			\left| \left(\tfrac{k}{n}\right)^{d-1} p_U(\tfrac{k}{n}y) - \lambda(y) \right| \,
		\diff y.
	\end{equation}
	By way of example, consider the multivariate Cauchy density restricted to $(0, \infty)^d$ with probability density function
	\[
		f(x) = 2^d \Gamma(\tfrac{d+1}{2}) \pi^{-(d+1)/2} \left(1+x_1^2+\cdots+x_d^2\right)^{-(d+1)/2},
		\qquad x \in (0, \infty)^d,
	\]
	with limit density
	\[
		\lambda(x) = 2^{d-1} \Gamma(\tfrac{d+1}{2}) \pi^{-(d-1)/2} x_1^{-2} \cdots x_d^{-2} \left( x_1^{-2} + \cdots + x_d^{-2} \right)^{-(1+d)/2}, 
		\qquad x \in (0, \infty)^d.
	\]
	In this case, the bound \eqref{eq:biasbound} is of the order $\Oh(k/n)$ as $k = k_n \to \infty$ in such a way that $k/n \to 0$. Detailed calculations are given in Appendix~\ref{sec:appendixBias} in the Supplement.
	\egroup
	\end{remark}

	\subsection{Examples of collections of subsets of the sphere}
	\label{subsec:examples}
	
	This sections aims at providing examples of classes $\cA$ related to a wide range of statistical machine learning algorithms (such as logistic regression,  classification and regression trees or linear discriminant analysis) for which Assumptions~\ref{ass:A} and~\ref{ass:VF} are satisfied. \chg{Proofs are deferred to Appendix~E in the Supplement.}
	
	Recall $\sphere_p^\tau$ in \eqref{eq:boundedfrom0}.
	The scalar product and the Euclidean norm on $\reals^d$ are denoted by $\inpr{x,y}$ and $\norm{x}_2 = \sqrt{\inpr{x,x}}$, respectively.
	
	\begin{example}[Linear restrictions]
	\label{ex:linear}
		Fix $\tau \in (0, 1)$ and consider the collection 
		\[ 
		\cA = \left\{ 
		A_{a, \beta, \tau} : \
		(a, \beta) \in \reals^d \times \reals, \; \norm{a}_2 = 1
		\right\} 
		\]
		of Borel subsets of $\sphere_p^\tau$ defined via
		\begin{align*}
			A_{a, \beta} &= \{ x \in \sphere_p : \ \inpr{a, x} \le \beta \}, &
			A_{a, \beta, \tau} &= A_{a, \beta} \cap \sphere_p^\tau.
		\end{align*}
	Then $\cA$ satisfies Assumption~\ref{ass:VF}, and, if $\Phi_p$ has a bounded Lebesgue density on $\sphere_p$, also Assumption~\ref{ass:A}.
	\end{example}

	\begin{example}[Stability under intersections and unions]
	\label{ex:intersec-union}
	\bgroup
	\color{chgcol}
		Let $\cA_1$ and $\cA_2$ be two collections of Borel subsets of $\sphere_p$ that satisfy Assumptions~\ref{ass:A} and~\ref{ass:VF}. Then the same is true for the collection of intersections
		\[
		\cA_1 \sqcap \cA_2 = \left\{ 
		A_1 \cap A_2 : \ A_1 \in \cA_1, \, A_2 \in \cA_2 
		\right\}
		\]
		and for the collection of unions
		\[
		\cA_1 \sqcup \cA_2 
		= \left\{ 
		A_1 \cup A_2 : \ A_1 \in \cA_1, \, A_2 \in \cA_2 
		\right\}.
		\]
		In combination with Example~\ref{ex:linear}, this property covers classifiers built from decision trees with a given depth. The leaves of such a tree correspond to unions of rectangles, the maximum number of rectangles in the union being determined by the tree depth.
	\egroup
	\end{example}

	\section{Applications to statistical learning} 
	\label{sec:app}
	\chg{
	We illustrate how a concentration inequality such as the one in Theorem~\ref{thm:conc} is useful to establish sound non-asymptotic guarantees for the validity of certain statistical learning procedures relying on the empirical angular measure and recently introduced in the literature. 
	In Section~\ref{sec:appliMVsets}, after introducing some minimal background about anomaly detection in extreme regions via minimum-volume set estimation, we show how our main results bring immediate guarantees on this matter. 
	In Section~\ref{sec:appliClassif}, we recall the whys and wherefores of classification in extreme regions and leverage the techniques developed in Section~\ref{sec:traditionalEmpirical} to control the excess risk of a specific empirical risk minimizer targeting the tail region of the covariate space.}
	
	\subsection{Minimum-volume set estimation}
	\label{sec:appliMVsets}
	\chg{
	To illustrate the usefulness of bounds on the supremum in \eqref{eq:max_dev}, consider the problem of estimating a \emph{minimum-volume set} \citep{EinmahlMason92}, such sets extending the notion of univariate quantiles. A minimum-volume set at level $\alpha$ is a subset of the sample space of minimum (Lebesgue) volume, constrained to contain a probability mass of at least $\alpha$. There are fruitful connections between minimum-volume sets and semi-supervised anomaly detection, where data are available from the majority class only and the goal is to construct a decision function delimitating the extent of the normal region. In a Neyman--Pearson framework, an optimal anomaly detection procedure at a certain level $ 0 < 1-\alpha \ll 1$ would declare abnormal any new point such that no minimum-volume set of level $\alpha$ contains it \citep{blanchard2010}. In the context of anomaly detection, the tail of the random vector under scrutiny is of particular interest because many anomalies correspond to unusually large values of at least one component. However, it may not be appropriate to declare as abnormal all such points and a finer analysis of the tails can improve the overall performance of an anomaly detection algorithm. For instance, consider a complex infrastructure monitored by several physical variables. Raising an alert at each extreme value of one of its physical variables can lead to high false alarm rates. A way to reduce this false alarm rate is to study the multivariate distribution of the set of observations such that at least one of their variables is large. This framework can be useful in a wide variety of applications (e.g., fraud detection, safety in aeronautics), where the control of the false alarm rate is crucial (given the cost of safety inspections), thus making the detection of anomalies among extremes undeniably relevant.}
	
	\chg{
	In this section, we follow in the footsteps of \cite{thomas17a} who consider the problem of constructing sets of relatively high probability in regions of the kind $\{x \in \RR^d: \|x\| > t \}$ for large values of $t$, under regular variation assumptions. Their statistical analysis is limited to the ideal case where the marginal distributions are known. We extend their guarantees in order to encompass the influence of the rank transformation.}
	
	In the context of the angular measure for multivariate upper extremes, the question is to find a Borel set $\Omega$ of the unit sphere $\sphere_p$ in $[0, \infty)^d$ with minimal volume $\lambda(\Omega)$---with $\lambda$ some reference measure such as the $(d-1)$-dimensional Hausdorff measure---but still having content $\Phi_p(\Omega)$ not smaller than some pre-specified lower limit $\alpha \in (0, \Phi_p(\sphere_p))$.
	In \cite{thomas17a}, such a set is used for the purpose of (unsupervised) anomaly detection in extreme regions. As $\Omega$ is supposed to contain a large fraction $\Phi_p(\Omega)/\Phi_p(\sphere_p)$ of the possible directions of extreme points, a new such point is deemed to be a potential anomaly if it lies in a direction outside $\Omega$. The fact that $\Omega$ has minimal volume $\lambda(\Omega)$ means that the critical region $\sphere_p \setminus \Omega$ to detect suspicious points is as large as possible.
	
	As $\Phi_p$ is unknown, $\Omega$ needs to be learned from a training sample.
	Although $\Omega$ may be characterized as a certain super-level set of the density of $\Phi_p$ with respect to the reference measure $\lambda$, a more practical approach than estimating this density, especially in high dimensions, is to limit the search to an algorithmically manageable collection $\cA$ of Borel subsets of $\sphere_p$.
	Let $\hat{\Phi}_p$ be any estimator of $\Phi_p$, not necessarily the empirical angular measure.
	Following the logic in \cite{ScottNowak06}, let $\hat{A}$ solve the empirical angular minimum-volume set problem
	\[
	\min \bigl\{ \lambda(A) : \ A \in \cA, \hat{\Phi}_p(A) \ge \alpha - \psi \bigr\}
	\]
	where $\psi \in (0, \alpha)$ is a tolerance parameter.
	The price to pay for having to estimate the angular measure is that the minimal required content $\alpha$ has been relaxed to $\alpha-\psi$.
	
	Bounds on the largest estimation error of $\hat{\Phi}_p$ over $\cA$ are helpful to provide probabilistic guarantees for $\hat{A}$.
	Suppose that, on some event $\mathcal{E}$, we have
	\begin{equation}
		\label{eq:sup:psi}
		\sup_{A \in \cA} \left|\hat{\Phi}_p(A) - \Phi_p(A)\right| \le \psi.
	\end{equation}
	Then, on the same event $\mathcal{E}$, we obviously have
	\begin{equation}
		\label{eq:MV:PhihA}
		\Phi_p(\hat{A}) \ge \hat{\Phi}_p(\hat{A}) - \psi \ge \alpha - 2\psi
	\end{equation}
	as well as
	\begin{equation}
		\label{eq:MV:lambdaA}
		\lambda(\hat{A}) \le \inf \left\{ \lambda(A) : \ A \in \cA, \, \Phi_p(A) \ge \alpha \right\}.
	\end{equation}
	Indeed, on $\mathcal{E}$, the collection $\{ A \in \cA : \Phi_p(A) \ge \alpha \}$ is contained in $\{A \in \cA : \hat{\Phi}_p(A) \ge \alpha - \psi \}$ so that the infimum of $\lambda(A)$ for $A$ in the latter collection must be the smaller one.
	By~\eqref{eq:MV:PhihA}, the empirical solution $\hat{A}$ is guaranteed to have at least content $\alpha - 2\psi$ under $\Phi_p$, while according to~\eqref{eq:MV:lambdaA}, the volume of $\hat{A}$ is smaller than the one of the actual minimum-volume set under $\Phi_p$. 
	
	Concentration inequalities for the supremum of $|\hat{\Phi}_p(A) - \Phi_p(A)|$ over $A \in \cA$ provide the existence, for a given $\delta > 0$, of a scalar $\psi \equiv \psi(\delta)$ such that \eqref{eq:sup:psi} holds on an event with probability at least $1-\delta$. It follows that, with high probability, the empirical minimum-volume set $\hat{A}$ satisfies~\eqref{eq:MV:PhihA} and~\eqref{eq:MV:lambdaA}.
	Provided $\psi$ and $\delta$ are both small, the combination of both properties justifies the use of $\hat{A}$ as an approximation to the true but unknown minimum-volume set under $\Phi_p$.
	For the empirical angular measure, a valid choice for the tolerance parameter $\psi(\delta)$ is given by the upper bound in Theorem~\ref{thm:conc}.

	\subsection{Classification in extreme regions}
	\label{sec:appliClassif}
	\chg{
	We apply Theorem~\ref{thm:conc} to binary classification in extreme regions. Classification is arguably one of the most studied problems in the statistical learning literature. Most existing guarantees are formulated in terms of a risk which is an integrated version of the loss function over the whole covariate space. However, the local performance of a global classifier is not necessarily guaranteed in low probability regions of the covariate space, typically in regions corresponding to an exceedance by one component of a high quantile---where few training points are available---or outside the convex envelope of the training set. In other words, the risk of an error conditional to the norm of the input being large is not adequately controlled in the classical setting. Nevertheless, in a wide variety of applications, ranging from finance/insurance to environmental sciences through teletraffic data analysis for instance, extreme observations of the covariates are of crucial importance.}
	
	\chg{
	In this section, we adopt the framework originally proposed in~\cite{jalalzai2018binary}, who formalize this argument and propose a risk minimization strategy aiming at improving performance of classification algorithms on such regions. In~\cite{jalalzai2018binary}, the marginal distributions of the predictor variables were assumed to be known, so that a standardized vector with exact unit-Pareto margins is observable. Here, we rather assume that the margins are unknown and employ a rank-based standardization instead.}
	
	First we recall the set-up of \cite{jalalzai2018binary}.
	Consider a random pair $(V, Y)$ where $Y$ in $\{-1, 1\}$ is the label to be predicted and $V$ in $[0, \infty)^d$ is the vector of predictors (features).
	The goal is to learn a classifier $g: [0, \infty)^d \to \{-1,1\}$ such that the classification risk for feature vectors $V$ far away from the origin is small.
	Let $\varrho = \PP{Y = 1}$ and assume $0 < \varrho < 1$.
	
	The starting point in \cite{jalalzai2018binary} is to assume a conditional version of the regular variation condition~\eqref{eq:standardRV}: there exist non-zero Borel measures $\mu_+$ and $\mu_-$ on $E = [0, \infty)^d \setminus \{0\}$ that are finite on Borel sets bounded away from the origin and such that
	\begin{equation}
		\label{eq:conditional_RV}
		\lim_{t \to \infty} t \PP{t^{-1} V \in B \given Y = \sigma\,1   }
		= \mu_{\sigma}(B)
	\end{equation}
	for $\sigma \in \{-, +\}$ and Borel sets $B \subseteq E$ bounded away from the origin satisfying $\mu_\sigma(\partial B)= 0$.
	The angular measure associated to the $L_p$-norm of $\mu_{\sigma}$ is 
	\[ 
	\Phi_p^\sigma(A) = \mu_{\sigma}(\cone_A), 
	\qquad \text{for $\sigma \in \{-, +\}$ and Borel sets $A \subseteq \sphere_p$}, 
	\]
	with $\cone_A$ as in \eqref{eq:angularMeasure}.
	By \eqref{eq:conditional_RV}, the unconditional distribution of $V$ is regularly varying as in \eqref{eq:standardRV} with limit measure $\mu = \varrho \mu_+ + (1-\varrho)\mu_-$ and angular measure $\Phi_p = \varrho \Phi_p^+ + (1-\varrho) \Phi_p^-$.
	
	In \cite{jalalzai2018binary} it was assumed that an independent random sample $\{(V_i, Y_i)\}_{i=1}^n$ from the distribution of $(V, Y)$ is given.
	Here, instead, the set-up is that we observe an independent random sample $\{(X_i, Y_i)\}_{i=1}^n$ from the distribution of $(X,Y)$ and that~\eqref{eq:conditional_RV} holds with $V = \Vf(X)$, where  $\Vf$ is defined via the probability integral transform in~\eqref{eq:vfunction}.
	This means that the classifier will have to be learned from the pairs $(\widehat{V}_i, Y_i)$ where $\widehat{V}_i = \hat{v}(X_i)$ in \eqref{eq:Vhat} is based on the rank transform.
	
	We emphasize that the marginal distribution functions $F_j$ in the definition of $\Vf$ are not conditioned upon $Y$.
	Indeed, for a new observation, marginal standardization will have to be carried out without the knowledge of the label to be predicted. 
	
	In \cite{jalalzai2018binary} the focus  is on the conditional classification risk above level $t > 0$ of a classifier $g: [0, \infty)^d\setminus\{0\}\to \{-1,1\}$ defined on the standardized input $V$:
	\begin{equation}
		\label{eq:extreme-risk}
		L_t^{\text{cond}}(g)
		=\PP{Y\neq g(V)\mid \norm{V}_p >t },
	\end{equation}
	Let $\cG$ be a pre-defined family of classifiers $g$ and let $L_t^{\text{cond}^*} = \inf_{g \in \cG} L_t^{\text{cond}}(g)$ be the smallest conditional classification risk for classifiers in $\cG$.
	The purpose is to learn from the training sample a classifier $\hat g \in\cG$ such that for large~$t$ the excess risk $L_t^{\text{cond}}(\hat g) - L_t^{\text{cond}*}$ is small.

	In \cite{jalalzai2018binary} it is argued that, asymptotically, attention can be restricted to \emph{angular classifiers}~$g$, that is, for which $g(x) = g(\theta_p(x))$ with $\theta_p(x) = x/\norm{x}_p$ for $x \in E$.
	Their analysis involves  a  random pair $(V_\infty, Y_\infty)$ whose distribution is the weak limit as $t \to \infty$ of $(t^{-1} V, Y)$ conditionally on $\norm{V}_p > t$, a limit which exists thanks to \eqref{eq:conditional_RV}.
	Let $\eta(x) = \PP{Y=1 \given V=x}$ denote the regression function of $(V, Y)$ and let $\eta_\infty(x) = \PP{Y_\infty = 1 \given V_\infty=x}$ be the one of $(V_\infty, Y_\infty)$.
	The respective Bayes classifiers are 
	\begin{equation}
		\label{eq:Bayesclassifiers}
		\begin{aligned}
			g^*(x) &= \un\{ \eta(x) \ge 1/2 \}, \\
			g_\infty^*(x) &= \un\{ \eta_\infty(x) \ge 1/2 \}.
		\end{aligned}
	\end{equation}
	Note that $g^*$ minimizes the conditional risk $L_t^{\text{cond}}$ for any $t > 0$.
	Assume that when $\norm{x}_p$ is large, $\eta(x)$ and $\eta_\infty(x)$ are uniformly close:
	\begin{equation}
		\label{eq:assumJalal_regressionFunction}   
		\sup_{x \in [0, \infty)^d: \norm{x}_p \ge t} \left|\eta(x) - \eta_\infty(x) \right|\to 0 \quad \text{ as } t\to\infty.
	\end{equation}
	Then by Theorem~1 in \cite{jalalzai2018binary}, (i) the asymptotic Bayes classifier $g_\infty^*$ is angular, and (ii) the latter's excess conditional risk over the actual Bayes classifier $g^*$ vanishes in the limit, that is, $L_t^{\text{cond}}(g_\infty^*) - L_t^{\text{cond}}(g^*) \to 0$ as $t \to \infty$.
	
	These properties motivate restricting the search to a class $\cG$ of candidate classifiers $g$ depending on the angle only. 
	Theorem~2 in \cite{jalalzai2018binary} provides a concentration inequality for the excess risk of the empirical risk minimizer $\hat{g}_k \in \cG$ learned from a sample $\{(V_i, Y_i)\}_{i=1}^n$, using only those points for which $\norm{V_i}_p$ belongs to a top fraction among those observed.
	Here, we intend to do the same but for the rank-based transformed feature vectors $\widehat{V}_i = \hat{v}(X_i)$ in \eqref{eq:Vhat}.
	\smallskip
	
\noindent
\textbf{Classification risk and angular measure.} \quad
	For $g$ in a class of angular classifiers $\cG$, recall the conditional classification risk $L_t^{\text{cond}}(g)$ above level $t$ in \eqref{eq:extreme-risk} and define its unconditional version by
	\begin{equation}
		\label{classifrisk-transfoDependent}
		L_t(g) 
		= t \PP{\norm{V}_p \ge t} L_t^{\text{cond}}(g) 
		= t \PP{g(V) \neq Y , \, \norm{V}_p \geq t}.
	\end{equation}
	The multiplicative factor $t \PP{ \norm{V}_p \ge t  } $ converges to $\Phi_p(\sphere_p)$ and does not change the minimizer in the class $\cG$. Working with the unconditional version $L_t(g)$ rather than with $L_t^{\text{cond}}(g)$ simplifies the analysis that follows.

	In view of Assumption~\ref{ass:A} required in  Section~\ref{sec:traditionalEmpirical}, we exclude from our empirical risk minimization (ERM) strategy those feature vectors whose angle (after standardization) is too close to the boundary of the unit sphere.
	Let $\tau \in (0, 1)$ and recall $\sphere_p^\tau$ in \eqref{eq:boundedfrom0}. 
	We have
	\begin{equation}
		\label{eq:Ltgtau}
		L_t(g) =  L_t^{>\tau}(g) +  L_t^{\le\tau}(g)
	\end{equation}
	with
	\begin{align*}
		L_t^{>\tau}(g) & =  t \PP{g(V) \neq Y , \, \theta_p(V)\in\sphere_p^\tau , \, \norm{V}_p \geq t}, \\
		L_t^{\le\tau}(g) & = t \PP{g(V) \neq Y , \, \theta_p(V)\notin\sphere_p^\tau , \,  \norm{V}_p \geq t}.
	\end{align*}
	
	The regions of the sphere $\sphere_p$  labeled $+1$ and $-1$ by $g \in \cG$ are denoted by
	\[
	\sphere_p^{\sigma}(g) = \{ x \in \sphere_p : g(x) = \sigma 1\},
	\qquad \text{for $\sigma \in \{-,+\}$.}
	\]
	We work hereafter under the following smoothness assumption. Let $\partial A$ denote the boundary of set $A$.
	
	\begin{assumption}[Smoothness]
		\label{assum:smoothness}
		The scalar $\tau \in (0, 1)$ is such that $\Phi_p(\partial \sphere_p^\tau) = 0$ and the class $\cG$ is such that     $\Phi_p(\partial\sphere_p^+(g))= \Phi_p(\partial\sphere_p^-(g)) = 0$ for all $g\in\cG$.
	\end{assumption}
	
	\begin{lemma}
		\label{lem:limLt-v-angular classif}
		If the conditional regular variation property~\eqref{eq:conditional_RV} and Assumption~\ref{assum:smoothness} hold, then for any angular classifier $g\in\cG$,
		\begin{align*}
			\lim_{t \to \infty} L_t^{>\tau}(g) 
			&= 
			L_\infty^{>\tau}(g) 
			:= \varrho \Phi_p^+(\sphere_p^-(g) \cap \sphere_p^\tau) + (1-\varrho) \Phi_p^-(\sphere_p^+(g) \cap \sphere_p^\tau), \\
			\lim_{t\to\infty} L_t^{\le\tau}(g) 
			&=
			L_\infty^{\le\tau}(g) :=
			\varrho \Phi_p^+(\sphere_p^-(g) \setminus \sphere_p^\tau) + (1-\varrho) \Phi_p^-(\sphere_p^+(g) \setminus \sphere_p^\tau),
		\end{align*}
		and thus
		\[
		\lim_{t\to\infty} L_t(g) 
		= L_\infty(g) :=  \varrho \Phi_p^+(\sphere_p^-(g)) + (1-\varrho) \Phi_p^-(\sphere_p^+(g)). 
		\]
	\end{lemma}
	
	The proof is deferred to  Appendix~\ref{sec:appendix_classif}. The idea of the decomposition~\eqref{eq:Ltgtau} is to discard points with angle outside $\sphere_p^\tau$. If $\Phi_p$ is concentrated on the interior of $\sphere_p$, the corresponding loss term $L_t^{\le\tau}(g)$ can be expected to be small for $\tau$ close zero, since
	\[
	\sup_{g \in \cG} L_t^{\le\tau}(g)
	\le t \PP{\theta_p(V) \not\in \sphere_p^\tau, \, \norm{V}_p > t}
	\to \Phi_p(\sphere_p \setminus \sphere_p^\tau), \qquad \text{as $t \to \infty$}.
	\]
	\smallskip
	
\noindent
\textbf{ERM classifier and decomposition of the excess risk.} \quad
	Given $0 < \tau < 1$ and integers $1 < k\le  n$, define the  empirical risk of a classifier $g \in \cG$ by
	\begin{equation}
		\label{eq:emp-risk-extreme}
		\widehat{L}^\tau(g) 
		= \frac{1}{k} \sum_{i=1}^n \un\left\{ 
		g(\widehat{V}_i) \neq Y_i, 
		\, \theta_p(\widehat{V}_i)\in\sphere_p^\tau , 
		\, \norm{\widehat{V}_i}_p \geq n/k 
		\right\}.
	\end{equation}
	Assuming a minimizer exists, the ERM classifier is defined as
	\[
	\hat{g}_k^\tau \in \argmin_{g \in \cG} \widehat{L}^\tau(g).
	\]
	Otherwise, introduce a tolerance parameter and consider an approximate minimizer instead, \chg{i.e., an argument where the value of the objective function is close to the infimum.}
	
	Recall $L_\infty$ in Lemma~\ref{lem:limLt-v-angular classif}. 
	A consequence of Theorem~1 in \cite{jalalzai2018binary} is that if \eqref{eq:assumJalal_regressionFunction} holds, the Bayes classifier $g_\infty^*$ in \eqref{eq:Bayesclassifiers} minimizes $L_\infty$ over all measurable classifiers.
	One way to measure the performance of the ERM classifier $\hat{g}_k^\tau$ is via the asymptotic excess risk
	\[
	L_\infty(\hat{g}_k^\tau) - \inf_{g \in \cG} L_\infty(g).
	\]
	The latter can be bounded in terms of the supremum deviation of the empirical and asymptotic risks over $\cG$: since $\widehat{L}^\tau(\hat{g}_k^\tau)$ is equal to the infimum of $\widehat{L}^\tau(g)$ over $g \in \cG$, we have
	\begin{equation}
		\label{eq:excessRisk-decompose}
		L_\infty(\hat{g}_k^\tau) - \inf_{g \in \cG} L_\infty(g)
		\le 2 \sup_{g \in \cG} \left|\widehat{L}^\tau(g) - L_\infty(g)\right|.
	\end{equation}
	Our main purpose is therefore to obtain a concentration inequality for the supremum on the right-hand side of this inequality.
	
	In our context, the supremum deviation itself decomposes further, since for all $g\in\cG$, we have $L_\infty^{\le\tau}(g) \le \Phi_p(\sphere_p \setminus \sphere_p^\tau)$ and thus
	\begin{equation}
		\label{eq:decomposeRisk-tau}
		\sup_{g \in \cG}
		\left| \widehat L^\tau(g) - L_\infty(g) \right|
		\le \sup_{g \in \cG} \left| \widehat L^\tau(g) - L_\infty^{>\tau}(g)\right| + \Phi_p(\sphere_p \setminus \sphere_p^\tau) .
	\end{equation}
	The term $ \Phi_p(\sphere_p \setminus \sphere_p^\tau)$ may be viewed as an additional bias term which vanishes as $\tau\to 0$ provided $\Phi_p$ is concentrated on the interior of $\sphere_p$. On the other hand, the upper bounds in Theorem~\ref{thm:conc} and~Theorem~\ref{theo:deviationsClassif} grow roughly as $1/\sqrt{\tau}$ as $\tau\to 0$. The choice of $\tau$ thus constitutes an additional bias-variance compromise.
	
	Lemma~\ref{lem:empirical-risk-angularMeasure} in Appendix~\ref{sec:appendix_classif} parallels Lemma~\ref{lem:limLt-v-angular classif} by relating the empirical risk $\widehat{L}^\tau(g)$ to the empirical angular measures of the positive and negative instances, 
	\begin{equation}
		\label{eq:empiricalSignedAM}
		\widehat\Phi_p^\sigma(A) = \frac{1}{k_\sigma} \sum_{i=1}^n 
		\un\left\{Y_i = \sigma 1\right\} \cdot 
		\un\left\{\theta_p(\widehat{V}_i) \in A, \; \norm{\widehat{V}_i}_p \ge n / k\right\}, 
		\qquad A\subseteq\sphere_p, \sigma\in\{-,+\}, 
	\end{equation}
	where $k_\sigma = kn_\sigma/n$ and $n_\sigma = \sum_{i=1}^n \un\{Y_i = \sigma 1\}$ is the number of points such that $Y_i = \sigma 1$. 
	
	In view of the error decomposition~\eqref{eq:excessRisk-decompose}--\eqref{eq:decomposeRisk-tau}, we state our main result in terms of the maximum deviation $\sup_{g \in \cG}|\widehat{L}^\tau(g) - L_\infty^{>\tau}(g)|$, following the techniques from Section~\ref{sec:traditionalEmpirical}.
	
	\begin{theorem}[Deviations of the empirical tail risk]
		\label{theo:deviationsClassif}
		Let $\cG$ be a class of angular classifiers.
		Consider the collection $\mathcal{A} = \{\sphere_p^+(g) \cap\sphere_p^\tau : g \in\cG \} \cup  \{\sphere_p^-(g) \cap\sphere_p^\tau : g \in\cG \}$.
		If Assumptions~\ref{ass:A} and~\ref{ass:VF} relative to the class $\cA$ and the unconditional angular measure $\Phi_p$ are satisfied
		and if the conditional regular variation assumption~\eqref{eq:conditional_RV} and Assumption~\ref{assum:smoothness} hold, then, with probability at least $1 - \delta$,
		\begin{equation*}
			\sup_{g \in \cG}|\widehat{L}^\tau(g) - L_\infty^{>\tau}(g)| \leq 
			2(\operatorname{error} + \operatorname{bias\,II} + \operatorname{gap}) 
		\end{equation*}
		where $\operatorname{error}$ and $\operatorname{gap}$ are nearly the same as in Theorem~\ref{thm:conc} ($\delta$ has been halved in the $\operatorname{error}$ term), that is,
		\begin{align*}
			\operatorname{error}
			&=  \sqrt{\frac{d^{1+1/p}(1+\Delta_2)}{k}} \left( \cst \sqrt{V_\F} + 2 \sqrt{\log(2(d+1)/\delta)} \right) + \frac{2}{3k} \log(2(d+1)/\delta), \\
			\operatorname{gap}
			&= 2d\Delta_2 + 3c\big(\log(d/3c) - \log(\Delta_1) + 1 \big) \Delta_1,
		\end{align*}
		with $\Delta_1,\Delta_2$ and $V_{\cF}$ as defined in Theorem~\ref{thm:conc},  and 
		\begin{multline*}
			\operatorname{bias\, II}
			=
			\sup \Bigl\{ \left|\tfrac{n}{k} \PP{V \in \tfrac{n}{k} B, Y = \sigma 1 } - \PP{Y =\sigma 1}\mu_\sigma(B)\right| : \\
			\text{$B = \Gamma_{A}^+$ or $B = \Gamma_{A}^-$ for some $A \in \cA$ and $\sigma\in\{-,+\}$}
			\Bigr\}.
		\end{multline*}
	\end{theorem}
	
	The proof relies on the relationships between the (empirical) classification risks and the (empirical) angular measure pointed out in Lemmata~\ref{lem:limLt-v-angular classif} and~\ref{lem:empirical-risk-angularMeasure}, which imply that
	\begin{align*}
		|\widehat{L}^\tau(g) - L_\infty^{>\tau}(g)| 
		& \le 
		\Bigl| \frac{n_+}{n} \widehat \Phi_p^+(\sphere_p^-(g) \cap \sphere_p^\tau) - \varrho \Phi_p^+(\sphere_p^-(g) \cap \sphere_p^\tau) \Bigr|  \\
		&\quad \mbox{} +
		\Bigl| \frac{n_-}{n} \widehat \Phi_p^-(\sphere_p^+(g) \cap \sphere_p^\tau) - (1-\varrho) \Phi_p^-(\sphere_p^+(g) \cap \sphere_p^\tau) \Bigr|    
	\end{align*}
	for $g \in\cG$.
	The right-hand side of the latter display is then bounded uniformly in $g \in \cG$ by adapting the proof of Theorem~\ref{thm:conc}; see Appendix~\ref{sec:appendix_classif} in the Supplement for  details.
	
	\section{Simulation experiments} 
	\label{sec:experiments}
	
	Our experiments aim at illustrating the influence of the threshold $\tau$ introduced in Equation~\eqref{eq:boundedfrom0} on the supremum error $\sup_{A\in\mathcal{A}_p} |\hPhi_p(A) - \Phi_p(A)|$. 
	For a simple class of sets $\cA_p$, we will demonstrate empirically that the supremum error increases as $\tau$ decreases. This finding suggests that this additional parameter is not a mere artifact from our proof.
	
	We report the results of Monte Carlo experiments on simulated data. The setting is such that $\Phi_p(A)$ can be approximated with arbitrary precision by Monte Carlo sampling, the standardization $\Vf$ to unit-Pareto margins can be computed analytically, and the bias term in the upper bounds of Theorems~\ref{thm:conc} 
	is zero.
	The estimation error then only stems from the stochastic error and framing gap in \eqref{eq:dec3} leading to the terms on the second and third lines in \eqref{eq:conc1}. 
	
	The experiments were implemented in Python~3 using the packages \textsf{numpy}, \textsf{scipy}, \textsf{matplotlib}, and \textsf{scikit-learn}.
	The computer code to reproduce the experiments is publicly available online.
	\footnote{\url{https://github.com/Hamid-Jalalzai/}} 
	\smallskip
	
\noindent
\textbf{Experimental setting.} \quad
	We consider angular measures with respect to the $L_p$-norm for 
	$p \in \{1,2, \infty\}$ and we limit ourselves to dimensions $d \in \{2,\ldots,5\}$, higher dimensions requiring much more computational effort to evaluate the supremum error on the class of sets described below.
	
	The different classes $\cA_p$ for different values of $p$ are all obtained by projection of a single class $\mathcal{A}$ defined on the $L_\infty$-sphere. Namely, let $\theta_p(x) = \|x\|_p^{-1} x$ for $x\in\rset^d\setminus\{0\}$. Recall $\sphere_p^\tau = \sphere_p \cap (\tau,\infty)^d$. Fix $0 < \tau < 1$. For a class $\cA$ of sets on $\sphere_\infty^\tau$, we define $\cA_p = \{\theta_p(A)\cap \sphere_p^\tau: A \in \cA\}$ for $p \in [1, \infty]$.
	We choose the class $\cA$ on $\sphere_\infty^\tau$ as the finite collection of hyper-rectangles forming a regular grid on $\sphere_\infty^\tau$ 
	with side length $h = (1-\tau)/S$ with $S = 10$.
	Each set $A \in \cA$ is of the form $A = A_1 \times \cdots \times A_d$, where, for some $j_0 \in \{1,\ldots,d\}$ and some integer vector $(i_j)_{j \in \{1,\ldots,d\} \ne j_0 }\in \{0, 1, \ldots, S-1\}^{d-1}$, we have
	\[
	A_j =
	\begin{dcases}
		\{ 1 \} & \text{if $j = j_0$}, \\
		(\tau + i_j h, \tau + (i_j+1)h) & \text{if $j \ne j_0$}.
	\end{dcases}
	\]

	We consider an independent random sample $X_1,\ldots,X_n$ drawn from the distribution of a random vector $X = R\Theta$ where $R$ and $\Theta$ are independent, the random variable $R$ follows a unit-Pareto distribution on $[1, \infty)$ and $\Theta$ follows a symmetric Dirichlet distribution on the unit simplex $\sphere_1 = \{ x \in [0, 1]^d : x_1+\cdots+x_d = 1\}$ with parameter $(\nu, \ldots, \nu)$ for some concentration parameter $\nu>0$. The larger $\nu$, the stronger $\Theta$ is concentrated around the barycenter $(1/d,\ldots,1/d)$.
	\chg{As detailed in Lemma~\ref{lem:explicit-V-Phi} in the Supplement, the angular measure $\Phi_p$ for $p \in [1, \infty]$ is $\Phi_p(A) = d \PP{X \in \cone_A}$ for Borel sets $A \subseteq \sphere_p$. If $p = 1$, then this simplifies to $\Phi_1(A) = d \PP{\Theta \in A}$.}
	
	In this setting, all statistics involved in our analysis are easily computable. The following facts are an immediate consequence of Lemma~\ref{lem:explicit-V-Phi} in the Supplement.
	Let $A \subseteq \sphere_p^\tau$ be a Borel set with $\tau \in (0, 1)$ and recall $\sphere_p^\tau$ in \eqref{eq:boundedfrom0}.
	
	\begin{itemize}
		\item
		The true $\Phi_p(A)$ may be approached with arbitrary precision by the Monte Carlo estimator
		\begin{equation}
			\label{eq:Phi-MC}
			\Phi_{p,\text{MC}}(A)
			= \frac{d}{N} \sum_{i=1}^N \un\left\{ X_i' \in \cone_A \right\}
			= \frac{d}{N} \sum_{i=1}^N \un\left\{ \Theta_i'/\norm{\Theta_i'}_p \in A ,\, R_i' \norm{\Theta_i'}_p \ge 1 \right\},
		\end{equation}
		where $X_i' = (R_i', \Theta_i')$ for $i \in \{1,\ldots,N\}$ is another independent random sample from the distribution of $(R, \Theta)$. The Monte Carlo estimator is unbiased and has variance bounded by $d^2/(4N)$.
		\item We have
			$
			\Phi_p(A) = \frac{n}{k}\PP{ V \in \frac{n}{k}\cone_A}  
			$
		as soon as  $n/k > d/\tau$, so that the bias term in Theorem~\ref{thm:conc} 
		is null under the latter condition.
	\end{itemize}
	\smallskip
	
\noindent
\textbf{Results.} \quad
	We choose a sample size of  $n=10^4$ and we let $k= 100$ in dimension $d\in\{2,\ldots,5\}$. 
	The Monte Carlo parameter $N$ in \eqref{eq:Phi-MC} is set to  $10^7$. The Dirichlet concentration parameter is chosen as  $\nu = 1/10$ so that the angular measure is concentrated near the boundaries of the positive orthant.
	
	Figure~\ref{fig:error_tau} displays the supremum errors $\sup_{A\in\cA_p}|\hPhi_p(A) - \Phi_p(A)|$  averaged over $500$  independent replications as a function of the parameter $\tau$.
	The latter varies in the range $[dk/n,0.1]$ in line with Lemma~\ref{lem:explicit-V-Phi}, for the reason explained above.
	The average errors of all three estimators  $\hPhi_p$ decrease for larger values of $\tau$ in line with our theoretical findings. Note that the errors also decrease when the dimension increases. This is a direct consequence of the definition of the class $\A$, forming a rectangular grid on the sphere $\sphere_\infty^\tau$ consisting of $d \times 10^{d-1}$ subsets (each face requires $10^{d-1}$ rectangles to be covered) so that the grid becomes finer as the dimension increases and the $\Phi_p$-mass of the sets $A$ considered in the supremum error is decreasing. 
	
	\begin{figure}
		\begin{subfigure}[b]{0.32\textwidth}
			\centering 
			\includegraphics[width=\textwidth]{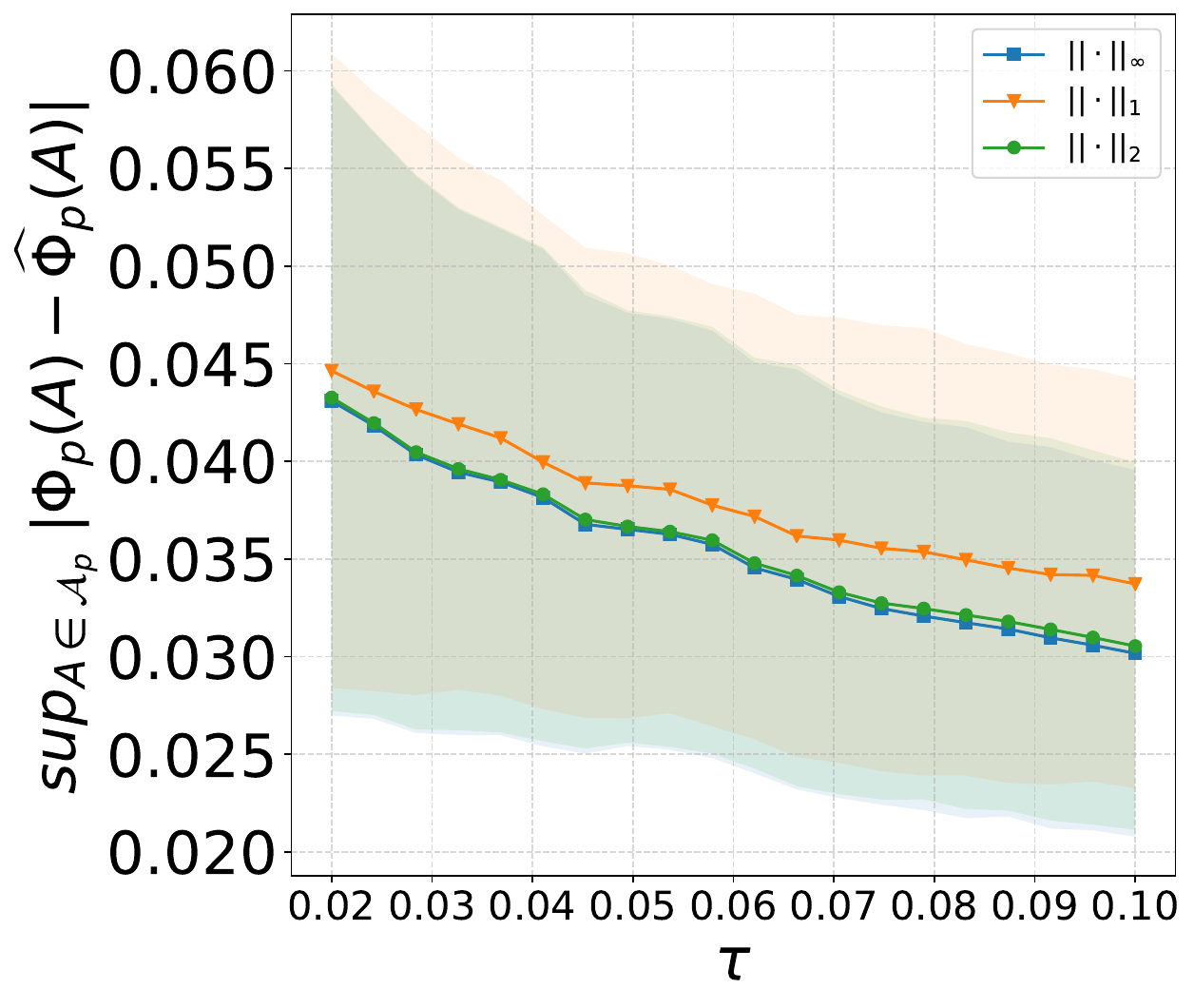}
			\caption{$d=2$}
		\end{subfigure}~
		\begin{subfigure}[b]{0.32\textwidth}
			\centering 
			\includegraphics[width=\textwidth]{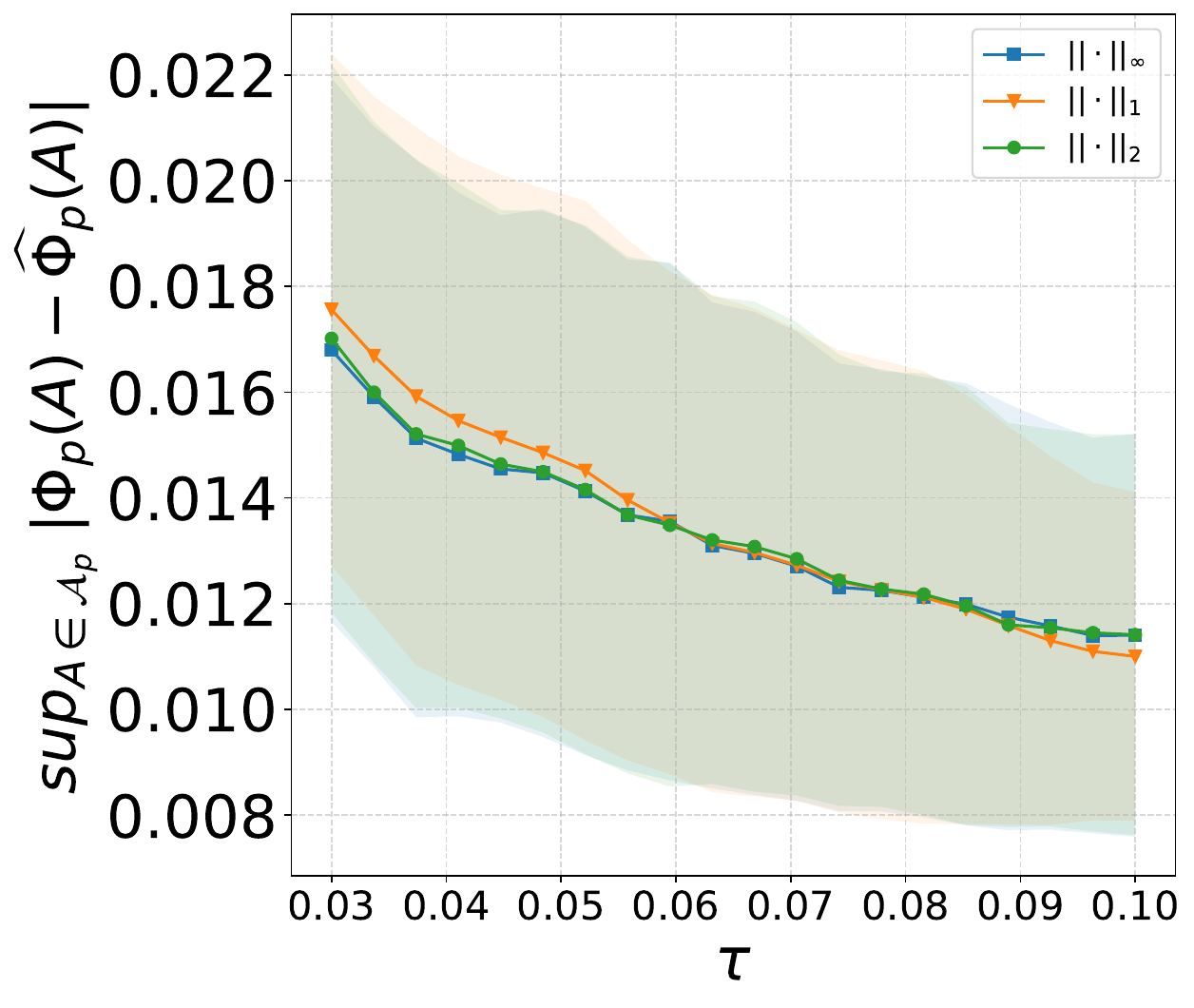}
			\caption{$d=3$}
		\end{subfigure}
		\\
		\begin{subfigure}[b]{0.32\textwidth}
			\centering 
			\includegraphics[width=\textwidth]{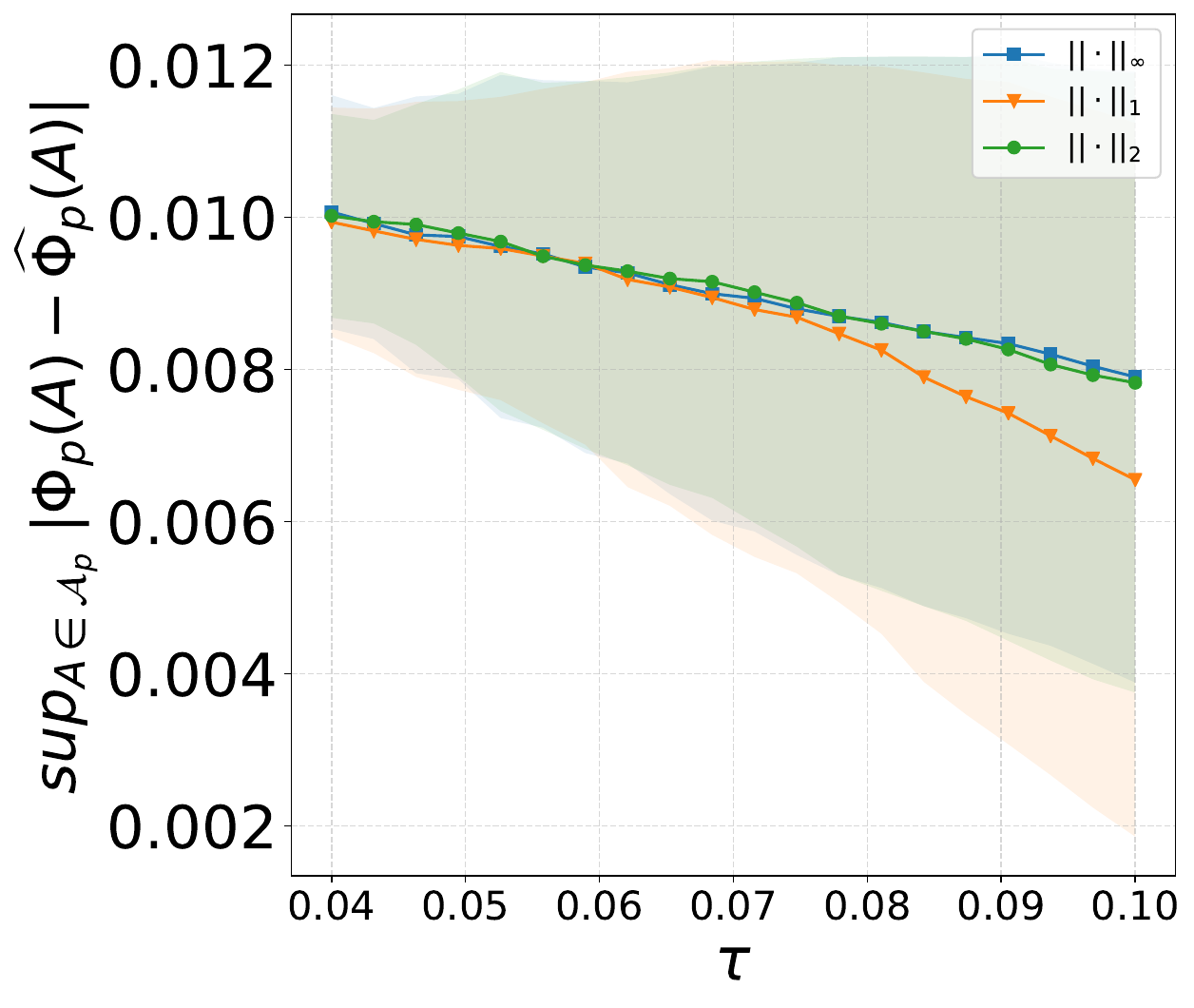}
			\caption{$d=4$}
		\end{subfigure}~
		\begin{subfigure}[b]{0.32\textwidth}
			\centering 
			\includegraphics[width=\textwidth]{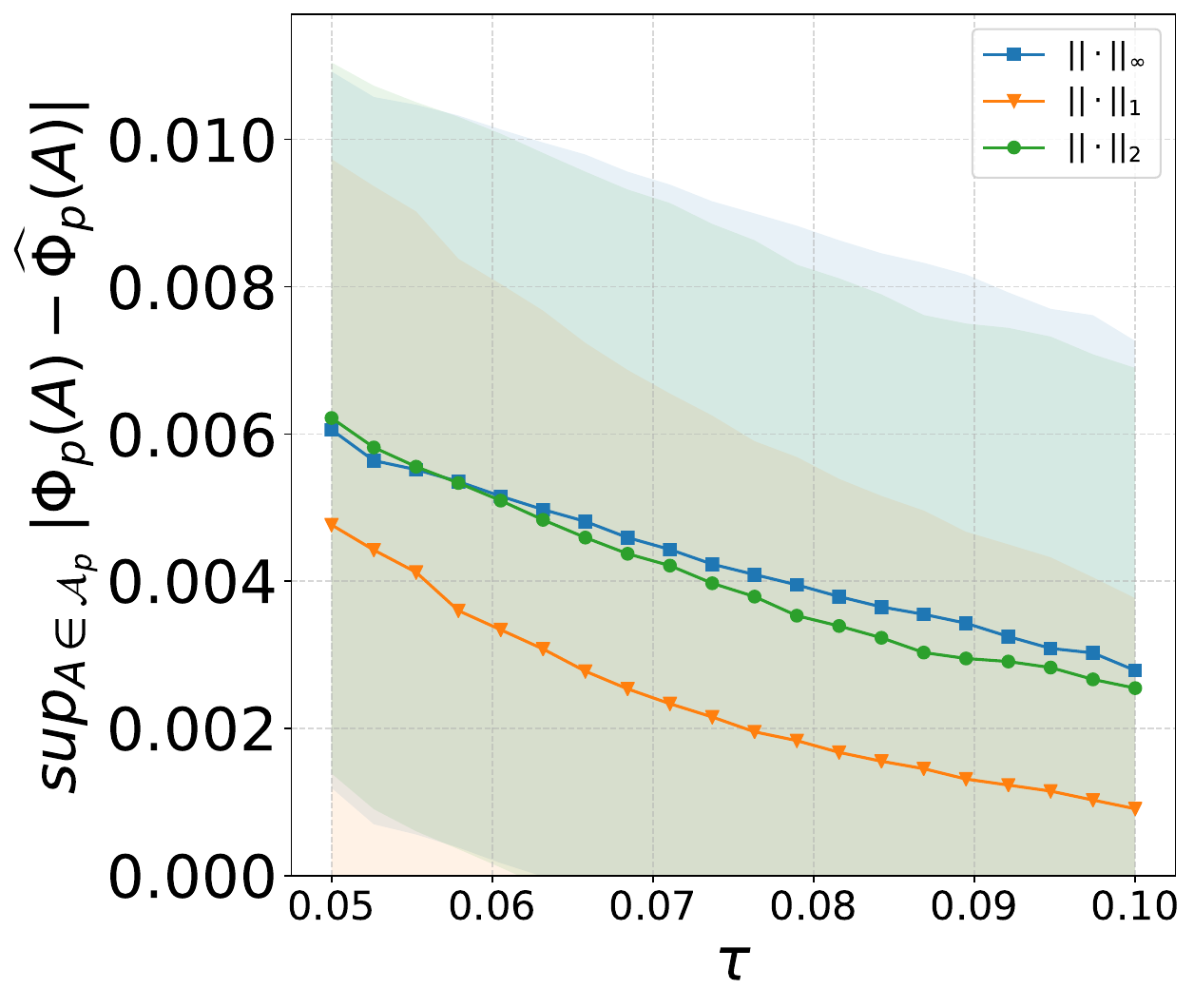}
			\caption{$d=5$}
		\end{subfigure}
		\caption{Average errors $\sup_{A\in\mathcal{A}}|\hPhi_p(A) - \Phi_p(A)|$ with $p \in \{1, 2, \infty\}$, in dimension $d=2$ (top left panel), $d=3$ (top right panel), dimension $d=4$ (bottom left panel) and $d=5$ (bottom right panel), as a function of $\tau$.  Colored intervals represent standard deviations of errors over $500$ independent replications.}
		\label{fig:error_tau}
	\end{figure}
	
	Of particular interest is the behaviour of the different norms. The error associated to the $L_2$-norm and the $L_\infty$-norm behave similarly while the one linked to the $L_1$-norm seems quite different. Moreover, the nature of the difference changes with the dimension. An explanation of these phenomena is depicted in Figure~\ref{fig:pizza&tau}. Since the estimation error of $\hPhi_p$ is measured with a supremum and increases while $\tau$ decreases, it suggests that this supremum is realized near the boundary of the considered subset $\sphere_p^\tau$ of the sphere, i.e., near the intersections with the coordinate axes. In this region, the $L_2$-sphere and the $L_\infty$-sphere are close to each other, the latter being even tangent to the former in dimension $d=2$. The $L_1$-sphere is quite different there since it forms a $45^\circ$ angle with the $L_\infty$-sphere. This explains why the error behaves so differently when the $L_1$-norm is considered. The fact that it becomes smaller in higher dimension can be understood as follows. Two forces play an opposite role in making the $L_1$-sphere different from the others: on the one hand, as illustrated in the left panel of Figure~\ref{fig:pizza&tau}, censuring coordinates smaller than $\tau$ has a higher impact on the $L_1$-sphere since it looses more volume than the others, on the other hand, as illustrated in the right panel, the cones $\mathcal{C}_A$ generated by the projection of a set $A \subseteq \sphere_\infty^\tau$ are larger for $p=1$ than for $p=2$ or $p=\infty$. The first force tends to reduce the $\Phi_1$-mass of the $L_1$-sphere while the second tends to increase it. Following the results of Figure~\ref{fig:error_tau}, the latter is dominated by the former when the dimension increases.
	
	\begin{figure}
		\centering
		\includegraphics[width=0.32\textwidth]{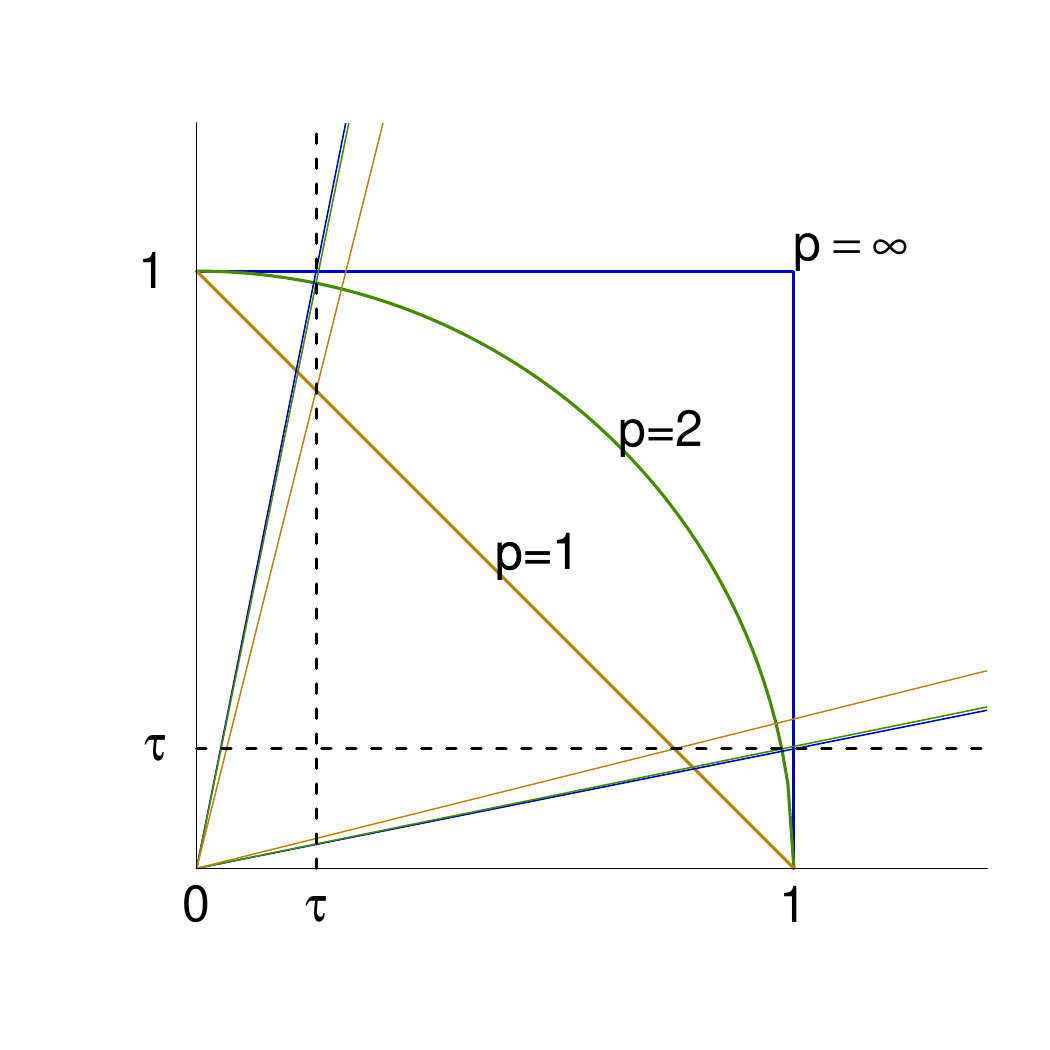}
		\includegraphics[width=0.32\textwidth]{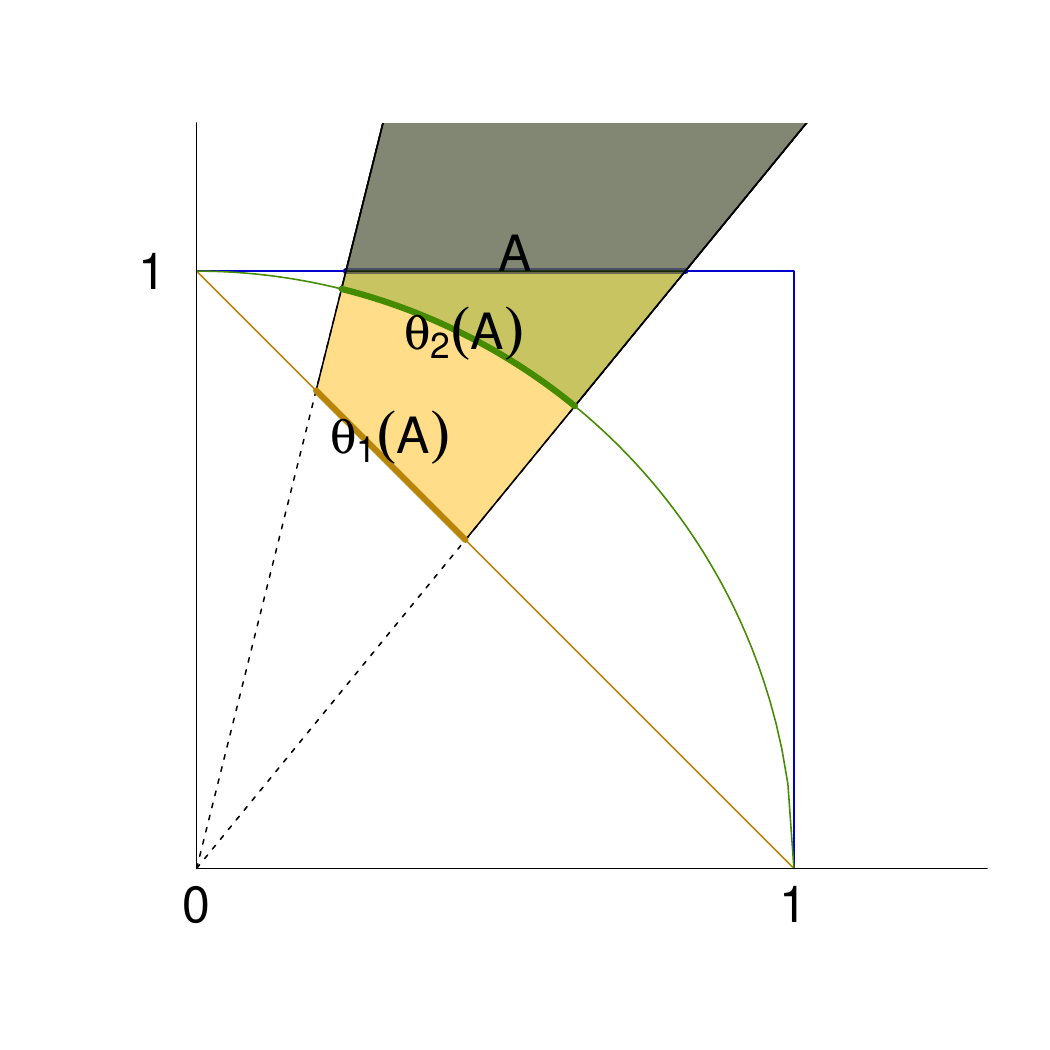}
		\caption{Restriction to coordinates larger than $\tau$ has an heavier impact on the $L_1$-sphere than on the $L_2$-sphere and the $L_\infty$-sphere, which touch each other near the coordinate axes (left panel). The cone $\{ x \in [0, \infty)^p : \norm{x}_p \ge 1, x \in \theta_p(A)\}$ generated by the projection $\theta_p(A)$ of a set $A \subseteq \sphere_\infty^\tau$ onto $\sphere_p^\tau$ is largest when the $L_1$-norm is considered (right panel).}
		\label{fig:pizza&tau}
	\end{figure}
	
	\section{Conclusion}
	\label{sec:discussion}
	
	We derived non-asymptotic guarantees in the form of concentration inequalities for the rank-based empirical angular measure with respect to any $L_p$-norm, $p \in [1,\infty]$.
	The bounds are valid in any dimension and concern the supremum error over certain collections of subsets of the $L_p$-sphere.
	Apart from a logarithmic term, the bounds match the convergence rate known in the bivariate case \cite{einmahl2001nonparametric,einmahl2009maximum}.
	Two applications to statistical learning based on observations in extreme regions were worked out: minimum-volume set estimation and binary classification.
	
	It would be interesting to be able to complement our upper bounds with lower bounds on the estimation error.
	Sharp bounds would provide guidance on the choice of the threshold parameter $k$, ideally in an adaptive manner as in \cite{boucheron2015tail}.
	
	The results are limited to subsets of the relative interior of the unit sphere to avoid non-extreme components, which are difficult to manage. Allowing for sets touching the boundary constitutes an important but challenging avenue left for further research. A numerical experiment has been provided to illustrate the influence of this restriction on the estimation error, for different norms and  dimensions. 
	
	The application to classification was limited to two balanced classes.
	Extensions to multiple classes and unbalanced situations can be developed in the same way.
	Of an altogether different nature, however, is the prediction of a continuous response from observations in extreme regions.
	The latter would require concentration inequalities for the empirical angular measure evaluated on collections of functions more general than set indicators. This topic has not yet even been broached in the bivariate case.


\begin{appendix}
	
	\section{Concentration inequality for rare events}
	\label{sec:concrare}
	
	The main concentration tool that we use in the proof of Theorem~\ref{thm:conc} is the following. Since the result may be of independent interest, we provide a detailed proof.
	
	\begin{theorem}
		\label{thm:chaining}
		Let $P_n$ denote the empirical distribution of an independent random sample $\xi_1, \ldots, \xi_n$ from a distribution $P$ on some measurable space $\mathcal{X}$. Let $\mathcal{A}$ be a VC-class of measurable subsets of $\mathcal{X}$ with VC-dimension $V_{\mathcal{A}}$. Let $B$ be a measurable subset of $\mathcal{X}$ containing $\bigcup_{A \in \mathcal{A}} A$ and write $\kappa = P(B)$. Then, for all $\delta \in (0, 1)$, there exists an event with probability at least $1 - \delta$ on which we have
		\[
		\sup_{A \in \mathcal{A}} \abs{P_n(A) - P(A)} \le \sqrt{\frac{\kappa}{n}} \left( \cst \sqrt{V_\cA} + 2 \sqrt{\log(1/\delta)} \right) + \frac{2}{3n} \log(1/\delta).
		\]
	\end{theorem}
	
	In~\cite[Theorem~1]{GoixCOLT}, a similar inequality is derived, but with a generic constant $C$ that is not made explicit. Just before their Lemma~14, there is a reference to \cite{Kolt06} providing bounds on the expectation of a symmetrized supremum, but from that source, the value of the constant looks nearly impossible to trace. We follow an alternative route to obtain that value via Theorems~1.16--17 in~\cite{lugosi:2002}, giving an explicit bound for the expectation of a symmetrized supremum in terms of an integral over covering numbers of the class of sets. In turn, the covering numbers of such a class can be bounded in terms of its VC-dimension. In this way, the constant $C$ can be made explicit. Its value is most likely not optimal but, should a sharp value for the constant be found in the future, it can be substituted in our result.
	
	We now give the proof of Theorem~\ref{thm:chaining}. It is based on a McDiarmid's concentration inequality for the supremum around its expected value in \cite{mcdiarmid1998concentration}, extending Bernstein's classical inequality and recalled in~\cite[Proposition~11]{GoixCOLT}. We rewrite it in a form which is more convenient for us~\cite[Proposition~5]{lhaut2021uniform}.
	
	\begin{proposition}
		\label{prop:mcDiarmid}
		Under the assumptions of Theorem~\ref{thm:chaining}, for all $\delta \in (0, 1)$, there exists an event with probability at least $1 - \delta$ on which we have
		\[
		\sup_{A \in \mathcal{A}} \abs{P_n(A) - P(A)} \le 2\sqrt{\frac{\kappa}{n}\log(1/\delta)} + \frac{2}{3n} \log(1/\delta) + \EE \left[ \sup_{A \in \mathcal{A}} \abs{P_n(A) - P(A)} \right]. 
		\]
	\end{proposition}
	
	\begin{proof}[Proof of Theorem~\ref{thm:chaining}]
		Apply Proposition~\ref{prop:mcDiarmid}.
		To bound the remaining expectation, we make use of Theorem~1.16 in~\cite{lugosi:2002}, which relies on a chaining argument to ensure that
		\begin{equation}
			\label{eq:chaining lugosi}
			\EE \left[ \sup_{A \in \mathcal{A}} \abs{P_n(A) - P(A)} \right] \leq \frac{24}{\sqrt{n}} \max_{x_1,\ldots,x_n \in \mathcal{X}} \int_0^1 \sqrt{\log\left(2\mathcal{N}\left(r,\cA(x_1^n)\right)\right)} \, \diff r,
		\end{equation}
		where $\mathcal{N}(r,\cA(x_1^n))$ is the \emph{covering number} (the smallest number of balls of radius $r$ needed to cover a set) of the set $\cA(x_1^n) := \left\{ (\1_A(x_i))_{i=1}^n : A \in \cA \right\} \subseteq \{0,1\}^n$ with respect to the metric $\rho(b, c) = n^{-1/2} \norm{b-c}_2$ for $b, c \in \{0, 1\}^n$, see~\cite[page~29]{lugosi:2002} for definitions and notation. Haussler~\citep{haussler1995} showed that, if $\cA$ has a finite VC-dimension, then the associated covering number is bounded in the following way: for any $0 < r \le 1$,
		\begin{equation}
			\label{eq:haussler}
			\mathcal{N}\left(r,\cA(x_1^n)\right) \le e(V_\cA +1)\left( \frac{2e}{r^2}\right)^{V_\cA}.
		\end{equation}
		Using~\eqref{eq:haussler} in~\eqref{eq:chaining lugosi} we can bound the integral as follows: for any points $x_1,\ldots,x_n \in \mathcal{X}$, we have
		\begin{align*}
			\int_0^1 \sqrt{\log\left(2\mathcal{N}\left(r,\cA(x_1^n)\right)\right)} \, \diff r
			&\le \int_0^1 \sqrt{\log\left(2e(V_\cA +1)\right) + V_\cA \log\left(2e/r^2\right)} \, \diff r \\
			&\le \sqrt{V_\cA} \int_0^1 \sqrt{3\log 2 + 2 - 2 \log r} \, \diff r 
			\chg{\le 2.44 \sqrt{V_{\cA}},}
		\end{align*}
		\chg{by numerical integration or by expressing the integral in terms of the standard normal cumulative distribution function.}
		Combining this result and~\eqref{eq:chaining lugosi}, we get
		\begin{equation}
			\label{eq:vcChaining}
			\chg{\EE \left[ \sup_{A \in \mathcal{A}} \abs{P_n(A) - P(A)} \right] \leq 59 \sqrt{\frac{V_\A}{n}}.}
		\end{equation}
	
		This bound is of the right order but does not use the extreme nature of the events in the class $\cA$. To this end, we use the \emph{conditioning trick} in~\cite[Lemma~2]{lhaut2021uniform}, which states that if $K = \sum_{i=1}^n \un \{X_i \in B\}$, \chg{with $B \supseteq \bigcup_{A \in \cA} A$ as in the statement of the theorem}, we have the distributional equality
		\[
		\left[(P_n(A))_{A \in \cA} \, \Big| \, K = k \right] \stackrel{d}{=} \left(\frac{k}{n} P_k^Y(A)\right)_{A \in \cA},
		\]
		where $P_k^Y$, for $k \in \{1,\ldots,n\}$, is the empirical measure associated to an independent random sample $Y_1,\ldots,Y_k$ from the conditional distribution $P(\, \cdot \mid B)$. It easily follows that~\cite[Lemma~6]{lhaut2021uniform}
		\[
		\EE \left[ \sup_{A \in \mathcal{A}} \abs{P_n(A) - P(A)} \right] 
		\le \EE \left[ \frac{K}{n} \EE \left[ \sup_{A \in \mathcal{A}} \left| P_K^Y(A) - P(A \mid B) \right| \mathrel{\Big|} K \right] \right] + \sqrt{\frac{\kappa}{n}}.
		\]
		The conditional expectation on the right-hand side is bounded by means of~\eqref{eq:vcChaining}, giving
		\begin{align*}
			\EE \left[ \frac{K}{n} \EE \left[ \sup_{A \in \mathcal{A}} \left| P_K^Y(A) - P(A \mid B) \right| \, \Big| \, K \right] \right]
			\le \EE \left[ \frac{K}{n} \cdot 59 \sqrt{\frac{V_{\mathcal{A}}}{K}} \right] 
			&\le \frac{59}{n} \sqrt{V_\cA} \sqrt{\EE[K]} 
			\le 59 \sqrt{\frac{\kappa}{n}V_\cA}
		\end{align*}
		in view of Jensen's inequality and $\EE[K] = n \kappa$.
		Combining everything and using the fact that $V_\A \geq 1$, we get the result. The proof is complete.
	\end{proof}
	\end{appendix}
	
	\begin{supplement}
		\stitle{\chg{Supplement to ``Concentration bounds for the empirical angular measure with statistical learning applications''}}
		\sdescription{\chg{The supplement contains some auxiliary lemmas and proofs of all results developed in the paper: main theorem, examples, binary classification application and simulations.}}
	\end{supplement}

\begin{acks}[Acknowledgments]
	\chg{
	The authors would like to thank two anonymous Referees for helpful comments and suggestions on an earlier version of the paper.  A significant part of the work done by Anne Sabourin on this project was accomplished while she was an associate professor at Télécom Paris.}
\end{acks}

	
	\bibliographystyle{imsart-number} 
	\bibliography{biblio}       
	

%

\newpage

\begin{center}
	\vspace{0.8cm}
	{\huge{\bf  Supplement}}
\end{center}
\vspace{0.5cm}

This supplement contains mathematical details related to in the paper \cite{clem2022}. Appendix~\ref{sec:app:auxiliary} contains auxiliary results. The proof of the main result in the paper, Theorem~\ref{thm:conc}, is given in Appendix~\ref{sec:proof:conc}. Appendix~\ref{sec:appendixBias} provides a practical criterion to control the bias term in Theorem~\ref{thm:conc} and applies it to the multivariate Cauchy distribution. In Appendix~\ref{sec:appendixExamples}, it is verified that the examples of collections of sets $\cA$ in Section~\ref{subsec:examples} in the paper indeed satisfy Assumptions~\ref{ass:A} and~\ref{ass:VF}. The proofs of the results in Section~\ref{sec:appliClassif} on classification in extreme regions are given in Appendix~\ref{sec:appendix_classif}, while Appendix~\ref{sec:app:simu} contains an additional result related to the simulation experiments in Section~\ref{sec:experiments} in the paper.

\begin{appendix}
	
	\setcounter{section}{1}
	\setcounter{equation}{44}
	
	\section{Auxiliary results}
	\label{sec:app:auxiliary}
	
	We will use the following minor results in the proof of Theorem~\ref{thm:conc}.
	
	\begin{lemma}
		\label{lem: radialframing}
		Let $x,v \in [1,\infty)^d$, $p \in [1,\infty]$ and $h \in [0,1/2)$. If
		\[
		\forall j \in \{1,\ldots,d\}, \qquad \left\vert \frac{x_j}{v_j} - 1 \right\vert \leq h,
		\]
		then
		\[
		\left\vert \frac{\norm{v}_p}{\norm{x}_p} - 1 \right\vert \leq \frac{h}{1-h} =: \Delta < 1,
		\]
		which implies 
		\[
		\forall r>0: \norm{x}_p \geq \frac{r}{1-\Delta} \implies \norm{v}_p \geq r \implies \norm{x}_p \geq \frac{r}{1+\Delta}.
		\]
	\end{lemma}
	
	\begin{proof}
		By hypothesis, for every $j \in \{1,\ldots,d\}$, we have
		\[
		1 - h \leq \frac{x_j}{v_j} \leq 1 + h,
		\]
		hence
		\[
		\frac{1}{1+h} \leq \frac{v_j}{x_j} \leq \frac{1}{1-h}
		\]
		and we deduce
		\[
		\left\vert \frac{v_j}{x_j} - 1 \right\vert \leq \frac{h}{1-h}. 
		\]
		In particular, 
		\[
		\forall j \in \{1,\ldots,d\}: \qquad |v_j - x_j| \leq \frac{h}{1-h} x_j.
		\]
		Taking the $p$-th power on each side, summing over $j$ and taking the $p$-th square root leads to
		\[
		\left\vert \|v\|_p - \|x\|_p \right\vert \leq \|v-x\|_p \leq \frac{h}{1-h} \|x\|_p.
		\]
		This shows the first part of the lemma. To obtain the second part, just note that the first part guarantees that
		\[
		\|x\|_p (1-\Delta) \leq \|v\|_p \leq \|x\|_p(1+\Delta),
		\]
		Those inequalities imply the second statement.
	\end{proof}
	
	\begin{lemma}
		\label{lem: vjtoPj}
		For every $j \in \{1,\ldots,d\}$ and $x_j > 0$, we have
		\[
		\left\vert \frac{x_j}{\hat{v}_j(x_j)} - 1 \right\vert \leq \left\vert x_j P_{n,j}((x_j,\infty)) - 1 \right\vert + \frac{|x_j-1|}{n}.
		\]
	\end{lemma}
	
	\begin{proof}
		By definition of the transform $\hat{v}$, we have for every $j \in \{1,\ldots,d\}$ and $x_j>0$:
		\[
		\hat{v}_j(x_j) = \frac{1}{1-\frac{n}{n+1}\hat{F}_j(x_j)} = \frac{n+1}{nP_{n,j}((x_j,\infty))+1}. 
		\]
		Therefore, by the triangle inequality,
		\begin{align*}
			\left\vert \frac{x_j}{\hat{v}_j(x_j)} - 1 \right\vert
			&= \left\vert \frac{x_j \left( nP_{n,j}((x_j,\infty))+1 \right)}{n+1} -1 \right\vert = \left\vert \frac{x_j \left( nP_{n,j}((x_j,\infty))+1 \right) - (n+1)}{n+1} \right\vert \\
			&= \left\vert \frac{n \left( x_j P_{n,j}((x_j,\infty)) - 1 \right) + (x_j - 1)}{n+1} \right\vert \\
			&\leq \left\vert x_j P_{n,j}((x_j,\infty)) - 1 \right\vert + \frac{|x_j-1|}{n}.
		\end{align*}
	\end{proof}

	\begin{lemma}
		\label{lem:theta}
		Let $\norm{\,\cdot\,}$ be a norm on a real vector space and write $\theta(z) = z/\norm{z}$ for non-zero vector $z$. For non-zero vectors $x$ and $y$, we have
		\[
		\norm{\theta(x) - \theta(y)}
		\le 2 \frac{\norm{x-y}}{\norm{x} \vee \norm{y}}.
		\]
	\end{lemma}
	
	\begin{proof}
		Since $\theta(\,\cdot\,)$ is scale-invariant, we can divide both $x$ and $y$ by $\norm{x} \vee \norm{y}$ without changing the two sides of the inequality. For the sake of the proof we can thus assume that $\norm{x} = 1 \ge \norm{y} > 0$, in which case we need to show that
		\[
		\norm{x - \theta(y)} \le 2 \norm{x - y}.
		\]
		By the triangle inequality, we have
		\[
		\norm{x - \theta(y)}
		\le \norm{x - y} + \norm{y - \theta(y)}
		\]
		and
		\[
		\norm{y - \theta(y)} = \left|1 - \dfrac{1}{\norm{y}}\right| \cdot \norm{y}
		= |\norm{y} - 1| = |\norm{y} - \norm{x}| \le \norm{y-x}.
		\qedhere
		\]
	\end{proof}
	
	\section{Proof of Theorem~\ref{thm:conc}}
	\label{supp:proofs:main}
	
	\begin{proof}[Proof of Theorem~\ref{thm:conc}]
		\label{sec:proof:conc}
		
		The empirical exponent and angular measures $\hmu$ and $\hPhi_p$ only depend on the sample $X_1, \ldots, X_n$ through the transformed data
		\[
		\hF_j(X_{ij})
		= \frac{1}{n} \sum_{t=1}^n \mathds{1}\{X_{tj} \le X_{ij}\}
		\] 
		for $i = 1, \ldots, n$ and $j = 1, \ldots, d$. The sum of indicators is equal to the rank of $X_{ij}$ within $X_{1j}, \ldots, X_{nj}$. As the marginal cumulative distribution function $F_j$ is continuous, the ranks of $X_{1j}, \ldots, X_{nj}$ are with probability one equal to those of $V_{1j}, \ldots, V_{nj}$, where $V_{ij} = 1 / (1 - F_j(X_{ij}))$.
		Hence, even though the margins $F_1, \ldots, F_d$ are unknown, the fact that $\hmu$ and $\hPhi_p$ are rank statistics implies that for the sake of the proof, we may and will henceforth assume that the margins $F_1, \ldots, F_d$ are unit-Pareto.
		
		Abbreviate $\Gamma_{A}^\pm = \Gamma_{A}^\pm(r_\pm, 3\Delta_1)$ with $r_\pm = 1 \pm \Delta_2$ and $0 < \Delta_1, \Delta_2 < 1$ are as in the statement of the theorem. Since $r_- \le 1 \le r_+$, we have
		\[
		\forall A \in \cA, \qquad 
		\Gamma_{A}^- \subseteq \cone_A \subseteq \Gamma_{A}^+.
		\]
		We shall construct an event $\event_1$ with probability at least $1 -  d\delta/(d+1)$, on which we have
		\begin{equation}
			\label{eq:inclusions}
			\forall A \in \cA, \qquad 
			\tfrac{n}{k} \Gamma_{A}^- \subseteq \widehat{\Gamma}_A \subseteq \tfrac{n}{k} \Gamma_{A}^+.
		\end{equation}
		Then, on the event $\event_1$, the decomposition \eqref{eq:dec3} of the estimation error holds true.
		We treat the three terms involved in the decomposition in turn. At the end, we construct the event $\event_1$ with the required properties.

		\textbf{Bias term.} \quad
		Taking the supremum over $A \in \cA$ immediately yields the first term on the right-hand side of the bound~\eqref{eq:conc1}.

		\textbf{Stochastic error.} \quad
		We apply Theorem~\ref{thm:chaining} to the collection 
		\begin{equation}
			\label{eq:cF}
			\cF = 
			\left\{ 
			\tfrac{n}{k} \Gamma_{A}^{\sigma} :\; \sigma\in\{-,+\},\;  A \in \cA 
			\right\},
		\end{equation}
		which has finite VC dimension $V_\F$ in view of Assumption~\ref{ass:VF} and the paragraph following Theorem~\ref{thm:conc}; the rescaling by the factor $n/k$ obviously does not change the VC dimension.
		For every $A \in \cA$, we have
		\[
		\tfrac{n}{k} \Gamma_{A}^- \subseteq \tfrac{n}{k} \Gamma_{A}^+
		\subseteq \left\{ x \in [0, \infty)^d : \ \norm{x}_p \ge \tfrac{n}{k} \tfrac{1}{1+\Delta_2} \right\}.
		\]
		By the equivalence of norms~\eqref{eq: equiv norms}, since the margins of $P$ are unit-Pareto, it follows that the probability $\kappa$ appearing in Theorem~\ref{thm:chaining} applied to $\cF$ is bounded by
		\begin{equation}
			\label{eq:Fprobmax}
			P\left[\textstyle{\bigcup_{j=1}^d} \left\{
			x \in [0, \infty)^d : x_j \ge \tfrac{n}{k} \tfrac{1}{d^{1/p}(1+\Delta_2)} \right\}\right]
			\le d^{1+1/p}(1+\Delta_2)\tfrac{k}{n}.
		\end{equation}
		As a consequence, on an event $\event_2$ with probability at least $1 - \delta/(d+1)$, we have,
		\begin{multline*}
			\lefteqn{\sup_{A \in \A, \sigma \in \{-,+\}} \tfrac{n}{k} \left\vert P_n \left( \tfrac{n}{k}\Gamma_A^\sigma \right) - P \left( \tfrac{n}{k}\Gamma_A^\sigma \right) \right\vert} \\
			\leq \sqrt{\frac{d^{1+1/p}(1+\Delta_2)}{k}} \left( \cst \sqrt{V_\F} + 2 \sqrt{\log((d+1)/\delta)} \right) + \frac{2}{3k} \log((d+1)/\delta).
		\end{multline*}
		This is the second term on the right-hand side of the bound~\eqref{eq:conc1}.
		Since $\event_1$ and $\event_2$ have probabilities at least $1 - d\delta/(d+1)$ and $1 - \delta/(d+1)$, respectively, the event $\event_1 \cap \event_2$ has probability at least $1 - \delta$.

		\textbf{Framing gap.} \quad
		For $A \in \cA$, any point $x \in \Gamma_{A}^+ \setminus \Gamma_{A}^-$ has either a norm $\norm{x}_p$ between $r_-$ and $r_+$ or an angle $\theta_p(x)$ contained in a set of the form $A_+(\eps) \setminus A_-(\eps)$ for some $\eps > 0$. Specifically,
		\begin{equation}
			\label{eq:framingGapDecomposition}
			\begin{split}
				\mu(\Gamma_{A}^+ \setminus \Gamma_{A}^-) 
				&\le \mu\left(\left\{ x \in [0, \infty)^d : \ (1+\Delta_2)^{-1} \le \norm{x}_p \le (1-\Delta_2)^{-1} \right\}\right) \\
				&\quad + \mu\left(\left\{ x \in [0, \infty)^d : \ \norm{x}_p \ge 1, \; \theta_p(x) \in A_+(3\Delta_1 \norm{x}_p) \setminus A_-(3\Delta_1 \norm{x}_p) \right\}\right).
			\end{split}
		\end{equation}
		The first term on the right-hand side in \eqref{eq:framingGapDecomposition} is equal to
		\begin{equation*}
			\left( (1+\Delta_2) - (1-\Delta_2) \right) \mu(\{x \in [0, \infty)^d : \norm{x}_p \ge 1 \}) = 2 \Delta_2 \Phi_p(\sphere_p).
		\end{equation*}
		The second term on the right-hand side in \eqref{eq:framingGapDecomposition} can be computed via the product representation \eqref{eq:polar} of $\mu$ in polar coordinates: in view of \eqref{eq:Phi+-} in Assumption~\ref{ass:A}, the result is
		\begin{align*}
			\int_{1}^\infty  \Phi_p \bigl(
			A_+(3\Delta_1 r) \setminus A_-(3\Delta_1 r)
			\bigr) \frac{\diff r}{r^2} 
			&\le \int_{1}^\infty \min\left\{\Phi_p(\sphere_p), 3c \Delta_1 r\right\}  \frac{\diff r}{r^2} \\
			&= 3 c \Delta_1 \left(1 + \log \Phi_p(\sphere_p) - \log(3c \Delta_1)\right),
		\end{align*}
		since $\int_1^\infty \min(b, ar) \frac{\diff r}{r^2} = a (\log(b/a)+1)$ for $a \in (0, b]$ and $3c \Delta_1 \le 1 \le \Phi_p(\sphere_p)$ by assumption.
		
		The resulting upper bound in \eqref{eq:framingGapDecomposition} does not depend on $A \in \cA$.
		Bounding $\Phi_p(\sphere_p)$ by $d$, we obtain the third term on the right-hand side of the bound~\eqref{eq:conc1}.
		
		\textbf{Construction of the event $\event_1$.} \quad
		We still need to construct an event $\event_1$ with probability at least $1 - d \delta/(d+1)$ on which the inclusions \eqref{eq:inclusions} hold. To do this, we apply Theorem~\ref{thm:chaining} to each of the collections
		\[
		\cF_j = \left\{
		\{ x \in [0, \infty)^d : x_j > \tfrac{n}{k} y \} : \
		y \in [\rho, \infty)
		\right\},
		\qquad j = 1, \ldots, d.
		\]
		Fix $j = 1, \ldots, d$ and let $P_{n,j} = n^{-1} \sum_{i=1}^n \delta_{X_{ij}}$. Each set in the collection $\cF_j$ is a subset of $\{ x \in [0, \infty)^d : x_j > \tfrac{n}{k} \rho \}$, whose $P$-probability is $\kappa = \tfrac{k}{n} \rho^{-1}$. The class $\cF_j$ has VC dimension~$1$. By Theorem~\ref{thm:chaining}, there exists an event $\event_{1,j}$ with probability at least $1 - \delta/(d+1)$ on which
		\begin{align}
			\sup_{x_j \ge \frac{n}{k}\rho} \left|P_{n,j}((x_j, \infty)) - x_j^{-1}\right| 
			&\le \frac{k}{n} \left\{ \sqrt{\frac{1}{k\rho}} \left(56 + 2 \sqrt{\log((d+1)/\delta)} \right) + \frac{2}{3k} \log((d+1)/\delta) \right\} \nonumber \\
			&= \frac{k}{n} \Delta_1. \label{eq: kndelta1}
		\end{align}
		Since
		\[
		\hat{v}_j(x_j)
		= \frac{1}{1 - \frac{n}{n+1} P_{n,j}((-\infty, x_j])}
		= \frac{n+1}{n P_{n,j}((x_j, \infty)) + 1},
		\]
		we have on the event $\event_{1,j}$ the bounds
		\begin{equation}
			\label{eq:xnkrho:bounds}
			\forall x_j \ge \tfrac{n}{k}\rho, \qquad
			\frac{n+1}{n x_j^{-1} + k\Delta_1 + 1}
			\le
			\hat{v}_j(x_j)
			\le
			\frac{n+1}{(n x_j^{-1} - k\Delta_1)_+ + 1}.
		\end{equation}
		Moreover, since $\hat{v}_j$ is monotone, we have on $\event_{1,j}$ the inequalities
		\begin{equation}
			\label{eq:xnkrho}
			\forall x_j \le \tfrac{n}{k} \rho, \qquad
			\hat{v}_j(x_j) \le \hat{v}_j(\tfrac{n}{k} \rho) \le \frac{n+1}{k(\rho^{-1} - \Delta_1) + 1}
			\le \frac{n}{k} \frac{\rho}{1 - \rho \Delta_1}
			\le \frac{n}{k} \tau.
		\end{equation}
		
		Let $\event_1 = \bigcap_{j=1}^d \event_{1,j}$, the probability of which is at least $1 - d \delta/(d+1)$, as required. We need to show that on $\event_1$, the inclusions \eqref{eq:inclusions} hold. To do so, we proceed in steps. Throughout, we work on $\event_1$.
		\smallskip
		
		\emph{Step 1: Restriction to $(\tfrac{n}{k}\rho, \infty)^d$.} --- 
		If $x \in [0, \infty)^d$ is such that $\norm{\hat{v}(x)}_p \ge \tfrac{n}{k}$ but there exists $j = 1, \ldots, d$ with $x_j \le \tfrac{n}{k} \rho$, then, by \eqref{eq:xnkrho}, we have on $\event_1$ the bound
		\[ 
		\theta_{p,j}(\hat{v}(x)) = \frac{\hat{v}_j(x_j)}{\norm{\hat{v}(x)}_p} \le \tau
		\]
		and thus, by Assumption~\ref{ass:A}, necessarily $\theta_p(\hat{v}(x)) \not\in A$ for all $A \in \cA$. Hence, on $\event_1$, we have
		\begin{equation}
			\label{eq:GammaArho}
			\widehat{\Gamma}_A
			=
			\left\{
			x \in (\tfrac{n}{k}\rho, \infty)^d : \
			\norm{\hat{v}(x)}_p \ge \tfrac{n}{k}, \;
			\theta_p(\hat{v}(x)) \in A
			\right\}.
		\end{equation}
		\smallskip
		
		\emph{Step 2: Radial framing.} ---
		We consider the following decomposition of $\widehat{\Gamma}_A$:
		\[
		\left( \widehat{\Gamma}_A \cap \{x \in (\tfrac{n}{k}\rho, \infty)^d: \norm{x}_\infty \leq 2\tfrac{n}{k}\} \right) \cup \left( \widehat{\Gamma}_A \cap \{x \in (\tfrac{n}{k}\rho, \infty)^d: \norm{x}_\infty > 2\tfrac{n}{k}\} \right),
		\] 
		and we frame each set separately. 
		
		For points $x$ with $\norm{x}_\infty \leq 2\tfrac{n}{k}$, we seek to apply Lemma~\ref{lem: radialframing}. To this end, we first construct $h \in [0,1/2)$ such that, on $\event_1$, we have
		\[
		\forall j \in \{1,\ldots,d\}: \qquad \left\vert \frac{x_j}{\hat{v}_j(x_j)} - 1 \right\vert \leq h,
		\]
		for every $x_j > \tfrac{n}{k}\rho$ (which is larger than~$1$, by assumption that $\rho > k/n$). We simply apply Lemma~\ref{lem: vjtoPj} combined with~\eqref{eq: kndelta1}, giving
		\[
		\left\vert \frac{x_j}{\hat{v}_j(x_j)} - 1 \right\vert \leq \tfrac{k}{n} \Delta_1 x_j + \tfrac{|x_j-1|}{n} \leq \tfrac{k}{n} x_j \big(\Delta_1 + \tfrac{1}{k}\big) \leq 2 \big(\Delta_1 + \tfrac{1}{k}\big) = h,
		\]
		since $x_j \leq 2 \tfrac{n}{k}$. Hence, we may apply Lemma~\ref{lem: radialframing} with $r=n/k$ and $\Delta=\Delta_2$ since
		\[
		\frac{h}{1-h} \leq 4(\Delta_1+1/k) = \Delta_2.
		\]
		This leads to the inclusions
		\begin{align*}
			&\widehat{\Gamma}_A \cap \{x \in (\tfrac{n}{k}\rho, \infty)^d: \norm{x}_\infty \leq 2\tfrac{n}{k}\} \\
			&\qquad \subseteq \left\{x \in (\tfrac{n}{k}\rho, \infty)^d: \norm{x}_p \geq \tfrac{n}{k} \tfrac{1}{1+\Delta_2} \right\} \cap \{x \in (\tfrac{n}{k}\rho, \infty)^d: \norm{x}_\infty \leq 2\tfrac{n}{k}\} \qquad \text{and} \\
			&\widehat{\Gamma}_A \cap \{x \in (\tfrac{n}{k}\rho, \infty)^d: \norm{x}_\infty \leq 2\tfrac{n}{k}\} \\
			&\qquad \supseteq \left\{x \in (\tfrac{n}{k}\rho, \infty)^d: \norm{x}_p \geq \tfrac{n}{k} \tfrac{1}{1-\Delta_2} \right\} \cap \{x \in (\tfrac{n}{k}\rho, \infty)^d: \norm{x}_\infty \leq 2\tfrac{n}{k}\}.
		\end{align*}
		
		For points $x$ with $\norm{x}_\infty > 2\tfrac{n}{k}$, similar inclusions hold trivially. Indeed, on this set, by definition of the $L_\infty$-norm, there exists $j \in \{1,\ldots,d\}$ such that $x_j > 2 \tfrac{n}{k}$. Therefore, by our assumption that $k$ is large enough such that $\Delta_1 + 1/k \le 1/4$, we have by~\eqref{eq:xnkrho:bounds},
		\[
		\hat{v}_j(x_j) \geq \frac{n+1}{\tfrac{k}{2} + k \Delta_1 + 1} \geq \frac{n}{k} \frac{1}{\tfrac{1}{2} + (\Delta_1 + \tfrac{1}{k})} \geq \frac{n}{k}.
		\]
		Hence, we have $\norm{\hat{v}(x)}_\infty \geq n/k$, and, by equivalence of norms~\eqref{eq: equiv norms}, also $\norm{\hat{v}(x)}_p \geq n/k$ for every $p \in [1,\infty]$. Furthermore, $\norm{x}_p \geq \norm{x}_\infty \geq 2 \tfrac{n}{k} \geq \tfrac{n}{k} \tfrac{1}{1+\Delta_2}$ for every $0 < \Delta_2 < 1$. These inequalities imply the inclusions
		\begin{align*}
			&\widehat{\Gamma}_A \cap \{x \in (\tfrac{n}{k}\rho, \infty)^d: \norm{x}_\infty > 2\tfrac{n}{k}\} \\
			&\qquad \subseteq \left\{x \in (\tfrac{n}{k}\rho, \infty)^d: \norm{x}_p \geq \tfrac{n}{k} \tfrac{1}{1+\Delta_2} \right\} \cap \{x \in (\tfrac{n}{k}\rho, \infty)^d: \norm{x}_\infty > 2\tfrac{n}{k}\} \qquad \text{and}\\
			&\widehat{\Gamma}_A \cap \{x \in (\tfrac{n}{k}\rho, \infty)^d: \norm{x}_\infty > 2\tfrac{n}{k}\} \\
			&\qquad \supseteq \left\{x \in (\tfrac{n}{k}\rho, \infty)^d: \norm{x}_p \geq \tfrac{n}{k} \tfrac{1}{1-\Delta_2} \right\} \cap \{x \in (\tfrac{n}{k}\rho, \infty)^d: \norm{x}_\infty > 2\tfrac{n}{k}\}.
		\end{align*}
		
		Combining the two cases, we get radial framing, i.e.,
		\begin{equation}
			\label{eq:inclusions:radial}
			\begin{split}
				\widehat{\Gamma}_A &\subseteq \left\{x \in (\tfrac{n}{k}\rho, \infty)^d: \norm{x}_p \geq \tfrac{n}{k} \tfrac{1}{1+\Delta_2} \right\} \\
				\widehat{\Gamma}_A &\subseteq \left\{x \in (\tfrac{n}{k}\rho, \infty)^d: \norm{x}_p \geq \tfrac{n}{k} \tfrac{1}{1-\Delta_2} \right\}.
			\end{split}
		\end{equation}
		\smallskip
		
		\emph{Step 3: Angular framing.} ---
		We will show that, on $\event_1$, for any $x \in (\tfrac{n}{k}\rho, \infty)^d$ and $A \in \cA$, 
		\begin{equation}
			\label{eq:Apmimplic}
			\theta_p(x) \in A_-(3\tfrac{k}{n}\Delta_1\norm{x}_p)
			\implies
			\theta_p(\hat{v}(x)) \in A
			\implies
			\theta_p(x) \in A_+(3\tfrac{k}{n}\Delta_1\norm{x}_p).
		\end{equation}
		In combination with \eqref{eq:inclusions:radial}, this will show the inclusions \eqref{eq:inclusions} and thus finish the proof.
		
		Fix such $x$ and $A$. Put $\eps = 3\tfrac{k}{n}\Delta_1 \norm{x}_p$. We consider two cases: $\eps < 1$ and $\eps \ge 1$.
		
		If $\eps \ge 1$, the two implications in \eqref{eq:Apmimplic} are trivially fulfilled: Since the $\norm{\,\cdot\,}_p$-diameter of $\sphere_p$ is equal to $1$, we have $A_-(\eps) = \varnothing$ (as $\sphere_p \setminus A$ is not-empty) while $A_+(\eps) = \sphere_p$.
		
		The interesting case is thus $\eps < 1$. By Lemma~\ref{lem:theta}, we have
		\[
		\norm{\theta_p(\hat{v}(x))-\theta_p(x)}_p
		\le 2 \frac{\norm{\hat{v}(x)-x}_p}{\norm{x}_p}.
		\]
		Since $\eps < 1$, we have $\norm{x}_\infty \leq \norm{x}_p \leq (n/k)/(3\Delta_1)$. Hence, for all $j = 1, \ldots, d$, we have $n x_j^{-1} - k \Delta_1 \ge n \cdot 3\frac{k}{n} \Delta_1 - k \Delta_1 > 0$. By \eqref{eq:xnkrho:bounds}, we deduce
		\begin{align*}
			|\hat{v}_j(x_j) - x_j|
			&\le \max_{\sigma \in \{-1,+1\}}
			\left|
			\frac{n+1}{n x_j^{-1} + \sigma k\Delta_1 + 1} - x_j
			\right| \\
			&= x_j \max_{\sigma \in \{-1,+1\}}
			\left|
			\frac{n+1}{n + (\sigma k\Delta_1 + 1)x_j} - 1
			\right| \\
			&= x_j \max_{\sigma \in \{-1,+1\}}
			\frac{|1 - (\sigma k\Delta_1 + 1) x_j|}{n + (\sigma k\Delta_1 + 1)x_j}.
		\end{align*}
		Recall that $x_j \ge \tfrac{n}{k} \rho \ge 1$ and $k \Delta_1 \ge 2$. For he case $\sigma = +1$, we have
		\[
		\frac{|1 - (+k\Delta_1 + 1) x_j|}{n + (+k\Delta_1 + 1)x_j} 
		\le \frac{(k\Delta_1 + 1)x_j}{n + (k\Delta_1 + 1)x_j}
		\le \frac{3}{2} \frac{k}{n} \Delta_1 x_j.
		\]
		For the case $\sigma = -1$, we also have
		\[
		\frac{|1 - (-k\Delta_1 + 1) x_j|}{n + (-k\Delta_1 + 1)x_j}
		= \frac{(k\Delta_1-1)x_j - 1}{n + (-k\Delta_1 + 1)x_j}
		\le \frac{k\Delta_1 x_j}{n - k\Delta_1 x_j}
		\le \frac{3}{2} \frac{k}{n} \Delta_1 x_j,
		\]
		since $\eps < 1$ implies $n - k \Delta_1 x_j \ge n - k \Delta_1 \norm{x}_p \ge n - n/3 = 2n/3$. We deduce that
		\[
		\forall j \in \{1,\ldots,d\}: \qquad |\hat{v}_j(x_j) - x_j| \le \frac{3}{2} \frac{k}{n} \Delta_1 x_j^2.
		\]
		Consequently,
		\[
		\norm{\hat{v}(x) - x}_p \le \frac{3}{2} \frac{k}{n} \Delta_1 \norm{x}_p^2,
		\]
		where we use the (easy to prove) inequality $\norm{(x_1^2,\ldots,x_d^2)}_p \le \norm{(x_1,\ldots,x_d)}_p^2$, and thus
		\[
		\norm{\theta_p(\hat{v}(x)) - \theta_p(x)}_p
		\le 3 \tfrac{k}{n} \Delta_1 \norm{x}_p = \eps.
		\]
		The implications \eqref{eq:Apmimplic} now follow by definition of $A_-(\eps)$ and $A_+(\eps)$.
		
		We conclude that, on $\event_1$, the inclusions \eqref{eq:inclusions} hold, as required. The proof Theorem~\ref{thm:conc} is complete. 
	\end{proof}
	
	\bgroup
	\color{chgcol}
	\section{Proofs of Remark~\ref{rem:bias:handle}}
	\label{sec:appendixBias}
	
	As $\norm{x}_p \ge \rho$ implies $\norm{x}_\infty \ge d^{-1/p} \rho$ for $x \in \reals^d$ and $\rho > 0$, the bias term appearing in Theorem~\ref{thm:conc} is bounded by the total variation distance between the measures $\tfrac{n}{k} P_V(\tfrac{n}{k} \cdot)$ and $\mu$ restricted to
	\[
	E_c = \left\{ x \in (0, \infty)^d : \max(x) \geq c \right\},
	\]
	for $c = d^{-1/p} r_+^{-1}$. (This $c$ is not the same as the one in the statement of Theorem~\ref{thm:conc}.) More precisely, we have
	\[
	\sup_{A\in \cA,\; \sigma\in\{+,-\}} 
	\left\vert 
	\tfrac{n}{k} P_V(\tfrac{n}{k}  \Gamma_{A}^{\sigma}) - \mu(\Gamma_{A}^{\sigma})
	\right\vert
	\le
	\sup_{B \in \mathcal{B}(E_c)} \left| \tfrac{n}{k}P_V(\tfrac{n}{k}B) - \mu(B) \right|,
	\]
	where $\mathcal{B}(E_c)$ denotes the Borel $\sigma$-field on $E_c$. Recall $p_U$ and $\lambda$ in Remark~\ref{rem:bias:handle} and define
	\[
	\mathcal{D}_{T}(s) =
	\int_{L_T} \left| s^{d-1} p_U(sy) - \lambda(y) \right| \, \diff y
	\text{ with } L_T = \{y \in (0, T]^d : \min(y) \le 1 \}.
	\]
	The following proposition provides a refined version of the bound \eqref{eq:biasbound} in Remark~\ref{rem:bias:handle} in the paper; the bound \eqref{eq:biasbound} itself appears in the course of the proof of the proposition.
	
	\begin{proposition}
		\label{prop:TVdist}
		For $c > 0$, let $\mathcal{B}(E_c)$ denote the Borel $\sigma$-field on $E_c$. Then, for $t \geq 1/c$,
		\[
		\sup_{B \in \mathcal{B}(E_c)} \left| tP_V(tB) - \mu(B) \right|
		\leq \tfrac{1}{c} \mathcal{D}_{ct}(\tfrac{1}{ct}) 
		+ \mu \left( \left\{ x \in (0,\infty)^d: \max(x) \geq c, \min(x) \leq 1/t \right\} \right).
		\]	
	\end{proposition}
	
	
	\begin{proof}
		Writing $u = t^{-1}$, we get
		\begin{align*}
			\sup_{B \in \borel(E_c)} \left| t \, P_V(t B) - \mu(B)\right|
			&= 	\sup_{B \in \borel(E_c)} \left| t \, P_U(t^{-1} \iota(B)) - \Lambda(\iota(B))\right| \\
			&= 	\sup_{B \in \borel(\iota(E_c))} \left| u^{-1} \, P_U(u B) - \Lambda(B)\right|.
		\end{align*}
		For any Borel set $B \subseteq (0, \infty)^d$, we have, by a change of variables,
		\[
		u^{-1} \, P_U(u B)
		= u^{-1} \int_{u B} p_U(z) \, \diff z	
		= u^{d-1} \int_{B} p_U(uy) \, \diff y.
		\]
		It follows that
		\begin{align*}
			\sup_{B \in \borel(E_c)} 	\left| t \, P_V(t B) - \mu(B)\right|
			&= \sup_{B \in \borel(\iota(E_c))} \left| \int_{B} \left( u^{d-1} p_U(uy) - \lambda(y) \right) \, \diff y \right| \\
			&\le \int_{\iota(E_c)} \left| u^{d-1} p_U(u y) - \lambda(y) \right| \diff y \\
			&= \int_{0 < \min(y) \le 1/c} \left| u^{d-1} p_U(u y) - \lambda(y) \right| \diff y,
		\end{align*}
		which is \eqref{eq:biasbound} in the paper with $u = k/n$ and $c = d^{-1/p} r_+^{-1}$. Since the copula density $p_U$ vanishes outside $[0, 1]^d$, we get
		\[
		\sup_{B \in \borel(E_c)} 	\left| t \, P_V(t B) - \mu(B)\right|
		= \int_{\substack{0 < \min(y) \le 1/c \\ \max(y) \le 1/u}} 
		\left| u^{d-1} p_U(u y) - \lambda(y) \right| \diff y 
		+ \int_{\substack{0 < \min(y) \le 1/c \\ \max(y) > 1/u}} \lambda(y) \, \diff y.
		\]
		The first integral on the right-hand side is denoted by $\mathcal{I}(u, c)$ and is analysed below. The second integral on the right-hand side is
		\begin{equation*}
			\Lambda \left( \left\{ 
			y \in (0, \infty)^d : \min(y) \le 1/c, \, \max(y) > 1/u
			\right\} \right)
			=
			\mu \left( \left\{ 		x \in (0, \infty)^d : \max(x) \ge c, \, \min(x) < u
			\right\} \right).
		\end{equation*}
		By a change of variables $y = c^{-1} x$ (component-wise), we find
		\[
		\mathcal{I}(u,c) 
		=
		c^{-d} \int_{\substack{0 < \min(x) \le 1 \\ \max(x) \le c/u}}
		\left| u^{d-1} p_U(uc^{-1}x) - \lambda(c^{-1}x) \right| \diff y.
		\]
		The density $\lambda$ is homogeneous of order $1-d$, i.e., $\lambda(c^{-1}x) = c^{d-1} x$. Writing $c^{-1}u = s$, we find
		\begin{align*}
			\mathcal{I}(u,c)
			&=
			c^{-d} \int_{\substack{0 < \min(x) \le 1 \\ \max(x) \le c/u}}
			\left| c^{d-1} s^{d-1} p_U(sx) - c^{d-1} \lambda(x) \right| \diff x \\
			&=
			c^{-1} \int_{\substack{0 < \min(x) \le 1 \\ \max(x) \le 1/s}}
			\left| s^{d-1} p_U(sx) - \lambda(x) \right|  \diff x \\
			&=
			c^{-1} \mathcal{D}_{1/s}(s).
		\end{align*}
		Substituting $s = c^{-1} u = (ct)^{-1}$ yields the stated bound. 
	\end{proof}
	
	\begin{example}[The multivariate Cauchy distribution]
		\label{ex:appendixMultCauchy}
		Let us assume that $X$ follows a multivariate Cauchy distribution on the positive orthant whose density is given by
		$$
		f(x) = \frac{2^d \Gamma(\tfrac{1+d}{2})}{\pi^{\frac{1+d}{2}}} \frac{1}{(1+\|x\|_2^2)^{\tfrac{1+d}{2}}}
		$$	
		for $x \geq 0$ and $f(x)=0$ otherwise. 
		We will show that the bound in Proposition~\ref{prop:TVdist} is $\Oh(1/t)$ as $t \to \infty$. This implies that the bias term is $\Oh(k/n)$ as $k = k_n \to \infty$ in such a way that $k/n \to 0$.
		
		For simplicity of notation, let $p = p_U$ denote the probability density function of $U = (1-F_1(X_1),\ldots,1-F_d(X_d))$. Some computations related to the univariate Cauchy distribution lead to
		$$
		p(x) = \pi^{\tfrac{d-1}{2}} \Gamma(\tfrac{1+d}{2}) \frac{\prod_{i=1}^{d} \left( 1 + \tan^2(\tfrac{\pi}{2}(1-x_i)) \right)}{\left( 1 + \sum_{i=1}^{d} \tan^2(\tfrac{\pi}{2}(1-x_i)) \right) ^{\tfrac{1+d}{2}}}.
		$$
		Asymptotic considerations on the tangent function permit to show that
		$$
		\lim_{s \rightarrow 0} s^{d-1} p(sx) = \frac{2^{d-1}\Gamma(\tfrac{1+d}{2})}{\pi^{\tfrac{d-1}{2}}} \frac{x_1^{-2} \cdots x_d^{-2}}{\left( x_1^{-2} + \cdots + x_d^{-2} \right)^{\tfrac{1+d}{2}} } =: \lambda(x).
		$$
		
		We start with the first term in the upper bound of Proposition~\ref{prop:TVdist}. We have the decomposition
		\begin{align}
			\mathcal{D}_{1/s}(s) &= \int_{L_{1/s}} \left(s^{d-1}p(sx) - \lambda(x)\right)_+ \, \diff x \label{eq:decomp1} \\
			&\qquad + \int_{L_{1/s}} \left(\lambda(x)-s^{d-1}p(sx)\right)_+ \, \diff x. \label{eq:decomp2}
		\end{align}
		We start by studying the integral~\eqref{eq:decomp1}. We observe that for $z \in (0,1]$
		$$
		\tan^2 \left(\tfrac{\pi}{2}(1-z)\right)
		= \cot^2 \left( \tfrac{\pi}{2} z \right) 
		= \frac{1}{\sin^2 \left(\tfrac{\pi}{2} z\right)} - 1
		= \frac{2}{1 - \cos(\pi z)} - 1.
		$$
		If $x_i \leq 1$ for every $i \in \{1,\ldots,d\}$, we may apply~\eqref{eq:cosbounds} in Lemma~\ref{lem:cosbounds} below to get
		\begin{align*}
			p(x) &= \pi^{\tfrac{d-1}{2}} \Gamma(\tfrac{1+d}{2}) \frac{\prod_{i=1}^{d} \left( 1 + \tan^2(\tfrac{\pi}{2}(1-x_i)) \right)}{\left( 1 + \sum_{i=1}^{d} \tan^2(\tfrac{\pi}{2}(1-x_i)) \right) ^{\tfrac{1+d}{2}}} \\
			&= \pi^{\tfrac{d-1}{2}} \Gamma(\tfrac{1+d}{2}) \frac{\frac{2}{1 - \cos(\pi x_1)} \cdots \frac{2}{1 - \cos(\pi x_d)}}{ \left( 1 - d + \frac{2}{1 - \cos(\pi x_1)} + \cdots + \frac{2}{1 - \cos(\pi x_d)} \right)^{\tfrac{1+d}{2}}} \\
			&= (2\pi)^{\tfrac{d-1}{2}} \Gamma(\tfrac{1+d}{2}) \frac{\frac{1}{1 - \cos(\pi x_1)} \cdots \frac{1}{1 - \cos(\pi x_d)}}{ \left( \frac{1-d}{2} + \frac{1}{1 - \cos(\pi x_1)} + \cdots + \frac{1}{1 - \cos(\pi x_d)} \right)^{\tfrac{1+d}{2}}} \\
			&\leq (2\pi)^{\tfrac{d-1}{2}} \Gamma(\tfrac{1+d}{2}) \frac{\left( \frac{2}{\pi^2 x_1^2} + \frac{1}{3} \right)  \cdots \left( \frac{2}{\pi^2 x_d^2} + \frac{1}{3} \right)}{ \left( \frac{1-d}{2} + \frac{2}{\pi^2 x_1^2} + \frac{1}{6} + \cdots + \frac{2}{\pi^2 x_d^2} + \frac{1}{6} \right)^{\tfrac{1+d}{2}}}.
		\end{align*}
		Hence, for $0 < s \leq 1$ and $x \in L_{1/s}$, we have the upper bound
		\begin{align*}
			s^{d-1}p(sx) &\leq (2\pi)^{\tfrac{d-1}{2}} \Gamma(\tfrac{1+d}{2}) \frac{s^{d-1} \left( \frac{2}{\pi^2 s^2 x_1^2} + \frac{1}{3} \right)  \cdots \left( \frac{2}{\pi^2 s^2 x_d^2} + \frac{1}{3} \right)}{ \left( \frac{1-d}{2} + \frac{2}{\pi^2 s^2 x_1^2} + \frac{1}{6} + \cdots + \frac{2}{\pi^2 s^2 x_d^2} + \frac{1}{6} \right)^{\tfrac{1+d}{2}}} \\
			&= \frac{2^{d-1}\Gamma(\tfrac{1+d}{2})}{\pi^{\tfrac{d-1}{2}}} \frac{\left( \frac{1}{x_1^2} + \frac{\pi^2 s^2}{6} \right) \cdots \left( \frac{1}{x_d^2} + \frac{\pi^2 s^2}{6} \right)}{\left( \frac{1}{x_1^2} + \cdots + \frac{1}{x_d^2} - \frac{\pi^2s^2}{4}(\tfrac{2d}{3}-1) \right) ^{\tfrac{1+d}{2}}}.
		\end{align*}
		Note that the upper bound converges to $\lambda(x)$ as $s \rightarrow 0$. 
		
		Now observe that
		$$
		\frac{1}{\left( \frac{1}{x_1^2} + \cdots + \frac{1}{x_d^2} - \frac{\pi^2s^2}{4}(\tfrac{2d}{3}-1) \right)^{\tfrac{1+d}{2}}} 
		= \frac{1}{(x_1^{-2} + \cdots + x_d^{-2})^{\tfrac{1+d}{2}}} \frac{1}{\left( 1 - \frac{\pi^2}{4}(\tfrac{2d}{3} -1) \tfrac{s^2}{x_1^{-2} + \cdots + x_d^{-2}} \right)^{\tfrac{1+d}{2}}},
		$$
		with
		$$
		\frac{\pi^2}{4}(\tfrac{2d}{3} -1) \tfrac{s^2}{x_1^{-2} + \cdots + x_d^{-2}} \leq \frac{\pi^2}{4}(\tfrac{2d}{3} -1)s^2.
		$$
		Since we are interested in the limit $s \rightarrow 0$, we may assume that $s^2 \leq \frac{2}{\pi^2(\tfrac{2d}{3} -1)}$ so that we may apply the upper bound in \eqref{eq:powerbound} in Lemma~\ref{lem:powerbound} below. We get
		\begin{align*}
			s^{d-1}p(sx) &\leq \frac{2^{d-1}\Gamma(\tfrac{1+d}{2})}{\pi^{\tfrac{d-1}{2}}} \frac{\left( \frac{1}{x_1^2} + \frac{\pi^2 s^2}{6} \right) \cdots \left( \frac{1}{x_d^2} + \frac{\pi^2 s^2}{6} \right)}{\left( \frac{1}{x_1^2} + \cdots + \frac{1}{x_d^2} - \frac{\pi^2s^2}{4}(\tfrac{2d}{3}-1) \right) ^{\tfrac{1+d}{2}}} \\
			&\leq \frac{2^{d-1}\Gamma(\tfrac{1+d}{2})}{\pi^{\tfrac{d-1}{2}}} \frac{\left( x_1^{-2} + \frac{\pi^2 s^2}{6} \right) \cdots \left( x_d^{-2} + \frac{\pi^2 s^2}{6} \right)}{\left( x_1^{-2} + \cdots x_d^{-2} \right) ^{\tfrac{1+d}{2}}} \cdot \left( 1 + s^2 \frac{2 \big( 2^{\tfrac{1+d}{2}} - 1 \big) \tfrac{\pi^2}{4}(\tfrac{2d}{3} -1)}{x_1^{-2} + \cdots + x_d^{-2}} \right).
		\end{align*}
		Defining $C_1,C_2 > 0$ by
		$$
		C_1 = \frac{2^{d-1}\Gamma(\tfrac{1+d}{2})}{\pi^{\tfrac{d-1}{2}}} \qquad \text{ and } \qquad C_2 = \tfrac{\pi^2}{2} \big( 2^{\tfrac{1+d}{2}} - 1 \big)(\tfrac{2d}{3} -1),
		$$ 
		we thus have the bound
		\begin{align*}
			(s^{d-1}p(sx) - \lambda(x))_+ &\leq s^2 C_1 C_2 \frac{x_1^{-2} \cdots x_d^{-2}}{\left( x_1^{-2} + \cdots x_d^{-2} \right)^{\tfrac{3+d}{2}}} + C_1 \frac{\sum_{k=0}^{d-1} \sum_{I \subset \{1,\ldots,d\}, |I| = k} x_{I}^{-2} (\tfrac{\pi^2 s^2}{6})^{d-k}}{\left( x_1^{-2} + \cdots x_d^{-2} \right)^{\tfrac{1+d}{2}}}  \\
			&\qquad + s^2 C_1 C_2 \frac{\sum_{k=0}^{d-1} \sum_{I \subset \{1,\ldots,d\}, |I| = k} x_{I}^{-2} (\tfrac{\pi^2 s^2}{6})^{d-k}}{\left( x_1^{-2} + \cdots x_d^{-2} \right)^{\tfrac{3+d}{2}}},
		\end{align*}
		where $x_{I}$ denotes the sub-vector of $x$ with coordinates in $I$ and the square is to be understood coordinate-wise.
		
		The first term is the easiest to handle as it is bounded by a multiple of $\lambda(x)$, it is integrable over $L_\infty$ and its contribution to the bias term is at most $O(s^2)$. The third term is bounded by a multiple of the second term times $s^2$ and is therefore negligible in front of the second term. Now note that that the domain of integration $L_{1/s}$ can be rewritten as
		$$
		L_{1/s} = \bigcup_{\emptyset \neq J \subset \{1,\ldots,d\}} B_{J,1/s} \text{ where } B_{J,1/s} = \left\{	x \in (0,1/s]^d: \forall j \in J, x_j \leq 1 ; \forall j \in J^c: x_j > 1 \right\}.
		$$
		Hence, integrating the second term over $L_{1/s}$ implies that we will need to bound integrals of the form
		$$
		s^{2(d-|I|)} \int_{B_{J,1/s}} \frac{\prod_{i \in I} x_i^{-2}}{\left( x_1^{-2} + \cdots + x_d^{-2} \right)^{\tfrac{1+d}{2}}} \diff x,
		$$
		for subsets $I,J \subset \{1,\ldots,d\}$ with $I^c \neq \emptyset$ and $J \neq \emptyset$. Without loss of generality we may assume that $J = \{1,\ldots,j\}$ for some $j \in \{1,\ldots,d\}$ and therefore, a change of variable $x_\ell = y_\ell^{-1}$ for all $\ell \in \{1,\ldots,d\}$ permits to rewrite the integral as
		\begin{multline*}
			s^{2(d - |I|)}
			\int_{x_1=0}^1 \ldots \int_{x_j=0}^1 
			\int_{x_{j+1}=1}^{1/s} \ldots \int_{x_d=1}^{1/s} 
			\left( x_1^{-2} + \cdots + x_d^{-2} \right)^{-(1+d)/2}
			\prod_{i \in I} x_i^{-2}
			\diff x \\
			=
			s^{2 |I^c|}
			\int_{y_1=1}^\infty \ldots \int_{y_j=1}^\infty
			\int_{y_{j+1}=s}^{1} \ldots \int_{y_d=s}^{1}
			\left( y_1^{2} + \cdots + y_d^{2} \right)^{-(1+d)/2}
			\prod_{i \in I^c} y_i^{-2}
			\diff y.
		\end{multline*}
		If $j=d$, the integral is finite and hence the contribution to the bias is at least of the order $O(s^2)$ as $|I^c| \geq 1$. Thus, we shall assume $j < d$. Observe that for every $i \in I^c$, if $i \leq j$, then we may upper bound $y_i^{-1} \leq 1$. Consequently, the worst case will happen whence $I^c \cap J = \emptyset$. In this case, bounding $\left( y_1^{2} + \cdots + y_d^{2} \right)^{-(1+d)/2}$ by $\left( y_1^{2} + \cdots + y_j^{2} \right)^{-(1+d)/2}$ permits to upper bound the latter integral by
		$$
		s^{2|I^c|}
		\int_{y_1=1}^\infty \ldots \int_{y_j=1}^\infty
		(y_1^2 + \cdots + y_j^2)^{-(1+d)/2} \cdot
		\diff y_1 \cdots \diff y_j
		\cdot
		\prod_{i \in I^c} \int_{y_i=s}^1 y_i^{-2} \, \diff y_i,
		$$ 
		which is of order $O(s^{|I^c|})$ as the integral over $(y_1,\ldots,y_j)$ is finite. Since $|I^c| \geq 1$, we find that the order of the first integral in the bias decomposition~\eqref{eq:decomp1} is of the order $O(s)$.
		
		We now claim that second term in the decomposition~\eqref{eq:decomp2} is always equal to zero for sufficiently small $s>0$ and therefore conclude that $\mathcal{D}_{1/s}(s) = O(s)$ as $s \rightarrow 0$. 
		Previous computations and bounds~\eqref{eq:cosbounds} show that for every $x$ such that $x_i \leq 1$ for $i \in \{1,\ldots,d\}$,
		\begin{align*}
			p(x) &= (2\pi)^{\tfrac{d-1}{2}} \Gamma(\tfrac{1+d}{2}) \frac{\frac{1}{1 - \cos(\pi x_1)} \cdots \frac{1}{1 - \cos(\pi x_d)}}{ \left( \frac{1-d}{2} + \frac{1}{1 - \cos(\pi x_1)} + \cdots + \frac{1}{1 - \cos(\pi x_d)} \right)^{\tfrac{1+d}{2}}} \\
			&\geq (2\pi)^{\tfrac{d-1}{2}} \Gamma(\tfrac{1+d}{2}) \frac{\left( \tfrac{2}{\pi^2 x_1^2} + \tfrac{1}{6} \right) \cdots \left( \tfrac{2}{\pi^2 x_d^2} + \tfrac{1}{6} \right)}{\left( \tfrac{1-d}{2} + \tfrac{2}{\pi^2x_1^2} + \tfrac{1}{3} + \cdots + \tfrac{2}{\pi^2x_d^2} + \tfrac{1}{3} \right)^{\tfrac{1+d}{2}}}.
		\end{align*}
		Hence, for $x \in L_{1/s}$, we have
		\begin{align*}
			s^{d-1}p(sx) 
			&\geq (2\pi)^{\tfrac{d-1}{2}} \Gamma(\tfrac{1+d}{2})  \frac{s^{d-1} \left( \tfrac{2}{\pi^2 s^2 x_1^2} + \tfrac{1}{6} \right) \cdots \left( \tfrac{2}{\pi^2 s^2 x_d^2} + \tfrac{1}{6} \right)}{\left( \tfrac{2}{\pi^2 s^2 x_1^2} + \cdots + \tfrac{2}{\pi^2 s^2 x_d^2} - (\tfrac{d}{6} - 1/2)  \right)^{\tfrac{1+d}{2}}} \\
			&= \frac{2^{d-1} \Gamma(\tfrac{1+d}{2})}{\pi^{\tfrac{d-1}{2}}} \frac{(x_1^{-2} + \tfrac{\pi^2s^2}{12}) \cdots (x_d^{-2} + \tfrac{\pi^2s^2}{12})}{\left( x_1^{-2} + \cdots + x_d^{-2} - s^2\pi^2(\tfrac{d}{3} - 1) \right)^{\tfrac{1+d}{2}}}.
		\end{align*}
		Note that the lower bound converges to $\lambda(x)$ as $s \rightarrow 0$. Since we are interested in $s \rightarrow 0$, we may assume that $s^2 \leq (2\pi^2(d/3-1))^{-1}$ and apply the lower bound in~\eqref{eq:powerbound} in Lemma~\ref{lem:powerbound} below to get that for every $x \in L_{1/s}$,
		$$
		s^{d-1}p(sx) \geq \frac{2^{d-1} \Gamma(\tfrac{1+d}{2})}{\pi^{\tfrac{d-1}{2}}} \frac{(x_1^{-2} + \tfrac{\pi^2s^2}{12}) \cdots (x_d^{-2} + \tfrac{\pi^2s^2}{12})}{\left( x_1^{-2} + \cdots x_d^{-2} \right)^{\tfrac{1+d}{2}}} \left( 1 + s^2 \frac{\pi^2(\tfrac{d}{3} - 1)(\tfrac{1+d}{2})}{x_1^{-2} + \cdots x_d^{-2}} \right).
		$$
		This shows in particular that $s^{d-1}p(sx) - \lambda(x) \geq 0$ for every value of $x \in L_{1/s}$ and $s \leq (2\pi^2(d/3-1))^{-1/2}$ so that the integral
		$$
		\int_{L_{1/s}} (\lambda(x)-s^{d-1}p(sx))_+ \diff x = 0,
		$$
		for those values of $s$ and we get that the bias $\mathcal{D}_{1/s}(s) = O(s)$ as $s \rightarrow 0$.
		
		To deal with the second term in the upper bound of Proposition~\ref{prop:TVdist}, we note that since,
		$$
		\frac{\diff \mu}{\diff x}(x) = \lambda(\iota(x)) x_1^{-2} \cdots x_d^{-2} = C_1 \|x\|_2^{-1-d},
		$$
		the density associated to $\Phi_2$ is constant, say $\varphi$, so that
		\begin{align*}
			\mu \left( \left\{ x \in (0,\infty)^d: \max(x) \geq c, \min(x) \leq 1/t \right\} \right) 
			&\leq c^{-1} \Phi_2 \left( \{\theta \in \sphere_2: \min(\theta) \leq \tfrac{1}{ct}\} \right) \\
			&= c^{-1} \varphi \leb_{d-1} \left( \{\theta \in \sphere_2: \min(\theta) \leq \tfrac{1}{ct}\} \right) = O(1/t),
		\end{align*}
		when $t\rightarrow \infty$. This concludes the analysis of the bias term for the multivariate Cauchy distribution.
	\end{example}
	
	\begin{lemma}
		\label{lem:cosbounds}
		We have
		\begin{equation}
			\label{eq:cosbounds}
			\frac{1}{6} \leq \frac{1}{1 - \cos(t)} - \frac{2}{t^2} \leq \frac{1}{3} \qquad \text{ for } 0 < t \leq \pi.
		\end{equation}
	\end{lemma}
	
	\begin{proof}
		Note that we may rewrite the inequality as follows: for $0 \leq t \leq \pi$, we must have 
		\begin{align*}
			\frac{1}{6} \leq \frac{1}{1-\cos t} - \frac{2}{t^2} \leq \frac{1}{3}  
			&\iff \frac{1}{6} + \frac{2}{t^2} \leq \frac{1}{1-\cos t} \leq \frac{1}{3} + \frac{2}{t^2} \\
			&\iff \frac{12+t^2}{6t^2} \leq \frac{1}{1-\cos t} \leq \frac{6+t^2}{3t^2} \\
			&\iff  \frac{3t^2}{6+t^2} \leq 1 - \cos t \leq \frac{6t^2}{12+t^2},
		\end{align*}
		which gives the inequalities
		\begin{align}
			\cos t &\leq 1 - \frac{3t^2}{6+t^2} = \frac{6 - 2t^2}{6+t^2} \label{eq:cosbound1} \\
			\cos t &\geq 1 - \frac{6t^2}{12+t^2} = \frac{12-5t^2}{12+t^2} \label{eq:cosbound2}.
		\end{align}
		We start by showing~\eqref{eq:cosbound1}. Taylor's theorem implies that
		$$
		\cos t \leq 1- \frac{t^2}{2!} + \frac{t^4}{4!} - \frac{t^6}{6!} + \frac{t^8}{8!}.
		$$
		Consequently,
		\begin{align*}
			(6+t^2) \cos t &\leq (6+t^2) \cdot \left( 1- \frac{t^2}{2!} + \frac{t^4}{4!} - \frac{t^6}{6!} + \frac{t^8}{8!} \right) \\
			&= 6 - 3t^2 + \frac{t^4}{4} - \frac{t^6}{120} + \frac{t^8}{6720} + t^2 - \frac{t^4}{2} + \frac{t^6}{24} - \frac{t^8}{720} + \frac{t^10}{40320} \\
			&= 6 - 2t^2 + \frac{t^4}{40320} \left( t^6 - 50t^4 + 1344t^2 - 10080 \right).
		\end{align*}
		Hence, if we show that the polynomial $R(s) := s^3 - 50s^2 + 1344s - 10080$ is strictly negative on $0 \leq s \leq \pi^2$, we are done proving~\eqref{eq:cosbound1}. 
		It is sufficient to observe that $R(\pi^2) < 0$ and $R'(s) \geq 0$ for the concerned values of $s$ since 
		$$
		R'(s) = 3s^2 - 100s + 1344
		$$
		satisfies $\Delta_R = 1000 - 12 \cdot 1344 < 0$ and $R'(0) = 1344 > 0$.
		
		To prove~\eqref{eq:cosbound2}, we follow the same path: Taylor's theorem implies that
		$$
		\cos t \geq 1 - \frac{t^2}{2} + \frac{t^4}{4!} - \frac{t^6}{6!}.
		$$
		Consequently,
		\begin{align*}
			(12+t^2)\cos t &\geq (12+t^2) \cdot \left( 1 - \frac{t^2}{2} + \frac{t^4}{4!} - \frac{t^6}{6!}\right) \\
			&= 12 - 6t^2 + \frac{t^4}{2} - \frac{t^6}{60} + t^2 - \frac{t^4}{2} + \frac{t^6}{24} - \frac{t^8}{720} \\
			&= 12 - 5t^2 + \frac{t^6}{720} \left( -t^2 + 18 \right).
		\end{align*}
		Since $-t^2 + 18 \geq 0$ for $0 \leq t \leq \sqrt{18}$ and $\sqrt{18} > \pi$, we have proven~\eqref{eq:cosbound2}.
	\end{proof}
	
	\begin{lemma}
		\label{lem:powerbound}
		For $a \in [0, 1/2]$, we have
		\begin{equation}
			\label{eq:powerbound}
			1 + \left( \frac{1+d}{2} \right)  a 
			\leq (1-a)^{-\tfrac{1+d}{2}} 
			\leq 1 + 2 \left( 2^{\tfrac{1+d}{2}} - 1 \right) a.
		\end{equation}
	\end{lemma}
	
	\begin{proof}
		To prove the upper bound, we will use a convexity argument. Note that the function $\phi(a) = (1-a)^{- \tfrac{1+d}{2}}$ on the left hand side is a convex function of its argument. Since it is a $C^2$ function, we may prove that fact by differentiating it two times and show that $\phi''(a) > 0$ for $0 \leq a \leq 1/2$. Since we compute that 
		$$
		\phi''(a) = \tfrac{d+1}{2} \cdot \tfrac{d+3}{2} \cdot (1-a)^{- \tfrac{d+5}{2}},
		$$
		the convexity on the domain of interest is proven. This implies that for any $\lambda \in [0,1]$ and any pair of points $a_1,a_2 \in [0,1/2]$,
		$$
		\phi(\lambda a_1 + (1-\lambda) a_2) \leq \lambda \phi(a_1) + (1-\lambda) \phi(a_2).
		$$
		In particular, the graph of $\phi$ is anywhere bounded by the straight line joining the points $\phi(0) = 1$ and $\phi(1/2) = 2^{\tfrac{1+d}{2}}$ on the domain of interest. This affine function has analytic expression $$L(a) = 1 + 2 \left( 2^{\tfrac{1+d}{2}} - 1 \right) a$$ and corresponds to the upper bound in the statement.
		
		To prove the lower bound, define $\phi(a) = (1-a)^{-(d+1)/2}$. Then, we already showed that it satisfies $\phi''(a) \geq 0$ for $a \in [0,1/2]$. Consequently, Taylor's theorem ensures that for every $a$ in the region of interest
		$$
		\phi(a) = \phi(0) + \phi'(0)a + R_0^1\phi(a) \geq \phi(0) + \phi'(0)a,
		$$
		since, by Lagrange's formula, the rest $R_0^1\phi(a) = \tfrac{\phi''(c)}{2}a^2$ for some $c \in (0,a)$ and the second derivative is positive at this point. Observing that $\phi(0)= 1$ and $\phi'(0) = (1+d)/2$ permits to conclude.
	\end{proof}
	\egroup

	\section{Proofs of examples}
	\label{sec:appendixExamples}
	\begin{proof}[Proof of Example~\ref{ex:linear}]
		We first consider Assumption~\ref{ass:A}.
		Restricting $a$ and $\beta$ to have rational coordinates yields a countable collection $\cA_0 \subset \cA$ satisfying item~(i) in Assumption~\ref{ass:A}.
		Item~(ii) is satisfied by definition.
		For $A_{a, \beta, \tau} \in \cA$ and $\eps > 0$, define the inner and outer hulls
		\begin{align*}
			A_{a, \beta, \tau; -}(\eps) &= A_{a, \beta - \sqrt{d} \eps, \tau + \eps}, \\
			A_{a, \beta, \tau; +}(\eps) &= A_{a, \beta + \sqrt{d} \eps, \tau - \eps},
		\end{align*}
		with the convention that $\sphere_p^\upsilon = \sphere_p$ if $\upsilon < 0$ (a case which arises if $\eps > \tau$). 
		We show that item~(iii) in Assumption~\ref{ass:A} is satisfied, provided the angular measure $\Phi_p$ has a bounded density on $\sphere_p^\tau$ with respect to $(d-1)$-dimensional Lebesgue measure.
		First, we show the inclusions~\eqref{eq:Ahull:inclusions}.
		\begin{itemize}
			\item For $x \in A_{a, \beta, \tau; -}(\eps)$ and for $y \in \sphere_p$ such that $\norm{y-x}_\infty \le \eps$, we have $y \in A_{a, \beta,\tau}$: indeed, 
			\[ 
			y_1 \wedge \cdots \wedge y_d 
			\ge x_1 \wedge \cdots \wedge x_d - \eps 
			> \tau + \eps - \eps 
			= \tau 
			\]
			as well as (by equivalence of norms~\eqref{eq: equiv norms})
			\[ 
			\inpr{a, y} 
			= \inpr{a, x} + \inpr{a, y-x}
			\le \beta - \sqrt{d} \eps + \norm{a}_2 \norm{y-x}_2
			\le \beta - \sqrt{d} \eps + \sqrt{d} \norm{y-x}_\infty
			\le \beta.
			\]
			This is the first inclusion in \eqref{eq:Ahull:inclusions}.
			\item For $x \in A_{a, \beta, \tau}$ and for $y \in \sphere_p$ such that $\norm{y-x}_\infty \le \eps$, we have
			\[
			y_1 \wedge \cdots \wedge y_d 
			\ge x_1 \wedge \cdots \wedge x_d - \eps 
			> \tau - \eps
			\]
			so that $y \in \sphere_p^{\tau - \eps}$ (if $\eps > \tau$ this trivial since $\sphere_p^\upsilon = \sphere_p$ for $\upsilon < 0$) together with
			\[
			\inpr{a, y}
			= \inpr{a, x} + \inpr{a, y-x}
			\le \beta + \norm{a}_2 \norm{y-x}_2
			\le \beta + \sqrt{d} \norm{y-x}_\infty
			\le \beta + \sqrt{d} \eps.
			\]
			This implies the second inclusion in \eqref{eq:Ahull:inclusions}.
		\end{itemize}
		The difference between the inner and outer hulls is
		\begin{equation*}
			A_{a, \beta, \tau; +}(\eps) \setminus A_{a, \beta, \tau; -}(\eps)
			\subseteq
			\left\{ 
			x \in \sphere_p : \ 
			\beta - \sqrt{d} \eps < \inpr{a, x} \le \beta + \sqrt{d} \eps 
			\right\} 
			\cup ( \sphere_p^{\tau-\eps} \setminus \sphere_p^{\tau+\eps} ).
		\end{equation*}
		Since $\norm{a}_2 = 1$, the $(d-1)$-dimensional Lebesgue measure of the set on the right-hand side is bounded by a constant multiple of $\eps$.
		
		Next we show that Assumption~\ref{ass:VF} is satisfied.
		The VC-dimension is preserved under bijections.
		We identify $E = [0, \infty)^d \setminus \{(0, \ldots, 0)\}$ with the product set $(0, \infty) \times \sphere_p$ via $x \mapsto (\norm{x}_p, x / \norm{x}_p)$.
		In the latter space, we need to establish the finiteness of the VC-dimensions of the collections $\{ \Upsilon_A^-(r, h) : A \in \cA\}$ and $\{ \Upsilon_A^+(r, h) : A \in \cA \}$ for fixed $r, h > 0$, where
		\[
		\Upsilon_A^\sigma(r, h)
		=
		\left\{ 
		(u, \theta) \in (\tfrac{1}{r}, \infty) \times \sphere_p : \
		\theta \in A_{\sigma}(h u)
		\right\}
		\]
		for $\sigma \in \{-, +\}$ and $A \in \cA$. For $A = A_{a, \beta, \tau}$ we have 
		\[ 
		\theta \in A_-(h u )
		\iff
		\left[
		\theta_1 \wedge \cdots \wedge \theta_d \ge \tau + h  u 
		\text{ and }
		\inpr{a, \theta} \le \beta - \sqrt{d} h u
		\right].
		\]
		This corresponds to $d+1$ linear inequality constraints on $( u , \theta)$ as $(a, \beta)$ ranges over $\reals^d \times \reals$.
		The VC dimension of $\{ \Upsilon_A^-(r, h) : A \in \cA\}$ as a collection of subsets of $(0, \infty) \times \sphere_p$ is therefore finite \citep[Lemma~2.6.17 and Exercise~14 on page~152]{vdvaart+w:1996}, and the same is then true for $\{ \Gamma_A^-(r, h) : A \in \cA\}$ as a collection of subsets of $E$.
		The argument for the outer hulls is similar.
	\end{proof}

	\begin{proof}[Proof of Example~\ref{ex:intersec-union}]
		We first consider the collection of intersections.
		In Assumption~\ref{ass:A}, item~(i) follows by considering the countable collection of intersections $\cA_{1,0} \sqcap \cA_{2,0}$, where $\cA_{j,0}$ is the countable subset of $\cA_j$ for $j \in \{1, 2\}$.
		Item~(ii) is trivially satisfied.
		Regarding~(iii): given $A_j \in \cA_j$ with inner and outer hulls $A_{j,\sigma}(\eps)$, the inner and outer hulls of $A_1 \cap A_2$ can chosen to be $A_{1,-}(\eps) \cap A_{2,-}(\eps)$ and $A_{1,+}(\eps) \cap A_{2,+}(\eps)$.
		The two inclusions~\eqref{eq:Ahull:inclusions} are straightforward to verify. The measure of the set difference between the inner and outer hulls can be controlled via
		\[
		\bigl( A_{1,+}(\eps) \cap A_{2,+}(\eps) \bigr)
		\setminus
		\bigl( A_{1,-}(\eps) \cap A_{2,-}(\eps) \bigr)
		\subseteq
		\bigl( A_{1,+}(\eps) \setminus A_{1,-}(\eps) \bigr)
		\cup
		\bigl( A_{2,+}(\eps) \setminus A_{2,-}(\eps) \bigr).
		\]
		If $c_1$ and $c_2$ denote the constants on the right-hand of \eqref{eq:Phi+-} for $\cA_1$ and $\cA_2$, respectively, then $c_1 + c_2$ is a valid constant for $\cA_1 \sqcap \cA_2$.
		
		To show that the collections of framing sets derived from the inner and outer hulls thus constructed have a finite VC-dimension (Assumption~\ref{ass:VF}), it suffices to observe that, by definition,
		\[
		\Gamma_{A_1 \cap A_2}^\sigma(r, h) 
		= \Gamma_{A_1}^\sigma(r, h) \cap \Gamma_{A_2}^\sigma(r, h)
		\]
		and thus
		\[
		\{ \Gamma_{A}^\sigma(r, h) : A \in \cA_1 \sqcap \cA_2 \}
		=
		\{ \Gamma_{A_1}^\sigma(r, h) : A_1 \in \cA_1 \}
		\sqcap
		\{ \Gamma_{A_2}^\sigma(r, h) : A_2 \in \cA_2 \}.
		\]
		The collection on the left-hand side has a finite VC-dimension since, by assumption, each of the two collections on the right-hand side has a finite VC-dimension; see for instance \citep[Lemma~2.6.17]{vdvaart+w:1996}.
		
		For the collection of unions, the verification of Assumptions~\ref{ass:A} and~\ref{ass:VF} is completely similar. The inner and outer hulls of a union $A_1 \cup A_2$ are now defined as the unions of the inner and outer hulls of $A_1$ and $A_2$.
	\end{proof}

	\section{Proofs of classification application}
	\label{sec:appendix_classif}
	
	In the proofs, we find it sometimes convenient to take explicit note of the map $T : \reals^d \to [0, \infty)^d$ used to transform the marginal distributions of the predictor variable; typically, $T = v$ or $T = \hat{v}$. In line with the theoretical and empirical classification risks $L_t(g)$ and $\widehat{L}^\tau(g)$ in \eqref{classifrisk-transfoDependent} and \eqref{eq:emp-risk-extreme} in the paper, define
	\begin{align*}
		L_t(g, T) 
		&= t \, \PP{g(T(X)) \ne Y, \, \norm{T(X)}_p \ge t}, \\
		\widehat{L}^\tau(g, T) 
		&= \frac{1}{k} \sum_{i=1}^n \un \{ 
		g(T(X_i)) \ne Y_i, \, 
		\theta_p(T(X_i)) \in \sphere_p^\tau, \, 
		\norm{T(X_i)}_p \ge n/k 
		\}.
	\end{align*}
	Then $L_t(g)$ in \eqref{classifrisk-transfoDependent} corresponds to $L_t(g, v)$ here and $\widehat{L}^\tau(g)$ in \eqref{eq:emp-risk-extreme} to $\widehat{L}^\tau(g, \hat{v})$.
	
	\begin{proof}[Proof of Lemma~\ref{lem:limLt-v-angular classif}]
		We prove the first statement only. The proof of the second one follows the same lines and is left to the reader. The third statement is obtained by adding up the left and right hand sides of the first two identities.
		
		Decompose $L_{t}^{>\tau}$ into a type-I and a type-II risk: $L_t^{>\tau}(g,T) = L_{t,+}^{>\tau}(g,T) + L_{t,-}^{>\tau}(g,T)$ with
		\begin{equation*}
			L_{t,\pm}^{>\tau}(g,T) = t \PP{g(T(X))\neq Y, \, Y= \pm 1, \, \theta_p(T(X)) \in \sphere_p^\tau, \, \norm{T(X)}_p \geq t}.
		\end{equation*}
		Consider the cones generated by regions $\sphere_p^\pm(g) = \{ x \in \sphere_p : g(x) = \pm 1\}$, that is, $R_p^{\pm}(g) = \{ tx : x \in  \sphere_p^\pm(g), \, t\ge 1\}$.  
		Equipped with this notation we may write
		\begin{align*}
			L_{t,+}^{>\tau}(g,T) 
			&= t  \PP{ 
				T(X) \in  R_p^-(g), \, 
				Y = +1, \, 
				\theta_p(T(X)) \in \sphere_p^\tau, \,
				\norm{T(X)}_p \ge t 
			}, \\
			L_{t,-}^{>\tau}(g,T) 
			&= t \PP{
				T(X) \in R_p^+(g), \,
				Y = -1, \,
				\theta_p(T(X)) \in \sphere_p^\tau, \,
				\norm{T(X)}_p \ge t 
			}.  
		\end{align*}
		Setting the standardization function $T$ to $v$ yields $T(X) = V$ and thus
		\begin{align*}
			L_{t,+}^{>\tau}(g,\Vf)
			& = 
				t \PP{V \in R_p^-(g), \, \theta_p(V) \in \sphere_p^\tau, \, \norm{V}_p \ge t, \, Y = +1 }\\
			& = \varrho 
				t \PP{\theta_p(V) \in \sphere_p^-(g) \cap \sphere_p^\tau, \, \norm{V}_p \ge t
					\given  Y = +1} \\
		&\to \varrho 
			\Phi_p^+(\sphere_p^-(g) \cap \sphere_p^\tau), \qquad t \to \infty,
	\end{align*}
	where the last convergence occurs because of Assumption~\ref{assum:smoothness} and the fact that $\Phi_p^+$ is dominated by $\Phi_p$, so that $\sphere_p^-(g) \cap \sphere_p^\tau$ is a $\Phi_p^+$-continuity set. 
	Proceeding similarly with $L_{t,-}^{>\tau}(g,\Vf)$, the result is obtained.
\end{proof}

The next result parallels Lemma~\ref{lem:limLt-v-angular classif} by relating the
empirical risk with the empirical angular measure of the positive and
negative classes.

Consider  the  type-I and type-II empirical errors,
\[
\widehat{L}_\pm^\tau(g, T) =
\frac{1}{k} \sum_{i=1}^n \un\{ 
g(T(X_i)) \neq Y_i, \, 
Y_i = \pm 1, \, 
\theta_p(T(X_i)) \in \sphere_p^\tau, \,
\norm{T(X_i)}_p \geq n/k 
\}.
\]
Setting $T = \hat{v}$ as in \eqref{eq:empcdf} yields $T(X_i) = \hV_i$ and thus
\begin{equation*}   
	\widehat{L}_+^\tau(g, \hVf) 
	= \frac{1}{k} \sum_{i=1}^n \un\{
	\theta_p(\hV_i) \in \sphere_p^-(g) \cap \sphere_p^\tau, \,
	Y_i = +1, \,
	\norm{\hV_i}_p \ge n / k 
	\}.
\end{equation*}
Recall from \eqref{eq:empiricalSignedAM} the empirical angular measures $\widehat{\Phi}_p^\sigma$ of the positive and negative instances.
A similar treatment of $\widehat{L}_-^\tau(g, \hVf) $ yields immediately: 

\begin{lemma}
	\label{lem:empirical-risk-angularMeasure}
	In the case where $T$ is the rank transformation $\hVf$,
	\begin{align*}
		\widehat{L}_+^\tau(g, \hVf)
		=  \frac{k_+}{k}
			\widehat\Phi_p^+ (\sphere_p^-(g) \cap \sphere_p^\tau) \,, \qquad 
		\widehat{L}_-^\tau(g, \hVf)
		=  \frac{k_-}{k}
			\widehat\Phi_p^- (\sphere_p^+(g) \cap \sphere_p^\tau).
	\end{align*}
\end{lemma}

\begin{proof}[Proof of Theorem~\ref{theo:deviationsClassif}]
	Recall from the proof of Lemma~\ref{lem:limLt-v-angular classif} the decomposition of $L_\infty^{>\tau}$ into type-I and type-II errors,  $L_\infty^{>\tau} =  L_{\infty,+}^{>\tau} + L_{\infty,-}^{>\tau}  $ with $ L_{\infty,+}^{>\tau}(g) = \varrho \Phi_p^+(\sphere_p^-(g) \cap \sphere_p^\tau)$ and
	$ L_{\infty,-}^{>\tau}(g) = (1-\varrho)\Phi_p^-(\sphere_p^+(g) \cap \sphere_p^\tau)$.  Recall also the identities for their empirical counterparts $\widehat L_{\pm}^\tau$  in Lemma~\ref{lem:empirical-risk-angularMeasure}. For $g\in\mathcal{G}$, the deviations of the empirical risk may be bounded by the sum of the deviations of the two error types, 
	\begin{align}
		\left|\widehat{L}^\tau(g, \hat v) - L_\infty^{>\tau}(g)\right|
		&= \left|\widehat{L}_+^\tau(g, \hat v) - L_{\infty,+}^{> \tau}(g) +
		\widehat{L}_-^\tau(g, \hat v) -   L_{\infty,-}^{> \tau}(g)\right| \nonumber \\
		&\leq \left|\widehat{L}_+^\tau(g, \hat v) - L_{\infty, +}^{> \tau}(g) \right| +
		\left|\widehat{L}_-^\tau(g, \hat v) -  L_{\infty, -}^{>\tau}(g)\right|.
		\label{eq:excessRisk-decompose-biasVar2}
	\end{align}
	
	Let us focus on the first term of the sum. The treatment of the second term is entirely similar and this accounts for the factor two on the right-hand side of the bound in the theorem. From Lemma~\ref{lem:empirical-risk-angularMeasure} and the definition of $L_{\infty,+}^{>\tau}(g)$ recalled above, we have
	\begin{equation*}
		|\widehat{L}_+^\tau(g, \hat v) -  L_{\infty, +}^{>\tau } (g) | 
		=  \bigg|\frac{k_+}{k}
		\widehat\Phi_p^+
		(\sphere_p^-(g) \cap \sphere_p^\tau)
		- \varrho 
		\Phi_p^+(\sphere_p^-(g) \cap \sphere_p^\tau)
		\bigg| , 
	\end{equation*}
	which suggests extending the concentration results concerning the
	empirical measure $\widehat\Phi_p$ to its conditional version
	$\widehat\Phi_p^+$. To do so we shall work on the product
	space $\rset^d\times \{-1,1\}$. We first introduce some notation.
	Let $Q$ be the joint distribution of the pair $(X,Y)$ on  
	$\rset^d\times \{-1,1\}$  and let $Q_n$ denote its empirical version, 
	$Q_n = n^{-1} \sum_{i=1}^n \delta_{(X_i,Y_i)}$. 
	As in Section~\ref{sec:traditionalEmpirical} we can and will assume that each margin $X_j$ is unit-Pareto so that $X = V$. The empirical measure of the rank-transformed data is
	\[
	\widehat Q_n = \frac{1}{n} \sum_{i=1}^n \delta_{(\hat V_i,Y_i)} =
	Q_n \circ (\hVf, \id)^{-1}
	\]
	where $\id$ is the identity function mapping (on $\{-1,1\}$).
	
	Define the limit measure on $E \times \{-1,1\}$: for $\sigma \in \{-, +\}$ and Borel sets $B \subseteq E$ bounded away from the origin with $\mu_\sigma(\partial B) =0$, put
	\[
	\nu(B \times\{ \sigma 1 \}) = \lim_{t\to\infty} t \PP{V \in tB, Y = \sigma 1} = \PP{Y = \sigma 1} \, \mu_\sigma(B),  
	\]
	a limit which exists by the conditional regular variation assumption~\eqref{eq:conditional_RV}.  Notice that $\nu$ is homogeneous of order $-1$ w.r.t. the first component. 
	With this notation and the one borrowed from the proof of Theorem~\ref{thm:conc} (see~\eqref{eq:GammaArho} for the definition of $\widehat{\Gamma}_A$), we have, for $A \subseteq\sphere_p$,
	\begin{align*}
		\tfrac{k_+}{k} \widehat\Phi_p^+(A)
		& = \tfrac{n}{k} \widehat Q_n \left(\tfrac{n}{k} \mathcal{C}_A \times \{+1\} \right)\\
		& = \tfrac{n}{k} Q_n \left(\hVf^{-1}(\tfrac{n}{k} \mathcal{C}_A) \times \{+1\} \right) \\
		& = \tfrac{n}{k} Q_n(\widehat{ \Gamma}_A \times \{+1\} )
	\end{align*}
	and
	\[
	\varrho \Phi_p^+(A) = \nu(\cone_A  \times \{ +1 \}). 
	\]
	It is shown in the proof of Theorem~\ref{thm:conc} that under the assumptions of the statement,
	there exists an event $\mathcal{E}_1$ of probability at least $1 - d\delta/(d+1)$ on which $\frac{n}{k}\Gamma_{A}^-\subseteq \widehat\Gamma_A\subseteq \frac{n}{k}\Gamma_{A}^+ $,  see~\eqref{eq:inclusions}. 
	In addition recall that \[
	\forall A \in \mathcal{A}, \qquad 
	\Gamma_{A}^- \subseteq \cone_A \subseteq \Gamma_{A}^+.
	\]
	Thus on the event $\mathcal{E}_1$, we can decompose the error as in the proof of Theorem~\ref{thm:conc} into 
	\begin{align*}
		\tfrac{k_+}{k}	\hPhi_p^+(A) - \varrho \Phi_p^+(A)
		&= \tfrac{n}{k} Q_n(\widehat{\Gamma}_A \times \{+1\}) - \nu(\cone_A \times \{+1\}) \\
		&\le \tfrac{n}{k} Q_n(\tfrac{n}{k} \Gamma_{A}^+  \times \{+1\}) - \nu(\Gamma_{A}^- \times \{+1\}) \\
		&\le \tfrac{n}{k} \left| Q_n(\tfrac{n}{k} \Gamma_{A}^+ \times \{+1\}) - Q(\tfrac{n}{k} \Gamma_{A}^+ \times \{+1\}) \right| \\
		& \qquad + \left| \tfrac{n}{k} Q(\tfrac{n}{k} \Gamma_{A}^+ \times \{+1\}) - \nu(\Gamma_{A}^+ \times \{+1\}) \right| \\
		&\qquad 	+ \nu\left((\Gamma_{A}^+ \setminus \Gamma_{A}^-) \times \{+1\}\right).
	\end{align*}
	A lower bound for the estimation error can be derived in a similar way, yielding, on $\event_1$,
	\begin{align*}
			\left\lvert \tfrac{k_+}{k} \hPhi_p^+(A) - \varrho \Phi_p^+(A) \right\rvert
		&\le
		\max_{B \in \{\Gamma_{A}^+, \Gamma_{A}^-\}}
		\tfrac{n}{k} 
		\left|
		Q_n(\tfrac{n}{k} B \times \{+1\}) 
		- Q(\tfrac{n}{k} B  \times \{+1\}) 
		\right|
		& \text{(stochastic error II)} \\
		&\quad {} +
		\max_{B \in \{\Gamma_{A}^+, \Gamma_{A}^-\}}
		\tfrac{n}{k}  \left| Q(\tfrac{n}{k} B \times \{+1\}) -
		\nu(B  \times \{+1\}) \right|
		& \text{(bias term II)} \\
		&\quad {} +
		\nu\left((\Gamma_{A}^+ \setminus \Gamma_{A}^-) \times \{+1\}\right)
		& \text{(framing gap II).}
	\end{align*}
	We treat the three terms separately, following closely the proof of Theorem~\ref{thm:conc}. 
	
	\textbf{Bias term II.} \quad
	Taking the supremum over $A \in \cA$ immediately yields the bias term in the statement of Theorem~\ref{theo:deviationsClassif}.
	
	\textbf{Stochastic error II.} \quad
	Since $Q(B\times\{+1\}) \le \PP{X \in B}$, the $Q$-probability of sets in the class
	$\mathcal{F}' = \mathcal{F} \times\{+1\}$ (with $ \mathcal{F}$ defined in \eqref{eq:cF})  is less than $d^{1+1/p}(1+\Delta_2) \tfrac{k}{n}$, see~\eqref{eq:Fprobmax}. The class $\mathcal{F}'$ has the same VC-dimension $V_{\cF}$ as $\mathcal{F}$. Thus on an event $\event_2^+$ of probability $1-\delta/(2(d+1))$ we have, by Theorem~\ref{thm:chaining},
	\begin{multline*}
		\sup_{A \in \cA} \max_{B \in \{\Gamma_{A}^+, \Gamma_{A}^-\}}
		\tfrac{n}{k} 
		\left| 
		Q_n(\tfrac{n}{k} B \times \{+1\}) 
		- Q(\tfrac{n}{k} B  \times \{+1\}) 
		\right| \\
		\le \sqrt{\frac{d^{1+1/p}(1+\Delta_2)}{k}} \left( 56 \sqrt{V_\F} + 2 \sqrt{\log(2(d+1)/\delta)} \right) + \frac{2}{3k} \log(2(d+1)/\delta),
	\end{multline*}
	which is the term labelled `$\operatorname{error}$' in the statement. 
	
	\textbf{Framing gap II.} \quad
	As in the proof of Theorem~\ref{thm:conc}, the framing gap in the product space satisfies
	\begin{equation}
		\label{eq:framingGapDecomposition-classif}
		\begin{split}
			&\nu\bigl((\Gamma_{A}^+ \setminus \Gamma_{A}^-) \times \{+1\}\bigr)  \\
			&\; \le \nu \left(\left\{ x \in [0, \infty)^d : \
			(1+\Delta_2)^{-1} 
			\le \norm{x}_p < (1-\Delta_2)^{-1} 
			\right\} \times \{+1\}\right) \\
			&\; + \nu\left(\left\{ x \in [0, \infty)^d : \
			\norm{x}_p \ge 1, \; \theta_p(x) \in A_+(3\Delta_1 \norm{x}_p) \setminus A_-(3\Delta_1 \norm{x}_p)
			\right\}\times\{+1\} \right).
		\end{split}
	\end{equation}
	The first term on the right-hand side
	of~\eqref{eq:framingGapDecomposition-classif} is equal to
	\[2 \Delta_2 \nu(\{ x \in [0, \infty)^d : \norm{x}_p \ge 1 \}\times\{+1\}) =
	2\Delta_2 \varrho \Phi_p^+(\sphere_p) \le 2 \Delta_2 \Phi_p(\sphere_p), \]
	where the latter inequality comes from the decomposition $\Phi_p = \varrho \Phi_p^+ + (1-\varrho)\Phi_p^-$.
	The second term on the right-hand side in
	\eqref{eq:framingGapDecomposition-classif} can be expressed using the
	polar decomposition of $\mu_+$ (and thus $\nu$), yielding
	\begin{multline*}
		\nu\left(\left\{ x \in [0, \infty)^d : \
		\norm{x}_p \ge 1, \; \theta_p(x) \in A_+(3\Delta_1 \norm{x}_p) \setminus A_-(3\Delta_1 \norm{x}_p)
		\right\}\times\{+1\} \right) \\
		= \int_{1}^\infty  \varrho \Phi_p^+ \bigl( A_+(3\Delta_1 r) \setminus A_-(3\Delta_1 r) \bigr) \frac{\diff r}{r^2} \le \int_{1}^\infty  \Phi_p \bigl( A_+(3\Delta_1 r) \setminus A_-(3\Delta_1 r) \bigr) \frac{\diff r}{r^2}.
	\end{multline*}
	In the proof of Theorem~\ref{thm:conc}, it has been shown that the bound in the latter display is less than
	\[
	3 c \Delta_1 (1 + \log \Phi_p(\sphere_p) - \log(3c \Delta_1)).
	\]
	We thus obtain the same $\operatorname{gap}$ term as in Theorem~\ref{thm:conc}.
	
	So far we have only treated one of the two terms of the error decomposition~\eqref{eq:excessRisk-decompose-biasVar2}. The second one is treated in the same way and the associated upper bound is identical, which yields the factor two in the statement of Theorem~\ref{theo:deviationsClassif}. The decomposition of $|\frac{k_-}{k}\hPhi_p^-(A) - (1-\varrho)\Phi_p^-(A)|$
	into a stochastic error, a bias term and framing gap holds true on the same event $\event_1$.  The bound on the stochastic error is valid on an event $\event_2^-$ of probability at least $1-\delta/(2(d+1))$. Thus the upper bound in the statement of Theorem~\ref{theo:deviationsClassif} holds true on the intersection $\event_1\cap\event_2^+ \cap \event_2^-$, which has probability at least $1 - \delta$. 
\end{proof}

\section{Proofs of simulations experiments}
\label{sec:app:simu}

The following lemma describes convenient properties relative to the simulation setting described in Section~\ref{sec:experiments} in the paper. Recall $\sphere_p^\tau = \{ x \in \sphere_p : \min(x) > \tau \}$ for $\tau \in (0, 1)$ as well as the cones $\cone_A$ in \eqref{eq:angularMeasure} in the paper.

\begin{lemma}\label{lem:explicit-V-Phi}
	Let $R$ be a unit-Pareto distributed random variable and let $\Theta$ be a random vector independent from $R$, with support in $\sphere_1 = \{x \in [0, 1]^d : x_1+\cdots+x_d=1\}$ and satisfying $\EE(\Theta_j) = 1/d$ for $j \in \{1,\ldots,d\}$.   
	Let $X = R\Theta$ and write $V = v(X)$ with $v(x) = (1/(1-F_1(x_1)), \ldots, 1/(1-F_d(x_d)))$ for $x \in \reals^d$, where $F_j(x_j) = \PP{X_j \le x_j}$.
	\begin{enumerate}[label=(\roman*)]
		\item We have $v(x) = dx$ for all $x \in [1,\infty)^d$. 
		Conversely, if $x \in \reals^d$ satisfies $v(x) \in (d, \infty)^d$, then $x \in (1, \infty)^d$ and thus $v(x) = dx$.
		\item 
		The distribution of $V$ is multivariate regularly varying and its angular measure $\Phi_p$ with respect to the $L_p$-norm, for $p \in [1,\infty]$, is given by
		\[
		\Phi_p(A) = d \PP{X \in \cone_A}
		\]
		for Borel sets $A \subseteq \sphere_p$. Moreover, if $A \subseteq \sphere_p^\tau$ for some $\tau \in (0, 1)$ and if $t > d/\tau$, then also
		\[
		\Phi_p(A) = t \PP{t^{-1} V \in \cone_A}.
		\]
	\end{enumerate}
\end{lemma}

\begin{proof}[Proof of Lemma~\ref{lem:explicit-V-Phi}]
	\emph{(i)}
	For any $j \in \{1,\ldots,d\}$ and any $x_j > 0$,
	\begin{align*}
		1 - F_j(x_j) &= \PP{R\Theta_j > x_j}\\
		&= \EE\{\Prob{R > x_j/\Theta_j \mid \Theta_j} \}\\
		&= \EE\{\min(1,  \Theta_j/x_j)\}.
	\end{align*}
	Since the support of $\Theta$ is contained in the unit simplex, we have $\Theta_j/x_j\le 1$ almost surely whenever $x_j \ge 1$. In addition $\EE(\Theta_j) = 1/d$ by assumption. It follows that
	\begin{align*}
		1/(1-F_j(x_j))=  \begin{cases}
			dx_j & \text{if $x_j \geq 1$,}\\
			1/\EE\{\min(1,  \Theta_j/x_j)\} & \text{if $0 < x_j < 1$.} 
		\end{cases}
	\end{align*}
	This proves that $\Vf(x) = dx\in [d,\infty)^d$ for $x \in [1,\infty)^d$.
	The converse statement follows from $1-F_j(1) = 1/d$ and monotonicity.
	
	\emph{(ii)} Let $x \in (0, \infty)^d$ and let $t > 0$ be sufficiently large so that $tx_j \ge d$ for all $j \in \{1,\ldots,d\}$. Then
	\begin{align*}
		1 - \PP{V_1 \le tx_1, \ldots, V_d \le tx_d}
		&= \PP{ \exists j = 1,\ldots,d : d R \Theta_j > t x_j } \\
		&= \mathbb{P} \left[ R > (t/d) \min_{j=1,\ldots,d} (x_j/\Theta_j) \right] \\
		&= (d/t) \EE\left[\max_{j=1,\ldots,d} (\Theta_j / x_j) \right].
	\end{align*}
	In the last step, we conditioned on $\Theta$ and we used the fact that $(t/d) (x_j/\Theta_j) \ge 1$ a.s., by the assumptions on $t$ and $\Theta$.
	Recall $E = [0, \infty)^d \setminus \{(0, \ldots, 0)\}$. It follows that
	\begin{equation}
		\label{eq:tPtVdEmax}
		\forall x \in (0, \infty)^d, \quad 
		\forall t > \frac{d}{\min_{j=1,\ldots,d} x_j}, \qquad
		t \PP{t^{-1} V \in E \setminus [0, x]}
		= d \EE\left[\max_{j=1,\ldots,d} (\Theta_j / x_j) \right].
	\end{equation}
	For fixed $x \in (0, \infty)^d$, the limit as $t \to \infty$ trivially exists.
	This implies that $V$ is multivariate regularly varying as in Section~\ref{sec:RVmu} in the paper; see for instance \citep[Lemma~6.1]{resnick2007heavy}.
	More precisely, $t \PP{t^{-1}V \in \cdot\,} \to \mu(\,\cdot\,)$ as $t \to \infty$ in the space $\mathcal{M}_0$, where $\mu$ is determined by the right-hand side of the previous display via the values of $\mu(E \setminus [0, x])$ for $x \in (0, \infty)^d$.
	
	The angular measure $\Phi_p$ determines the exponent measure $\mu$ uniquely via the relation \eqref{eq:polar} in the paper. In that identity, letting $f$ be the indicator function of the set $E \setminus [0, x]$ with $x \in (0, \infty)^d$, it follows by a standard calculation involving Fubini's theorem that
	\[
	\mu(E \setminus [0, x]) 
	= \int_{\sphere_p} \max_{j=1,\ldots,d}(\theta_j/x_j) \, \diff \Phi_p(\theta).
	\] 
	Define a Borel measure $\Phi_p'$ on $\sphere_p$ by
	\begin{align*}
		\Phi_p'(A) 
		&= d \PP{X \in \cone_A} \\
		&= d \PP{R \norm{\Theta}_p > 1, \, \Theta / \norm{\Theta}_p \in A} \\
		&= d \EE[\norm{\Theta}_p \un\{ \Theta / \norm{\Theta}_p \in A \}]
	\end{align*}
	for Borel sets $A \subseteq \sphere_p$. Recall that $\Phi_p$ is considered with respect to a general $L_p$-norm but $\Theta$ is a random vector in the unit simplex, so $0 < \norm{\Theta}_p \le 1$ almost surely since $p \geq 1$. The last equality follows by conditioning on $\Theta$, the independence of $R$ and $\Theta$, and the assumption that the distribution of $R$ is unit-Pareto. For Borel measurable, nonnegative functions $g$ on $\sphere_p$, we get
	\[
	\int_{\sphere_p} g(\theta) \, \diff \Phi_p'(\theta)
	= d \EE[\norm{\Theta}_p \, g(\Theta / \norm{\Theta}_p)].
	\]
	Applying the latter identity to the function $g$ defined by $g(\theta) = \max_{j=1,\ldots,d}(\theta_j/x_j)$ for some fixed $x \in (0, \infty)^d$ yields
	\[
	\int_{\sphere_p} \max_{j=1,\ldots,d}(\theta_j/x_j) \, \diff \Phi_p'(\theta)
	= d \EE\left[\max_{j=1,\ldots,d} (\Theta_j / x_j)\right] 
	= \mu(E \setminus [0, x]).
	\]
	We conclude that $\Phi_p$ is equal to $\Phi_p'$, as required.
	
	Finally, let $0 < \tau < 1$ and let $A \subseteq \sphere_p^\tau$ be a Borel set. The cone $\cone_A = \{x \in E : \norm{x}_p \ge 1, x / \norm{x}_p \in A\}$ is a subset of $(\tau, \infty)^d$. The equality \eqref{eq:tPtVdEmax} together with the inclusion--exclusion formula and the fact that rectangles form a measure-determining class imply that the measures $t \PP{t^{-1} V \in \cdot\,}$ and $\mu$ correspond on $(\tau, \infty)^d$. It follows that
	\[
	t \PP{t^{-1} V \in \cone_A} = \mu(\cone_A) = \Phi_p(A). \qedhere
	\]
\end{proof}

\end{appendix}

\end{document}